\documentclass[nointlimits,12pt,oneside]{amsart}
\usepackage{amssymb,cases,enumitem, verbatim}
\usepackage{xcolor}
\usepackage{hyperref}

\usepackage{accents}

\hypersetup{
	    colorlinks=true,
	    linkcolor=blue,
	    citecolor=blue,
	    filecolor=green,
	    urlcolor=cyan,
	    bookmarks=true
	}
	
\makeatletter
\renewcommand*{\eqref}[1]{%
  \hyperref[{#1}]{\textup{\tagform@{\ref*{#1}}}}%
}
\makeatother

\setenumerate{label=(\roman*)}

\usepackage[%
	a4paper,
	total={16cm,23cm},
	left=3cm, top=3cm,
	marginparsep=2pt]
{geometry}

\theoremstyle{plain}
\newtheorem{theorem}{Theorem}[section]
\newtheorem{lemma}[theorem]{Lemma}
\newtheorem{corollary}[theorem]{Corollary}
\newtheorem{proposition}[theorem]{Proposition}
\theoremstyle{definition}
\newtheorem{remark}[theorem]{Remark}

\newtheorem{definition}[theorem]{Definition}

\numberwithin{equation}{section}

\def\esssup{\operatornamewithlimits{ess\,\sup}}

\DeclareMathOperator{\sgn}{sgn}

\def\L1loc{L^1_{\text{loc}}}

\begin{document}
\title{Lorentz--Karamata spaces}
\author{Dalimil Pe{\v s}a}

\address{Dalimil Pe{\v s}a, Department of Mathematical Analysis, Faculty of Mathematics and
Physics, Charles University, Sokolovsk\'a~83,
186~75 Praha~8, Czech Republic}
\email{pesa@karlin.mff.cuni.cz}
\urladdr{0000-0001-6638-0913}


\thanks{This research was supported by the grants no.~P201-18-00580S, P201/21-01976S, and P202/23-04720S of the Czech Science Foundation; the Primus research programme PRIMUS/21/SCI/002 of Charles University; the grant SFG205 of Faculty of Mathematics and Physics, Charles University, Prague; Charles University Research program No. UNCE/SCI/023; and by the grant SVV-2023-260711.}

\begin{abstract}
In this paper, we consider Lorentz--Karamata spaces with slowly varying functions and provide a comprehensive study of their properties.

We consider Lorentz--Karamata functionals over an arbitrary sigma-finite measure space equipped with a non-atomic measure and the corresponding Lorentz--Karamata spaces. We characterise non-triviality of said spaces, then study when they are equivalent to a Banach function space and obtain a complete characterisation. We compute the fundamental function of said spaces and describe the corresponding endpoint spaces. We further provide a complete characterisation of when the Lorentz--Karamata spaces defined using non-increasing rearrangement are equivalent to those defined using maximal function. We provide a complete description of the associate spaces of Lorentz--Karamata spaces. We also treat other topics like embeddings, absolute continuity of the (quasi)norm, and Boyd indices.
\end{abstract}

\date{\today}

\maketitle

\makeatletter
   \providecommand\@dotsep{2}
\makeatother

\section{Introduction}

In this paper we focus on one particular scale of function spaces, called Lorentz--Karamata spaces. These function spaces were introduced in 2000 by Edmunds, Kerman and Pick in~\cite{EdmundsKerman00}, and their name reflects the fact that their construction encapsules both the Lorentz-type structure of fine tuning of function spaces and the concept of the so-called slowly-varying functions that had been studied by Karamata. The original motivation for the introduction of these spaces was connected with the investigation of the very important problem of nailing down optimal partner function spaces in Sobolev embeddings on regular domains in the Euclidean space. During the last two decades the Lorentz--Karamata spaces have extended their field of applications for example to Gauss-Sobolev embeddings  in  \cite{CianchiPick09} and the related interpolation theory in \cite{Baena-MiretGogatishvili22}, boundedness of operators on probability spaces in \cite{CianchiPick15}, and traces of Sobolev functions in \cite{CianchiPick16}. They were also studied in connection with Bessel potential type spaces in \cite{GogatishviliOpic04} and \cite{Neves02}, employed in describing embeddings of Besov-type spaces in \cite{CaetanoGogatishvili11} and \cite{GurkaOpic07}, and their interpolation properties were the focus of \cite{Bathory18}. Their upmost importance seems to consist of the fact that they provide a class of very good and useful examples for various tasks in functional analysis and its applications, which is on one hand quite versatile (note that Lorentz--Karamata spaces contain Lebesgue spaces,  Lorentz spaces, Zygmund classes, Lorentz--Zygmund spaces, a good deal of Orlicz spaces, many Lorentz and Marcinkiewicz endpoint spaces, the space of Br\'ezis and Wainger, and many more), on the other hand they are relatively easily manageable, which makes them extremely handy. 

During the last 20 years, Lorentz--Karamata spaces have been several times put under a detailed scrutiny. For instance, an~alternative characterization by ``norms within norms'' of them is given in~\cite{EdmundsOpic08}. Further characterisation by the means of alternative norms was obtained in \cite{FernandezSignes14}. They are briefly mentioned also in~\cite{FucikKufner13}. A detailed study focused on some of their basic functional properties can be found in \cite{EdmundsEvans04} and \cite{Neves02}. But none of these works provided a completely satisfactory result as they usually contained various restrictions. One of our principal goals is to fill in the gap and provide a comprehensive study of these spaces.

We would like to stress that we work with Lorentz--Karamata spaces in their full generality, using the modern definition of slowly varying functions that was introduced in \cite{GogatishviliOpic04} or \cite{GogatishviliOpic05} (it is unclear which paper is in fact older). Our approach is thus more general than that taken by some other authors, see for example \cite{EdmundsEvans04}, \cite{Ho20}, and \cite{Neves02}, whose definition of slowly varying functions requires that the behaviour near zero is the same as that near infinity. This not only makes the Lorentz--Karamata spaces less general, it also causes the definition to be more complicated. The reason for this approach is probably the way the Lorentz--Karamata spaces were originally introduced in \cite{EdmundsKerman00}, where they were considered only in the case when the underlying measure space was of finite measure, so the question of behaviour near infinity did not arise. The modern approach, as taken in, for example, \cite{Bathory18}, \cite{CaetanoGogatishvili11}, \cite{EdmundsOpic08}, \cite{GogatishviliNeves10}, \cite{GogatishviliOpic04}, \cite{GurkaOpic07}, and \cite{NevesOpic20}, resolves this issue and we thus believe it to be significantly better. However, most of the papers employing the modern definition do not concern themselves with the properties of Lorentz--Karamata spaces, or only do so in a very limited fashion (as in \cite{Bathory18}, where non-triviality and embeddings are characterised), so most of the properties of said spaces are shown in the older papers which use the restrictive definitions. We have thus opted to provide full proofs for all our results to save the reader the tedious work of checking whether the arguments that were used in the restricted case still work in our more general setting.

It is worth noting that even though the Lorentz--Karamata spaces are a fairly general class of function spaces, they are further generalised by the classical Lorentz spaces, which have been studied quite extensively since their introduction by Lorentz in \cite{Lorentz51}. In some cases we use in our proofs this more abstract theory of classical Lorentz spaces, namely the results contained in \cite{CarroPick00}, \cite{CarroSoria96}, \cite{CarroRaposo07}, \cite{CarroSoria93}, \cite{GogatishviliSoudsky14}, \cite{Sawyer90}, and \cite{Stepanov93}, while in other cases we opted for a more elementary approach. As a further reading on the topic we would recommend \cite{ArinoMuckenhoupt90}, \cite{CarroGogatishvili08}, \cite{CarroSoria97}, \cite{GogatishviliKrepela17}, \cite{GogatishviliPick03}, \cite{GogatishviliPick06}, \cite{Lorentz61}, and \cite{Sinnamon02}.

The paper is structured as follows. We first provide the necessary theoretical background in Section~\ref{SP}. We then provide the comprehensive treatment of Lorentz--Karamata spaces in Section~\ref{SMR}.

To be more specific, we first define the Lorentz--Karamata spaces in Section~\ref{SSLK} and prove some of their basic properties. Section~\ref{SecFF} is then devoted to the description of the fundamental functions of Lorentz--Karamata spaces and the corresponding endpoints. In Section~\ref{SBI} we compute the Boyd indices of Lorentz--Karamata spaces. In Section~\ref{SSE}, embeddings between Lorentz--Karamata spaces are completely characterised. The equivalence of Lorentz--Karamata spaces defined using the non-increasing rearrangement and those defined using the maximal function  is studied in Section~\ref{SSQN} where we obtain a complete characterisation. Section~\ref{SecACN} is devoted to characterising which functions in Lorentz--Karamata spaces have absolutely continuous quasinorm. In Section~\ref{SSLKAS} we describe the associate spaces of Lorentz--Karamata spaces for all choices of parameters. Finally, in Section~\ref{SSBFS} we provide a complete characterisation of the cases when a Lorentz--Karamata space is a Banach function space.

We would like to emphasize that we do not assume that the measure of the underlying measure space is finite. On the contrary, we generally assume it to be infinite as those cases are usually more interesting. There are, however, few exceptions to this rule and there we consider also the case of finite measure in order to obtain the most interesting versions of the results.
	
\section{Preliminaries}\label{SP}

The aim of this section is to establish the basic framework for later work. We strive to keep the definitions and notation as standard as possible.

From now on, we will denote by $(R, \mu)$, and sometimes $(S, \nu)$, some arbitrary $\sigma$-finite measure space. When $E \subseteq R$, we will denote its characteristic function by $\chi_E$. The set of all extended complex-valued  $\mu$-measurable functions defined on $R$ will be denoted by $M(R, \mu)$, its subsets of all non-negative functions and of all functions finite $\mu$-almost everywhere on $R$ will be denoted by $M_+(R, \mu)$ and $M_0(R, \mu)$ respectively. As usual, we identify functions that are equal $\mu$-almost everywhere. We will usually abbreviate $\mu$-almost everywhere to $\mu$-a.e.\ and simply write $M$, $M_+$ and $M_0$, instead of $M(R, \mu)$, $M_+(R, \mu)$ and $M_0(R, \mu)$ respectively, whenever there is no risk of confusion. In some special cases when $R = \mathbb{R}^n$ we will denote the $n$-dimensional Lebesgue measure by $\lambda^n$ (or just $\lambda$ when $n=1$).

When $X$ is a set and $f, g: X \to \mathbb{C}$ are two maps satisfying that there is some positive and finite constant $C$, depending only on $f$ and $g$, such that $\lvert f(x) \rvert \leq C \lvert g(x) \rvert$ for all $x \in X$, we will denote this by $f \lesssim g$. We will also write $f \approx g$, or sometimes say that $f$ and $g$ are equivalent, whenever both $f \lesssim g$ and $g \lesssim f$ are true at the same time. We choose this general definition because we will use the symbols ``$\lesssim$'' and ``$\approx$'' with both functions and functionals.

When $X$ and $Y$ are quasinormed linear spaces, we will use the notation $X \hookrightarrow Y$ to mean that $X \subseteq Y$ and that the identity operator is bounded from $X$ to $Y$, i.e.~that it holds for every $x \in X$ that
\begin{equation*}
	\lVert x \rVert_Y \lesssim \lVert x \rVert_X,
\end{equation*}
where $\lVert \cdot \rVert_X$ and $\lVert \cdot \rVert_Y$ are the respective quasinorms in $X$ and $Y$. When $X \hookrightarrow Y$ and $Y \hookrightarrow X$ hold at the same time, we will say that $X = Y$ up to equivalence of quasinorms. In case when the functionals $\lVert \cdot \rVert_X$ and $\lVert \cdot \rVert_Y$ are norms, we will say that the equality holds up to equivalence of norms. When the type of functional is not clear (e.g.~when one is a norm and the other is not, when it depends on a particular choice of parameters, etc.) we might say that the equality holds up to equivalence of defining functionals, in order to avoid making any implicit statement about the triangle inequality of the functionals in question.

Finally, when $q \in (0, \infty]$, we will denote by $L^q$ the classical Lebesgue space (of functions in $M(R, \mu)$) defined by
\begin{equation*}
L^q = \left \{ f \in M(R, \mu); \; \int_R \lvert f \rvert^q \: d\mu < \infty \right \},
\end{equation*}
equipped with the customary (quasi-)norm
\begin{equation*}
\lVert f \rVert_q = \left ( \int_R \lvert f \rvert^q \: d\mu \right )^{\frac{1}{q} },
\end{equation*}
with the usual modifications when $q=\infty$.

\subsection{Non-increasing rearrangement} \label{SSNR}

In this section, we define the non-increasing rearrangement of a function and some related terms. We proceed in accordance with \cite[Chapter~2]{BennettSharpley88}.

The non-increasing rearrangement of a function $f \in M$, traditionally denoted $f^*$, is defined as the generalised inverse of the distribution function, that is for any $t \in [0, \infty)$
\begin{equation*}
	f^*(t) = \inf \{ s \in [0, \infty); \; \mu_f(s) \leq t \},
\end{equation*}
where the distribution function $\mu_f$ of a function $f \in M$ is defined for $s \in [0, \infty)$ by
\begin{equation*}
	\mu_f(s) = \mu(\{ t \in R; \; \lvert f(t) \rvert > s \}).
\end{equation*}

Some basic properties of the distribution function and the non-increasing rearrangement, with proofs, can be found in \cite[Chapter~2, Proposition~1.3]{BennettSharpley88} and \cite[Chapter~2, Proposition~1.7]{BennettSharpley88}.

A very important classical result is the Hardy-Littlewood inequality which asserts that it holds for all $f, g \in M$ that
\begin{equation}
	\int_R \lvert f g \rvert \: d\mu \leq \int_0^{\infty} f^* g^* \: d\lambda. \label{THLI}
\end{equation}
For details, see for example \cite[Chapter~2, Theorem~2.2]{BennettSharpley88}. As an immediate consequence, we get that, for all $f, g \in M$, 
\begin{equation}
	\sup_{\substack{\tilde{g} \in M \\ \tilde{g}^* = g^*}} \int_R \lvert f \tilde{g} \rvert \: d\mu \leq \int_0^{\infty} f^* g^* \: d\lambda. \label{DR}
\end{equation}
This leads to the definition of resonant measure spaces, as those spaces where we have equality in \eqref{DR}. It has been proven that in order for a measure space to be resonant, it suffices for its measure to be non-atomic. For details, see \cite[Chapter~2, Theorem~2.7]{BennettSharpley88}.

We will also need in the paper the so called maximal function of $f^*$ (where $f \in M$), which is traditionally denoted $f^{**}$ and is defined for $t \in [0, \infty)$ by 
\begin{equation*}
	f^{**}(t) = \frac{1}{t} \int_0^{t} f^*(s) \: ds.
\end{equation*} 
Some properties of the maximal function can be found in \cite[Chapter~2, Proposition~3.2]{BennettSharpley88} and  \cite[Chapter~2, Theorem~3.4]{BennettSharpley88}.

\subsection{Function norms and quasinorms} \label{SSFN}

The following two definitions are adapted from \cite[Chapter~1, Definition~1.1]{BennettSharpley88} and \cite[Chapter~2, Definition~4.1]{BennettSharpley88}, respectively.

\begin{definition}
	Let $\lVert \cdot \rVert : M \to [0, \infty]$ be some non-negative functional such that it holds for all $f \in M$ that $\lVert f \rVert = \lVert \; \lvert f \rvert \; \rVert$. We then say that $\lVert \cdot \rVert$ is a Banach function norm if its restriction on $M_+$ satisfies the following conditions:
	
	\begin{enumerate}[label=\textup{(P\arabic*)}, series=P]
		\item \label{P1} $\lVert \cdot \rVert$ is a norm, i.e.
		\begin{enumerate}
			\item it is positively homogeneous, i.e.\ $\forall a \in \mathbb{C} \forall f \in M_+ \: : \: \lVert a \cdot f \rVert = \lvert a \rvert \lVert f \rVert$,
			\item it satisfies $\lVert f \rVert = 0 \Leftrightarrow f = 0$  $\mu$-a.e.,
			\item it is subadditive, i.e.\ $\forall f,g \in M_+ \: : \: \lVert f+g \rVert \leq \lVert f \rVert + \lVert g \rVert$.
		\end{enumerate}
		\item \label{P2} $\lVert \cdot \rVert$ has the lattice property, i.e.\ if some $f, g \in M_+$ satisfy $f \leq g$ $\mu$-a.e., then also $\lVert f \rVert \leq \lVert g \rVert$.
		\item \label{P3} $\lVert \cdot \rVert$ has the Fatou property, i.e.\ if  some $f_n, f \in M_+$ satisfy $f_n \uparrow f$ $\mu$-a.e., then also $\lVert f_n \rVert \uparrow \lVert f \rVert $.
		\item \label{P4} $\lVert \chi_E \rVert < \infty$ for all $E \subseteq R$ satisfying $\mu(E) < \infty$.
		\item \label{P5} For every $E \subseteq R$ satisfying $\mu(E) < \infty$ there exists some finite constant $C_E$, dependent only on $E$, such that for all $f \in M_+$ the inequality $ \int_E f \: d\mu \leq C_E \lVert f \rVert $ holds.
	\end{enumerate} 
\end{definition}

We will be interested in particular in the class of rearrangement invariant Banach function norms.

\begin{definition}
	We say that a Banach function norm $\lVert \cdot \rVert$ is rearrangement invariant, abbreviated r.i., if  it satisfies the following additional condition:
	\begin{enumerate}[resume*=P]
		\item \label{P6} If two functions $f,g \in M_+$ satisfy $f^* = g^*$ then $\lVert f \rVert = \lVert g \rVert$.
	\end{enumerate}
\end{definition}

For the properties of Banach function norms and r.i.~Banach function norms see \cite[Chapter~1]{BennettSharpley88} and \cite[Chapter~2]{BennettSharpley88}, respectively.

It will be also useful to define a somewhat weaker version of r.i.\ Banach function norms, namely the rearrangement-invariant quasi-Banach function norms. 

\begin{definition} \label{Q}
	Let $\lVert \cdot \rVert : M \to [0, \infty]$ be some non-negative functional such that it holds for all $f \in M$ that $\lVert f \rVert = \lVert \; \lvert f \rvert \; \rVert$. We then say that $\lVert \cdot \rVert$ is a quasi-Banach function norm, if it satisfies the axioms \ref{P2}, \ref{P3}, and \ref{P4} from the definition of r.i.\ Banach function norm and also a weaker version of \ref{P1}, namely
	\begin{enumerate}[label=\textup{(Q\arabic*)}]
		\item \label{Q1} $\lVert \cdot \rVert$ is a quasinorm, i.e.
		\begin{enumerate}
			\item it is positively homogeneous, i.e.\ $\forall a \in \mathbb{R} \forall f \in M_+ \: : \: \lVert a \cdot f \rVert = \lvert a \rvert \lVert f \rVert$,
			\item it satisfies  $\lVert f \rVert = 0 \Leftrightarrow f = 0$ $\mu$-a.e.,
			\item it is subadditive up to a constant, i.e.\ there is some finite constant $C$ such that $\forall f,g \in M_+ \: : \: \lVert f+g \rVert \leq C(\lVert f \rVert + \lVert g \rVert)$.
		\end{enumerate}
	\end{enumerate}
\end{definition}

\begin{definition}
	We say that a quasi-Banach function norm $\lVert \cdot \rVert$ is rearrangement-invariant, abbreviated r.i., if it satisfies \ref{P6}.
\end{definition}

Our definition of quasi-Banach function norms is standard, the same as for example in \cite{CaetanoGogatishvili16}, \cite{NekvindaPesa20}, and \cite{Lopez08} (note that the requirement of completeness in \cite{Lopez08} is redundant as follows from \cite[Lemma~3.6]{CaetanoGogatishvili16}). The paper \cite{NekvindaPesa20} treats the properties of (r.i.) quasi-Banach function norms in some detail.

\begin{definition}
	Let $\lVert \cdot \rVert$ be a Banach function norm on $M$. Then the set 
	\begin{equation*}
	X = \{ f \in M; \; \lVert f \rVert < \infty \}
	\end{equation*} 
	equipped with the norm $\lVert \cdot \rVert$ will be called a Banach function space. Further, if $\lVert \cdot \rVert$ is rearrangement invariant, we shall say that $X$ is a rearrangement invariant Banach function space.
	
	If $\lVert \cdot \rVert$ is an (r.i.)\ quasi-Banach function norm, we define an (r.i.)\ quasi-Banach function space in exactly the same manner.
\end{definition}

To provide an example, let us note that the classical Lebesgue functional  $\lVert \cdot \rVert_q$ is an r.i.\ Banach function norm on $M$ for $q \in [1, \infty]$ and an r.i.\ quasi-Banach function norm on $M$ for $q \in (0,\infty]$, as follows from the formula 
\begin{equation} \label{FormulaLebesgue}
	\lVert f \rVert_q = \left(  \int_0^{\infty} (f^*)^q \: d\lambda \right)^{\frac{1}{q}}
\end{equation}
(and its obvious modification for the case $q=\infty$) that holds for every $f \in M$ (see \cite[Chapter~2, Proposition~1.8]{BennettSharpley88}).

Another examples are some cases of the classical Lorentz spaces $\Lambda^q(v)$ and $\Gamma^q(v)$, where $q \in (0, \infty)$ and $v : (0, \infty) \to [0, \infty) $ is a non-negative $\lambda$-measurable weight, which are defined as the sets of functions $f \in M$ for which the following respective functionals are finite:
\begin{align*}
	\lVert f \rVert_{\Lambda^q(v)} &= \left( \int_0^{\infty} (f^*)^q v \: d\lambda \right) ^{\frac{1}{q}}, \\
	\lVert f \rVert_{\Gamma^q(v)} &= \left( \int_0^{\infty} (f^{**})^q v \: d\lambda \right) ^{\frac{1}{q}}.
\end{align*}
We would like to stress that there are many choices of $q$ and $v$ for which those spaces even fail to be linear, so some additional assumption are certainly necessary in order for those spaces to serve as examples of r.i.~(quasi-)Banach function spaces. We do not wish to go too deep into the rich and complex theory of Classical Lorentz spaces, so let us just say that $\lVert f \rVert_{\Lambda^q(v)}$ is an r.i.~Banach function norm provided that $q \geq 1$ and that $v$ is non-increasing, non-zero, and locally integrable, while $\lVert f \rVert_{\Gamma^q(v)}$ is an r.i.~Banach function norm provided that $q>1$ and $v$ is strictly positive, locally integrable, and bounded at some neighbourhood of infinity. We emphasise that those are only some simple sufficient conditions, as the actual characterisation is more complex. To provide some references, we turn to the works already mentioned in the introduction, i.e.~\cite{ArinoMuckenhoupt90}, \cite{CarroGogatishvili08}, \cite{CarroPick00}, \cite{CarroSoria96}, \cite{CarroRaposo07}, \cite{CarroSoria93}, \cite{CarroSoria97}, \cite{GogatishviliKrepela17}, \cite{GogatishviliPick03}, \cite{GogatishviliPick06}, \cite{GogatishviliSoudsky14}, \cite{Lorentz61}, \cite{Sawyer90}, \cite{Sinnamon02}, and \cite{Stepanov93}.

An important concept in the theory of quasi-Banach function spaces is the absolute continuity of the quasinorm.

\begin{definition}
	Let $\lVert \cdot \rVert_X$ be a quasi-Banach function norm and let $X$ be the corresponding quasi-Banach function space. We say that a function $f \in X$ has absolutely continuous quasinorm if it holds that $\lVert f \chi_{E_k} \rVert_X \to 0$ (as $k \to \infty$) whenever $E_k$ is a sequence of $\mu$-measurable subsets of $R$ such that $\chi_{E_k} \to 0$ $\mu$-a.e.~(as $k \to \infty$).
	
	The set of all $f \in X$ that have absolutely continuous quasinorms will be denoted by $X_a$. If every $f \in X$ has absolutely continuous quasinorm (i.e.~if $X_a = X$) we further say that the space $X$ itself has absolutely continuous quasinorm.
\end{definition}

This concept is important, because it is deeply connected to separability and also reflexivity of Banach function spaces, for details see \cite[Chapter~1, Sections~3, 4 and 5]{BennettSharpley88}.

To provide some examples, the Lebesgue spaces $L^q$ have absolutely continuous quasinorms for $q < \infty$, while $L^{\infty}_a = \{0\}$.

We will need the following characterisation that describes the functions having absolutely continuous quasinorms. The result is well known for the case of Banach function spaces (see e.g.~\cite[Chapter~1, Proposition~3.5]{BennettSharpley88}), but almost exactly the same proof can be used to extend it to the context of quasi-Banach function spaces.

\begin{proposition} \label{PropACN}
	Let $\lVert \cdot \rVert_X$ be a quasi-Banach function norm and let $X$ be the corresponding quasi-Banach function space. Then $f \in X$ has absolutely continuous quasinorm if and only if every sequence $f_n$ of $\mu$-measurable functions satisfying $\lvert f \rvert \geq \lvert f_n \rvert \downarrow 0$ $\mu$-a.e.~(as $n \to \infty$) satisfies $\lVert f_n \rVert_X \downarrow 0$ (as $n \to \infty$).
\end{proposition}

\subsection{Associate spaces} \label{SSAS}
In this section we define the associate functionals and consequently the associate spaces. We intentionally make the definitions very broad and general. For details on associate spaces of Banach function spaces, see \cite[Chapter~1, Section~2, 3 and 4]{BennettSharpley88}. We will approach the matter in a more general way.

When talking about associate spaces, it is useful to use the notation $q'$ for the number defined by
\begin{equation*}
q' = 
\begin{cases}
\frac{q}{q-1} & \text{for } q \in (1, \infty), \\
1  & \text{for } q = \infty, \\
\infty  & \text{for } q = 1,
\end{cases}
\end{equation*} 
which we will do from now on.

\begin{definition} \label{DAS}
	Let $\lVert \cdot \rVert: M \to [0, \infty]$ be some non-negative functional and put
	\begin{equation*}
	X = \{ f \in M; \; \lVert f \rVert < \infty \}.
	\end{equation*} 
	Then the set 
	\begin{equation*}
	X' = \left \{ f \in M; \; \sup_{\substack{g \in M \\ \lVert g \rVert \leq 1}} \int_R \lvert f g \rvert \: d\mu < \infty \right \}
	\end{equation*}
	will be called the associate space of $X$ and the functional $\lVert \cdot \rVert_{X'}$ defined for $f \in M$ by 
	\begin{equation*}
	\lVert f \rVert_{X'} = \sup_{\substack{g \in M \\ \lVert g \rVert \leq 1}} \int_R \lvert f g \rvert \: d\mu
	\end{equation*}
	will be called the associate functional of $\lVert \cdot \rVert$.
\end{definition}

As the notation suggests, these terms are interesting mainly when the functional $\lVert \cdot \rVert$ is at least a quasinorm, but we wanted to emphasize that the definition itself lays no requirements on it.

The point of the associate space is to provide a general form of H{\"o}lder inequality, namely that it holds for all $f, g \in M$ that
\begin{equation}
	\int_R \lvert f  g \rvert \: d\mu \leq \lVert g \rVert \lVert f \rVert_{X'}, \label{THAS}
\end{equation}
where $\lVert \cdot \rVert: M \to [0, \infty]$ is assumed to be some positively homogeneous functional and in the case when the right-hand side is of the form $0 \cdot \infty$ it is to be interpreted as $\infty$.

To provide an example, if we take $\lVert \cdot \rVert_q$, $q \in [1, \infty]$, as our $\lVert \cdot \rVert$, we obtain that $\lVert \cdot \rVert_{X'} = \lVert \cdot \rVert_{q'}$. On the other hand, if we take $\lVert \cdot \rVert_q$, $q \in (0,1)$, we get $\lVert f \rVert_{X'} = \infty$ for all $f$ other than zero (and of course $\lVert 0 \rVert_{X'} = 0$) and thus $X' = \{0\}$.

It is quite obvious from Definition~\ref{DAS} that if two positively homogeneous functionals $\lVert \cdot \rVert_X$ and $\lVert \cdot \rVert_Y$ satisfy $\lVert \cdot \rVert_X \lesssim \lVert \cdot \rVert_Y$ then also $\lVert \cdot \rVert_{Y'} \lesssim \lVert \cdot \rVert_{X'}$. Equivalently, we may say that if $X \hookrightarrow Y$ then $Y' \hookrightarrow X'$, where the continuity  of the embedding is understood relatively to the defining functionals, i.e.\ it means that the identity operator is bounded.

It has been proven, see for example \cite[Chapter~1, Theorem~2.2]{BennettSharpley88} and \cite[Chapter~1, Theorem~2.7]{BennettSharpley88}, that the associate functional of a Banach function norm $\lVert \cdot \rVert$ is itself a Banach function norm, and that its associate functional is $\lVert \cdot \rVert$. This result has been improved recently by Gogatishvili and Soudsk{\'y} in \cite{GogatishviliSoudsky14}. Since we will use this result later in the paper, we present it below. The symbol $\lVert \cdot \rVert_{X''}$ denotes the second associate functional, i.e.~the associate functional of $\lVert \cdot \rVert_{X'}$.

\begin{theorem} \label{TFA}
	Let $\lVert \cdot \rVert : M \to [0, \infty]$ be a functional that satisfies \ref{P4} and \ref{P5} and which also satisfies for all $f \in M$ that $\lVert f \rVert$ = $\lVert \: \lvert f \rvert \: \rVert$. Then the functional $\lVert \cdot \rVert_{X'}$ is a Banach function norm. In addition, $\lVert \cdot \rVert$ is equivalent to a Banach function norm if and only if $\lVert \cdot \rVert \approx \lVert \cdot \rVert_{X''}$
	
\end{theorem}

\subsection{Boyd indices}

A useful concept from the theory of interpolation is that of Boyd indices, which were introduced by Boyd in \cite{Boyd69}. We do not wish to go deep into the related theory at this point; we only want to give the necessary background for our computations in Section~\ref{SBI}. To the readers interested in this topic we would recommend \cite[Chapter~3, Section~5]{BennettSharpley88}, \cite[Sections~8 and 9]{JohnsonMaurey79}, \cite[Chapter~2]{KreinPetunin82}, and \cite{Maligranda85}.

Let $\lVert \cdot \rVert_X$ be an r.i.~quasi-Banach function norm over $M(R, \mu)$ such that there is some r.i.~quasi-Banach function norm $\lVert \cdot \rVert_{\bar{X}}$ over $M([0, \infty), \lambda)$ such that it holds for every $f \in M(R, \mu)$ that $\lVert f \rVert_X = \lVert f^* \rVert_{\bar{X}}$. Denote by $X$ and $\bar{X}$ the respective corresponding r.i.~quasi-Banach function spaces. The quasinorm $\lVert \cdot \rVert_{\bar{X}}$ is the called the representation functional of $\lVert \cdot \rVert_X$ and, similarly, $\bar{X}$ is called the representation space of $X$.

It is well known that a representation functional exists for every r.i.~Banach function space and that in this case it is also an r.i.~Banach function norm. As for the r.i.~quasi-Banach function spaces in general, we were not able to find any such result in the literature. However, in many concrete cases the representation functional is available directly from the definition. This for example holds for Lebesgue spaces, as follows from \eqref{FormulaLebesgue}, and classical Lorentz spaces.

Consider now the family of dilation operators $D_t : M([0, \infty), \lambda) \to  M([0, \infty), \lambda)$ that are defined for $t \in (0, \infty)$ by
\begin{align*}
	D_t f(s) &= f(ts) &\textup{for } s \in [0, \infty).
\end{align*}
It has been proven in \cite[Theorem~3.22]{NekvindaPesa20} that such an operator is, for every choice of $t \in (0, \infty)$, bounded on any r.i.~quasi-Banach function space. The following definition thus makes sense:
\begin{definition}
	Let $\lVert \cdot \rVert_X$ be an r.i.~quasi-Banach function norm over $M(R, \mu)$ for which the representation functional exists and denote by $h_X(t)$ the operator norm of $D_{t^{-1}} : \bar{X} \to \bar{X}$. We then define the lower and upper Boyd indices, denoted by $\underline{\alpha}_{X}$ and  $\overline{\alpha}_{X}$, respectively, as
	\begin{align*}
		\underline{\alpha}_{X} &= \sup_{t \in (0,1)} \frac{\log(h_X(t))}{\log(t)}, &\overline{\alpha}_{X} = \inf_{t \in (1, \infty)} \frac{\log(h_X(t))}{\log(t)}.
	\end{align*}
\end{definition}

The following result is very useful in practical applications, since suprema and infima of functions are, in general, difficult to compute:

\begin{proposition} \label{PBI}
	Let $\lVert \cdot \rVert_X$ be an r.i.~quasi-Banach function norm over $M(R, \mu)$ for which the representation functional exists and denote by $h_X(t)$ the operator norm of $D_{t^{-1}} : \bar{X} \to \bar{X}$. Then the lower and upper Boyd indices of $X$ satisfy
	\begin{align*}
	\underline{\alpha}_{X} &= \lim_{t \to 0_+} \frac{\log(h_X(t))}{\log(t)}, &\overline{\alpha}_{X} = \lim_{t \to \infty} \frac{\log(h_X(t))}{\log(t)}.
	\end{align*}
\end{proposition}

This result is classical in the case when $\lVert \cdot \rVert_X$ is an r.i.~Banach function norm, see \cite[Chapter~3, Proposition~5.13]{BennettSharpley88} and \cite[Lemma~3]{Boyd69}. On the other hand, we were unable to find in the literature any proofs of the general form needed for our purposes. We therefore present a proof here, which is, similarly to the classical case, based on the elementary properties of subadditive functions. For the sake of completeness we present here a proof of the necessary properties that is an extension of the ideas contained in \cite[Chapter~3, Lemma~5.8]{BennettSharpley88}. We note that a similar result appears in \cite[Theorem~7.6.2]{HillePhillips57}, with the difference being that the authors work with weaker assumptions and thus arrive at a weaker conclusion.

\begin{lemma} \label{LSAF}
	Let $\omega : \mathbb{R} \to \mathbb{R}$ be a non-decreasing subadditive function (i.e.~it holds for every $s,t \in \mathbb{R}$ that $\omega(s+t) \leq \omega(s)+\omega(t)$) such that $\omega(0) = 0$. Then it holds for every $s \in \mathbb{R}$ that
	\begin{equation} \label{LSAF1}
		-\omega(-s) \leq \omega(s).
	\end{equation} 
	Furthermore, we have that
	\begin{equation} \label{LSAF2}
		0 \leq \sup_{s \in (-\infty, 0) } \frac{\omega(s)}{s} = \lim_{s \to -\infty} \frac{\omega(s)}{s} \leq \lim_{s \to \infty} \frac{\omega(s)}{s} = \inf_{s \in (0, \infty)} \frac{\omega(s)}{s} < \infty.
	\end{equation}
\end{lemma}

\begin{proof}
	The estimate \eqref{LSAF1} follows directly from our assumptions, since they guarantee that
	\begin{align*}
		0 &= \omega(0) \leq \omega(-s) + \omega(s) &\textup{for } s \in \mathbb{R}.
	\end{align*}
	Furthermore, it is clear that
	\begin{equation*}
		 \inf_{s \in (0, \infty)} \frac{\omega(s)}{s} \leq \omega(1) < \infty
	\end{equation*}
	and that
	\begin{equation*}
		\sup_{s \in (-\infty, 0) } \frac{\omega(s)}{s} \geq -\omega(-1) \geq 0.
	\end{equation*}
	
	Let us now fix some $\varepsilon > 0$, some $t \in (0, \infty)$, and some $n_0 \in \mathbb{N}$ such that
	\begin{equation*}
		\left (1 + \frac{1}{n_0} \right) \frac{\omega(t)}{t} < \inf_{s \in (0, \infty)} \frac{\omega(s)}{s} + \varepsilon.
	\end{equation*}
	Then for any $s \geq n_0t$ there is some $n \geq n_0$ such that $s \in [nt, (n+1)t)$ and thus we may estimate
	\begin{equation*}
		\frac{\omega(s)}{s} \leq \frac{(n+1)\omega(t)}{nt} \leq \left (1 + \frac{1}{n_0} \right) \frac{\omega(t)}{t} < \inf_{s \in (0, \infty)} \frac{\omega(s)}{s} + \varepsilon.
	\end{equation*}
	Since $\varepsilon$ was arbitrary and the opposite estimate is trivial, we have shown that
	\begin{equation*}
		\lim_{s \to \infty} \frac{\omega(s)}{s} = \inf_{s \in (0, \infty)} \frac{\omega(s)}{s}.
	\end{equation*}
	
	To compute the second limit, let us fix some $\varepsilon > 0$ and denote
	\begin{align*}
		\xi =
			\begin{cases}
				\sup_{s \in (-\infty, 0) } \frac{\omega(s)}{s} - \varepsilon &\text{if } \sup_{s \in (-\infty, 0) } \frac{\omega(s)}{s} < \infty, \\
				\frac{1}{\varepsilon} &\text{if } \sup_{s \in (-\infty, 0) } \frac{\omega(s)}{s} = \infty.
			\end{cases} 
	\end{align*} 
	We may now find $t \in (-\infty, 0)$ and $n_0 \in \mathbb{N}$ such that
	\begin{equation*}
		\left( 1 - \frac{1}{n_0 + 1} \right) \frac{\omega(t)}{t} > \xi.
	\end{equation*}
	Then for any $s \leq n_0t$ there is some $n \geq n_0$ such that $s \in ((n+1)t, nt]$. Whence,
	\begin{equation*}
		\omega(s) \leq n\omega(t)
	\end{equation*}
	and thus (since $s$ and $t$ are both negative while $\omega(s)$ and $\omega(t)$ are both non-positive)
	\begin{equation*}
		\frac{\omega(s)}{s} \geq \frac{n\omega(t)}{s} \geq \frac{n\omega(t)}{(n+1)t} \geq \left( 1 - \frac{1}{n_0 + 1} \right) \frac{\omega(t)}{t} > \xi.
	\end{equation*}
	Since $\varepsilon$ was arbitrary, we have thus shown that
	\begin{equation*}
		\lim_{s \to -\infty} \frac{\omega(s)}{s} = \sup_{s \in (-\infty, 0) } \frac{\omega(s)}{s},
	\end{equation*}
	regardless of the finiteness or infiniteness of the right-hand side. This also shows that the limit on the left-hand side always exist. The desired estimate
	\begin{equation*}
		\lim_{s \to -\infty} \frac{\omega(s)}{s} \leq \lim_{s \to \infty} \frac{\omega(s)}{s}
	\end{equation*}
	therefore follows from \eqref{LSAF1}, which implies for every $s \in (0, \infty)$ that
	\begin{equation*}
		\frac{\omega(-s)}{-s} = \frac{-\omega(-s)}{s} \leq \frac{\omega(s)}{s}.
	\end{equation*}
	This shows the second inequality in \eqref{LSAF2} and concludes the proof.
\end{proof}

\begin{proof}[Proof of Proposition~\ref{PBI}]
	Consider the function $\omega : \mathbb{R} \to \mathbb{R}$ given by
	\begin{align*}
		\omega(s) &= \log(h_X(\exp(s))) &\textup{for } s \in \mathbb{R},
	\end{align*}
	where $h_X$ is the function defined above. The properties of $h_X$---namely that it is non-decreasing, that it satisfies for every $t_1, t_2 \in \mathbb{R}$ the estimate
	\begin{equation*}
		h_X(t_1t_2) \leq h_X(t_1) h_X(t_2),
	\end{equation*}
	and that $h_X(1) = 1$ (as $D_1$ is the identity operator)---ensure that our $\omega$ satisfies the requirements of Lemma~\ref{LSAF}. Therefore it follows by a simple change of variables that
	\begin{align*}
		\sup_{t \in (0,1)} \frac{\log(h_X(t))}{\log(t)} &= \sup_{s \in (-\infty, 0) } \frac{\omega(s)}{s} = \lim_{s \to -\infty} \frac{\omega(s)}{s} = \lim_{t \to 0_+} \frac{\log(h_X(t))}{\log(t)}, \\
		\inf_{t \in (1,\infty)} \frac{\log(h_X(t))}{\log(t)} &= \inf_{s \in (0, \infty) } \frac{\omega(s)}{s} = \lim_{s \to \infty} \frac{\omega(s)}{s} = \lim_{t \to \infty} \frac{\log(h_X(t))}{\log(t)}.
	\end{align*}
\end{proof}

\subsection{Fundamental function} \label{SSFF}

In this section, we define the fundamental function of an r.i.\ quasi-Banach function space and state some of its properties, and then do the same with the endpoint spaces. We proceed in accordance with \cite[Section~7.9]{FucikKufner13} and \cite[Section~7.10]{FucikKufner13}. We note that this topic has been also covered in \cite[Chapter~2, Section~5]{BennettSharpley88}.

In this section, we will always assume that the space $(R, \mu)$ is completely non-atomic.

Given an r.i.~quasi-Banach function norm $\lVert \cdot \rVert$ and its corresponding r.i.~quasi-Banach function space $X$, the fundamental function $\varphi_X$ of $X$ is defined for all $t$ in the range of $\mu$ by
\begin{equation*}
	\varphi_X(t) = \lVert \chi_E \rVert,
\end{equation*} 
where $E$ is some subset of $R$ of measure $\mu(E) = t$. Note that the set $E$ in the definition always exists by the assumption that $t$ is in range of $\mu$, and that the definition does not depend on the choice of $E$ since the norm satisfies \ref{P6}.

One important result is the relation of the respective fundamental functions of $X$ and of its associate space $X'$. If $X$ is a r.i.\ quasi-Banach function space and $X'$ its associate space, then the corresponding fundamental functions satisfy the estimate
\begin{equation}
\varphi_{X'}(t) \geq \frac{t}{\varphi_X(t)} \label{TFFAS1}
\end{equation} 
for all finite, non-zero $t$ in the range of $\mu$. Furthermore, if $X$ is an r.i.~Banach function space, then we have equality in \eqref{TFFAS1} for all relevant $t$. The weaker result for r.i.\ quasi-Banach function spaces follows directly from the H{\" o}lder inequality \eqref{THAS}. Proof of the stronger result for r.i.~Banach function spaces can be found for example in \cite[Theorem~7.9.6]{FucikKufner13}.

It has been shown, that fundamental function of any r.i.~Banach function space over a resonant measure space is quasiconcave, in the sense that is is non-decreasing, it attains zero only at the point zero and that the function $\frac{t}{\varphi(t)}$ is non-decreasing on $(0, \mu(R))$. For details, see for example \cite[Remark~7.9.7]{FucikKufner13}.

We are now equipped to define the Marcinkiewicz endpoint space.

\begin{definition}\label{DMES}
	Let $\varphi$ be a quasiconcave function on $[0, \mu(R))$. Then the set $M_{\varphi}$, defined as
	\begin{equation*}
	M_{\varphi} = \left \{ f \in M(R, \mu); \; \sup_{t \in (0, \mu(R))} \varphi(t) f^{**}(t) < \infty \right \},
	\end{equation*}
	is called the Marcinkiewicz endpoint space.
\end{definition} 

It follows, see for example \cite[Proposition~7.10.2]{FucikKufner13}, that the Marcinkiewicz endpoint space equipped with the naturally chosen functional
\begin{align*}
\lVert f \rVert  &= \sup_{t \in (0, \infty)} \varphi(t) f^{**}(t) & \text{for } f \in M(R, \mu)
\end{align*} 
is an r.i.\ Banach function space and its fundamental function coincides with $\varphi$ on the range of $\mu$. In fact, it is the largest such space, in the sense that if $X$ is an r.i.\ Banach function space the fundamental function of which coincides with $\varphi$ on the range of $\mu$ then $X \hookrightarrow M_{\varphi}$. Proof can be found for example in \cite[Proposition~7.10.6]{FucikKufner13}. 

To find the smallest r.i.\ Banach funtion space with given fundamental function, one must first observe that any quasiconcave function is equivalent to a concave function. This result can be found for example in \cite[Proposition~7.10.10]{FucikKufner13}.  It can be then shown, as in, for example, \cite[Theorem~7.10.12]{FucikKufner13}, that any r.i.\ Banach function space can be equivalently renormed in such way that its fundamental function is then concave (and obviously non-decreasing). For concave $\varphi$, we may then define the Lorentz endpoint space as follows.

\begin{definition} \label{DLES}
	Let $\varphi$ be a non-decreasing concave function on $[0, \mu(R))$. Then the set $\Lambda_{\varphi}$, defined as
	\begin{equation*}
	\Lambda_{\varphi} = \left \{ f \in M(R, \mu); \; \int_0^{\mu(R)}  f^* \: d\varphi < \infty \right \},
	\end{equation*}
	is called the Lorentz  endpoint space.
\end{definition}

We note that the Lebesgue--Stieltjes integral in question is well-defined since $\varphi$ is non-decreasing. We further get, see for example \cite[Proposition~7.10.16]{FucikKufner13}, that the Lorentz endpoint space equipped with the naturally chosen functional
\begin{align*}
\lVert f \rVert  &= \int_0^{\mu(R)}  f^* \: d\varphi & \text{for } f \in M(R, \mu),
\end{align*} 
is an r.i.\ Banach function space fundamental function of which coincides with $\varphi$ on the range of $\mu$. It is the smallest space with those properties, in the sense that if $X$ is an r.i.\ Banach function space such that its fundamental function coincides with $\varphi$ on the range of $\mu$, then $\Lambda_{\varphi} \hookrightarrow X$. For proof see \cite[Proposition~7.10.15]{FucikKufner13}.

The Lorentz and Marcinkiewicz endpoint spaces are mutually associate, in the sense that if $\varphi$ is a concave function, then $\Lambda_{\varphi}' = M_{\bar{\varphi}}$ and $M_{\varphi}' = \Lambda_{\bar{\varphi}}$, where $\bar{\varphi}$ is defined on $(0, \mu(R))$ by
\begin{equation*}
	\bar{\varphi}(t) = \frac{t}{\varphi(t)}.
\end{equation*}
and $\bar{\varphi}(0) = 0$. This follows directly from the discussion above, since if $X$ is an r.i.~Banach function space whose fundamental function is $\bar{\varphi}$, then the fundamental function of $X'$ is $\varphi$ and thus $\Lambda_{\varphi} \hookrightarrow X'$. Hence, $X \hookrightarrow \Lambda_{\varphi}'$, and it follows that $\Lambda_{\varphi}'$ is the largest r.i.~Banach function space whose fundamental function is $\bar{\varphi}$, i.e.~it must coincide with $M_{\bar{\varphi}}$. The argument proving that $M_{\varphi}' = \Lambda_{\bar{\varphi}}$ is similar. Furthermore, since a Banach function space is uniquely determined by its associate space, we observe that if $X$ is an r.i.~Banach function space that is distinct from both $\Lambda_{\varphi_X}$ and $M_{\varphi_X}$ then its associate space $X'$ must be distinct from both $\Lambda_{\varphi_{X'}}$ and $M_{\varphi_{X'}}$ (where $\varphi_{X'} = \overline{\varphi_X}$, as mentioned above).

In conclusion, we note that the endpoint spaces do not change, up to equivalence of the defining functionals, if $\varphi$ is replaced with an equivalent function. To be more precise, if $\varphi$ and $\psi$ are two quasi-concave functions such that $\varphi \approx \psi$ then $M_{\varphi} = M_{\psi}$, up to equivalence of norms, while if $\varphi$ and $\psi$ are both also concave then further $\Lambda_{\varphi} = \Lambda_{\psi}$, up to equivalence of norms. This is rather obvious for the Marcinkiewicz endpoints and by the mutual associateness of Lorentz and Marcinkiewicz endpoints, as discussed in the previous paragraph, it extends to the Lorentz endpoints as well.

\subsection{Slowly varying functions} \label{SSSV}

The key concept in the theory of Lorentz--Karamata spaces is that of a slowly varying function.

\begin{definition} \label{DSV}
	Let $b:(0,\infty) \rightarrow (0, \infty)$ be a measurable function. Then $b$ is said to be slowly varying, if for every $\varepsilon > 0$ there exists a non-decreasing function $b_{\varepsilon}$ and non-increasing function $b_{-\varepsilon}$ such that $t^{\varepsilon}b(t) \approx b_{\varepsilon}(t)$ on $(0, \infty)$ and $t^{-\varepsilon}b(t) \approx b_{-\varepsilon}(t)$ on $(0, \infty)$.
\end{definition}

We will, for the sake of brevity, usually abbreviate  slowly varying as s.v.

The class of s.v.\ functions includes, for example, constant positive functions and, to provide something at least slightly less trivial, the functions $t \mapsto 1+\lvert \log(t) \rvert$ and $t \mapsto 1+\log(1+\lvert \log(t) \rvert)$ occurring in the definition of Lorentz--Zygmund spaces and generalized Lorentz--Zygmund spaces, the former being introduced by Bennett and Rudnick in \cite{BennettRudnick80} while the latter were introduced by Edmunds, Gurka, and Opic in \cite{EdmundsGurka95} and later treated in great detail by Opic and Pick in \cite{OpicPick99}. As a further example of an s.v.~function, this time non-logarithmic, we present a function $b$ defined on $(0, \infty)$ by

\begin{equation*}
b(t) = \begin{cases}
e^{\sqrt{\log t}} & \text{for } t \in [1, \infty), \\
e^{\sqrt{\log t^{-1}}} & \text{for } t \in (0, 1).
\end{cases}
\end{equation*} 
Restriction of this function to the interval $(0, 1)$ was used in \cite{CianchiPick09} to characterise self-optimal spaces for Gauss-Sobolev embeddings.

Let us now recall some important properties of s.v.~functions that are either known or that follow from known results by straightforward generalisations.

\begin{lemma} \label{LSV}
	Let $b, b_2$ be s.v.\ functions. Then the following statements hold:
	\begin{enumerate}[label=\textup{(SV\arabic*)}, start=0]
		\item \label{SVEq} If $c:(0,\infty) \rightarrow (0, \infty)$ satisfies $c \approx b$ then $c$ is also a s.v.~function. 
		\item \label{SV1} The function $b^r$ is slowly varying for every $r \in \mathbb{R}$. 
		\item \label{SV2} The functions $b+b_2$, $bb_2$, $\frac{b}{b_2}$ and $t \mapsto b(\frac{1}{t})$ are	are s.v.
		\item \label{SVC} The function $b$ is bounded by positive constants on compact subsets of $(0, \infty)$.
		\item \label{CSV} Let $c > 0$, then $b(ct) \approx b(t)$ on $(0, \infty)$. More specifically, for every $\varepsilon >0$ there is a constant $C_{\varepsilon} \geq 1$ such that it holds for every $c > 0$ and every $t \in (0, \infty)$ that
		\begin{equation} \label{CSV1}
		C_{\varepsilon}^{-1} \min\{c^{-\varepsilon}, \, c^{\varepsilon} \} b(t) \leq b(ct) \leq C_{\varepsilon} \max\{c^{-\varepsilon}, \, c^{\varepsilon} \} b(t).
		\end{equation}
		\item \label{SV3} It holds for every $ \alpha \neq 0$ that
		\begin{align*}
		\lim_{t \to 0^+} t^{\alpha}b(t) = \lim_{t \to 0^+}t^{\alpha}, \\
		\lim_{t \to \infty} t^{\alpha}b(t) = \lim_{t \to \infty}t^{\alpha}.
		\end{align*}
		\item \label{SV4} Let $\alpha \neq -1$. Then
		\begin{equation*}
		\int_0^1 t^{\alpha}b(t) \: dt < \infty \iff \int_0^1 t^{\alpha} \: dt < \infty
		\end{equation*}
		and
		\begin{equation*}
		\int_1^{\infty} t^{\alpha}b(t) \: dt < \infty \iff \int_1^{\infty} t^{\alpha} \: dt < \infty.
		\end{equation*}
		Consequently,
		\begin{equation*}
		\int_0^{\infty} t^{\alpha} b(t) \: dt = \infty.
		\end{equation*}
		\item \label{LEFF} \label{LEFFH} Let $\alpha \in (0, \infty)$. Then it holds for every $t \in (0, \infty)$ that
		\begin{align}
		\int_0^{t} s^{\alpha-1} b(s) \: ds &\approx t^{\alpha} b(t), & \esssup_{s \in (0,t)} s^{\alpha} b(s) &\approx t^{\alpha} b(t), \label{LEFF1}\\
		\int_t^{\infty} s^{-\alpha-1} b(s) \: ds &\approx t^{-\alpha} b(t), & \esssup_{s \in (t,\infty)} s^{-\alpha} b(s) &\approx t^{-\alpha} b(t).  \label{LEFFH1}
		\end{align}
		\item \label{LTb} \label{LHb} Let us define, for $t \in (0, \infty)$,
		\begin{align*}
		\tilde{b}(t) &= \int_0^{t} s^{-1} b(s) \: ds, & \tilde{b}_{\infty}(t) &= \esssup_{s \in (0, t)} b(s), \\
		\hat{b}(t) &= \int_t^{\infty} s^{-1} b(s) \: ds, & \hat{b}_{\infty}(t) &= \esssup_{s \in (t, \infty)} b(s).
		\end{align*}
		Then $b \lesssim \tilde{b}$, $b \lesssim \tilde{b}_{\infty}$, $b \lesssim \hat{b}$, and $b \lesssim \hat{b}_{\infty}$ on $(0, \infty)$ and the functions $\tilde{b}$, $\tilde{b}_{\infty}$, $\hat{b}$, and  $\hat{b}_{\infty}$ are s.v.~if and only if they are finite $\lambda$-a.e.~on $(0, \infty)$. Furthermore, it holds that
		\begin{align} \label{LTb1}
		\limsup_{t \to 0^+} \frac{\tilde{b}(t)}{b(t)} = \limsup_{t \to \infty} \frac{\tilde{b}(t)}{b(t)} &= \infty, & 
		\limsup_{t \to 0^+} \frac{\hat{b}(t)}{b(t)} = \limsup_{t \to \infty} \frac{\hat{b}(t)}{b(t)} &= \infty.
		\end{align}
	\end{enumerate}
\end{lemma}

\begin{proof}
	For \ref{SV1}, \ref{SV2},  \ref{CSV}, and \ref{LEFF} see \cite[Proposition~2.2]{GogatishviliOpic05}. We note that it is not quite clear from the original formulation of \ref{CSV} in \cite[Proposition~2.2, (iii)]{GogatishviliOpic05} that the constant $C_{\varepsilon}$ depends only on $\varepsilon$, but if follows from the presented proof that it is indeed so. For \ref{SVC} see \cite[Proposition~2.1]{Bathory18} and for the first part of \ref{LTb} see \cite[Lemma~2.2]{Bathory18} and \cite[Lemma~2.1]{GurkaOpic07}. The remaining part of \ref{LTb}, i.e.~\eqref{LTb1}, has been partially proved in \cite[Lemma~2.2]{CaetanoGogatishvili11} and \cite[(2.1) and (2.2)]{NevesOpic20} and the method used in the former paper can be easily adapted to cover all the four cases mentioned in \eqref{LTb1}.
	
	The properties \ref{SVEq}, \ref{SV3}, and \ref{SV4} follow directly from the definition. We leave the details to the reader.
\end{proof}

\begin{remark} \label{RLC}
	The restrictions $ \alpha \neq 0$ and $ \alpha \neq -1$ in \ref{SV3} and \ref{SV4}, respectively, are necessary, as the property that $b$ is s.v.~is simply too weak to determine the behaviour of  $t^{\alpha}b(t)$ in these limiting cases. The counterexamples are easy to construct.
\end{remark}

\begin{remark}
	There exist three different definitions of slowly varying functions that can be found in the literature. The three defining conditions (for $b \in M((0,\infty), \lambda)$) are:
	\begin{enumerate}
		\item \label{DSVi}
		For every $\varepsilon > 0$ the function $t^{\varepsilon}b(t)$ is non-decreasing on some neighbourhoods of zero and infinity while the function $t^{-\varepsilon}b(t)$ is non-increasing on some neighbourhoods of zero and infinity.
		\item \label{DSVii}
		For every $\varepsilon > 0$ there exists a non-decreasing function $b_{\varepsilon}$ and non-increasing function $b_{-\varepsilon}$ such that
		\begin{equation*}
		\lim_{t \to 0} \frac{t^{\varepsilon}b(t)}{b_{\varepsilon}(t)} = \lim_{t \to \infty} \frac{t^{\varepsilon}b(t)}{b_{\varepsilon}(t)} = \lim_{t \to 0} \frac{t^{-\varepsilon}b(t)}{b_{-\varepsilon}(t)} =\lim_{t \to \infty} \frac{t^{-\varepsilon}b(t)}{b_{-\varepsilon}(t)} = 1.
		\end{equation*}
		\item \label{DSViii}
		$b$ satisfies Definition~\ref{DSV}.
	\end{enumerate}
	The condition \ref{DSVi} was used in the original definition of Lorentz--Karamata spaces in \cite{EdmundsKerman00} and the functions satisfying it are said to belong to the Zygmund class $\mathcal{Z}$ (see \cite{Zygmund35_3} for further details). The condition \ref{DSVii} is the closest to the original definition of slowly varying functions as given by Karamata in \cite{Karamata30} and \cite{Karamata33} --- while it is not the original definition, it is equivalent to it. Functions satisfying this condition are treated thoroughly in \cite{BinghamGoldie87}. The condition \ref{DSViii} is the one currently used in most papers concerning Lorentz--Karamata spaces and it first appeared in either \cite{GogatishviliOpic04} or \cite{GogatishviliOpic05} (the idea to use only the equivalence with monotone functions appeared originally in the paper \cite{Neves02}, but in this case the function was also required to behave the same way near zero as it does near infinity, so the definition was significantly less general).
	
	Provided that the function $b$ is assumed to be bounded on compact sets (which is very reasonable for our purposes), then clearly the validity of \ref{DSVi} implies that of \ref{DSVii} and the validity of \ref{DSVii} implies that of \ref{DSViii}. On the other hand, for any function $b$ satisfying \ref{DSVii} there is a function $b_0$ satisfying \ref{DSVi} such that
	\begin{equation*}
	\lim_{t \to 0} \frac{b(t)}{b_{0}(t)} = \lim_{t \to \infty} \frac{b(t)}{b_{0}(t)} = 1.
	\end{equation*}
	
	Some of the properties shown in Lemma~\ref{LSV} can be strengthened if we assume that $b$ satisfies \ref{DSVii}. Specifically, in \ref{CSV} and \ref{LEFF} we can replace the relation ``$\approx$" in all its occurrences with ``the ratio of the left-hand side and the right-hand side converges to $1$ at both $0$ and $\infty$" while	in \eqref{LTb1} we get that
	\begin{align*}
	\lim_{t \to 0^+} \frac{\tilde{b}(t)}{b(t)} = \lim_{t \to \infty} \frac{\tilde{b}(t)}{b(t)} &= \infty, & 
	\lim_{t \to 0^+} \frac{\hat{b}(t)}{b(t)} = \lim_{t \to \infty} \frac{\hat{b}(t)}{b(t)} &= \infty.
	\end{align*}
	For details, see \cite[Chapter~1]{BinghamGoldie87}.
\end{remark}

\begin{remark}
	The property \ref{LEFF} is very important. Not only it provides an essential tool for computations, but it also gives us a finer grasp on the definition of s.v. functions. To be more specific: while the definition requires only that for every $\varepsilon > 0$ there is some pair of monotone functions that satisfies the appropriate requirements, \ref{LEFF} provides us with a specific pair, given by explicit formulas, and asserts that there is a pair of monotone functions satisfying the above mentioned requirements if and only if this specific pair satisfies them.
\end{remark}

On several occasions it will be convenient to assume that the s.v.~function in question is smooth. This is of course not true for every s.v.~function (consider \ref{SVEq}), but it can be ensured by passing to an equivalent function, which in our applications causes no loss of generality.

\begin{theorem} \label{TSmSV}
	Let $b$ be a s.v.~function. Then there is a function $c : (0, \infty) \to (0, \infty)$ satisfying $c \approx b$ that has continuous classical derivatives of all orders (i.e.~$c \in \mathcal{C}^{\infty}$).
\end{theorem}

This theorem has been proved in full strength in \cite{PesaSV}; weaker version of it has also been developed earlier in \cite{OpicGroverTBD}.

We conclude this section by noting that if the s.v.~function $b$ is smooth (in the sense $b \in \mathcal{C}^{\infty}$) and bounded near zero, then the function $\tilde{b}_{\infty}$ is locally Lipschitz continuous, and that the same holds for $\hat{b}_{\infty}$ when $b$ is smooth and bounded near infinity.

\section{Lorentz--Karamata spaces} \label{SMR}

From now on we suppose that our measure space $(R, \mu)$ has a non-atomic measure. As noted in Section~\ref{SSNR}, this means that it is always resonant. We will also, unless stated otherwise, assume that $\mu(R) = \infty$.
\subsection{Lorentz--Karamata functionals} \label{SSLK}

\begin{definition} \label{DLK}
	Let $p \in (0, \infty], q \in (0, \infty]$ and let $b$ be an s.v.\ function. We then define the Lorentz--Karamata functionals $ \lVert \cdot \rVert_{p, q, b}$ and $ \lVert \cdot \rVert_{(p, q, b)}$,  for $f \in M$, as follows:
	\begin{align*}
	\lVert f \rVert_{p, q, b} = \lVert t^{\frac{1}{p}-\frac{1}{q}}b(t)f^*(t) \rVert_q, \\
	\lVert f \rVert_{(p, q, b)} = \lVert t^{\frac{1}{p}-\frac{1}{q}}b(t)f^{**}(t) \rVert_q,
	\end{align*}
	where $\lVert \cdot \rVert_q$ is the classical Lebesgue functional on $(0, \infty)$.
	
	We further define the corresponding Lorentz--Karamata spaces $L^{p, q, b}$ and $L^{(p, q, b)}$ as
	\begin{align*}
	L^{p, q, b} &= \{ f \in M; \; \lVert f \rVert_{p, q, b} < \infty \}, \\
	L^{(p, q, b)} &= \{ f \in M; \; \lVert f \rVert_{(p, q, b)} < \infty \}.
	\end{align*}
\end{definition}

Note that the Lorentz functionals $\lVert \cdot \rVert_{p, q}$ and $ \lVert \cdot \rVert_{(p, q)}$, defined for example in \cite[Chapter~4, Definiton~4.1]{BennettSharpley88}, are special cases of Lorentz--Karamata functionals with $b(t) = 1$. If further $p=q$ then we have that $\lVert \cdot \rVert_{p, q, b}$ coincides with the classical Lebesgue functional $\lVert \cdot \rVert_q$, as follows from \cite[Chapter~2, Proposition~1.8]{BennettSharpley88}. Another examples of Lorentz--Karamata functionals are the generalised Lorentz-Zygmung functionals studied in \cite{OpicPick99}.

The following proposition establishes the basic properties of the Lorentz--Karamata functional $\lVert \cdot \rVert_{p, q, b}$.

\begin{proposition} \label{LQ}
	Let $p, q, b$ be as above. Then the Lorentz--Karamata functional $\lVert \cdot \rVert_{p, q, b}$ is an r.i.~quasi-Banach function norm, and consequently $(L^{p, q, b},\lVert \cdot \rVert_{p, q, b})$ is an r.i.\ quasi-Banach function space, if and only if one of the following conditions hold:
	\begin{enumerate}
		\item $p < \infty$,
		\item $p = \infty$ and $\lVert t^{- \frac{1}{q}} b(t) \chi_{(0, 1)}(t) \rVert_q < \infty$.
	\end{enumerate}
	Furthermore, if those conditions are not satisfied, then $L^{p, q, b}=\{ 0 \}$.
\end{proposition}

\begin{proof}
	Throughout this proof, the letters $f, g$ will denote some arbitrary functions belonging to $M$. We begin by proving the sufficiency.
	
	All the properties, except the the subadditivity up to a constant and \ref{P4} are immediate consequences of the definition, properties of the non-increasing rearrangement, and properties of the Lebesgue functionals.
	
	For the subadditivity up to a constant, if we remember that the non-increasing rearrangement satisfies $(f+g)^*(t_1+t_2) \leq f^*(t_1) + g^*(t_2)$, we may use \ref{CSV} to get the following estimate:
	\begin{equation*} 
	\begin{split}
		\lVert f+g \rVert_{p, q, b} &= \lVert t^{\frac{1}{p}-\frac{1}{q}}b(t)(f+g)^*(t) \rVert_q \\
		&\leq \left \lVert t^{\frac{1}{p}-\frac{1}{q}}b(t) \left ( f^* \left (\frac{t}{2} \right )+g^* \left (\frac{t}{2} \right )  \right ) \right  \rVert_q \\
		&\lesssim \left  \lVert \left ( \frac{t}{2} \right )^{\frac{1}{p}-\frac{1}{q}}b \left (\frac{t}{2} \right ) \left (  f^* \left (\frac{t}{2} \right )+g^* \left (\frac{t}{2} \right ) \right ) \right  \rVert_q.
	\end{split}
	\end{equation*}
	Now, it remains only to use the simple change of variables $s = \frac{t}{2}$ and use the subadditivity (up to a constant) of $\lVert \cdot \rVert_q$ to get the desired result. We have thus verified that $ \lVert \cdot \rVert_{p, q, b}$ satisfies \ref{Q1}.
	
	It remains to verify \ref{P4}. Since it holds for $E \subseteq R$, $\mu(E) < \infty$ that $\chi_E^* = \chi_{(0, \mu (E))}$, it follows that \ref{P4} is equivalent to the question whether $\lVert t^{\frac{1}{p} - \frac{1}{q}} b(t) \chi_{(0, 1)}(t) \rVert_q < \infty$. If $p \neq \infty$, this is satisfied either by \ref{SV4}  or by \ref{SV3}, while in the remaining case it is assumed.
	
	The preceding argument also shows that $\lVert \cdot \rVert_{p, q, b}$ does not have the property \ref{P4} when $p=\infty$ and $\lVert t^{- \frac{1}{q}} b(t) \chi_{(0, 1)}(t) \rVert_q = \infty$, which establishes the necessity. It also follows immediately that, in this case, $L^{p, q, b}=\{ 0 \}$ which concludes the proof.
\end{proof}

It is worth mentioning that for some special choices of parameters $p, q, b$, $ \lVert \cdot \rVert_{p, q, b}$ may qualify to be an r.i.\ Banach function norm on $M$. Let us name at least the cases $p = q \in [1, \infty]$ while $b(t) = 1$ for all $t \in (0, \infty)$ when $ \lVert \cdot \rVert_{p, q, b}$ coincides with the classical Lebesgue functional $\lVert \cdot \rVert_q$. Furthermore, even if it is not itself an r.i.~Banach function norm, it may still be equivalent to it; we characterise for which choices of parameters this happens in Theorem~\ref{TLKBFS}. Finally, the range of parameters for which $ \lVert \cdot \rVert_{p, q, b}$ satisfies the condition \ref{P5} is characterised in Corollary~\ref{CP5b}.

We now present the following useful observation that we may, when working with $L^{\infty, \infty, b}$, always assume $b$ to be non-decreasing. This result first appeared in \cite{Bathory18}, where one can also find its proof. Note that the assumption that $b$ is continuous is present only to avoid unnecessary technicalities and comes at no loss of generality due to Theorem~\ref{TSmSV}. Furthermore, the assumption that $b$ is bounded near zero is natural and, as follows from Proposition~\ref{LQ}, also causes no loss of generality.  

\begin{proposition} \label{PbND}
	Let $b$ be a continuous s.v.~function and assume that it is bounded near zero. Then $L^{\infty, \infty, b} = L^{\infty, \infty, \tilde{b}_{\infty}}$, up to equivalence of the defining functionals, where $\tilde{b}_{\infty}$ is the function defined in \ref{LTb}, that is
	\begin{align*}
	\tilde{b}_{\infty}(t) &= \esssup_{s \in (0, t)} b(s) &\textup{for } t \in (0, \infty).
	\end{align*}
\end{proposition}

We would like to note that Proposition~\ref{PbND} is a special case of a more general principle, for details see \cite[Lemma~5.1]{GogatishviliSoudsky14}.

We now turn our attention to $\lVert \cdot \rVert_{(p, q, b)}$. Because different properties of $\lVert \cdot \rVert_{(p, q, b)}$ have different requirements, we will treat them separately. We start by examining the non-triviality.

\begin{proposition} \label{PP4}
	Let $p, q, b$ be as in Definition~\ref{DLK}. Then $\lVert \cdot \rVert_{(p, q, b)}$ satisfies \ref{P4} if and only if $p$ and $b$ satisfy one of the following conditions:
	\begin{enumerate}
		\item \label{Casei} $p \in (1, \infty)$,
		\item \label{Caseii}$p = 1$ and $\lVert t^{- \frac{1}{q}} b(t) \chi_{(1, \infty)}(t) \rVert_q < \infty$,
		\item \label{Caseiii}$p= \infty$ and $\lVert t^{- \frac{1}{q}} b(t) \chi_{(0, 1)}(t) \rVert_q < \infty$.
	\end{enumerate}
	Furthermore,  if none of those conditions is satisfied, then $L^{(p, q, b)}=\{ 0 \}$.
\end{proposition}

\begin{proof}
	The situation is more complicated than when we were examining \ref{P4} in Proposition~\ref{LQ}, because, if we fix some $E \subseteq R$ with measure $\mu(E) < \infty$, then $(\chi_E)^{**}$ is not equal to $\chi_{(0, \mu(E))}$ as before but rather to the function $\varrho$ defined as follows
	\begin{equation*}
	\varrho(t) = 
	\begin{cases}
	1 & \text{for } t \in (0, \mu(E)), \\
	\frac{\mu(E)}{t} & \text{for } t \in [ \mu(E), \infty).
	\end{cases}
	\end{equation*}
	Thus, in order to verify \ref{P4}, it is necessary to deal not only with integrability/boundedness  on deleted neighbourhoods of zero, but also on deleted neighbourhoods of infinity. To be precise, if $q \in (0, \infty)$, then \ref{P4} is equivalent to the simultaneous validity of the following two conditions: 
	\begin{enumerate}[label=(\alph*)]
		\item \label{LNa} $t^{\frac{q}{p}-1}b^q(t)$ is integrable on some deleted neighbourhood of zero.
		\item \label{LNb} $t^{\frac{q}{p} -q -1}b^q(t)$ is integrable on some deleted neighbourhood of infinity.
	\end{enumerate}
	Just as in Proposition~\ref{LQ}, condition \ref{LNa} is satisfied for $p < \infty$ by \ref{SV4}, and assumed in the remaining case by \ref{Caseiii}. Similarly, \ref{LNb} is satisfied for $p > 1$ by \ref{SV4} and assumed in the remaining case by \ref{Caseii}. The method for $q = \infty$ remains unchanged, only instead of \ref{SV4}, \ref{SV3} is used.
	
	These arguments also show that the presented conditions are necessary for \ref{P4}. The remaining claim that $L^{(p, q, b)}=\{ 0 \}$ when they are not satisfies is a direct consequence.
\end{proof}

We now show some properties that are common for all parameters $p, q, b$ allowed in Definition~\ref{DLK}. We note that these results are not interesting in the cases when $L^{(p, q, b)}=\{ 0 \}$, but we want to emphasise that it is only the property \ref{P4} that gets violated in those situations.

\begin{proposition}  \label{PP5}
	Let $p, q, b$ be as in Definition~\ref{DLK}. Then $\lVert \cdot \rVert_{(p, q, b)}$ satisfies \ref{Q1}, \ref{P2}, \ref{P3}, \ref{P5} and \ref{P6}.
\end{proposition}

\begin{proof}
	All the properties except \ref{P5} are direct consequences of the definition, the properties of the maximal function and the properties of the Lebesgue functionals. We thus focus on \ref{P5}.
	
	Fix some $E \subseteq R$ with measure $\mu(E) < \infty$ and arbitrary $f \in M$. We may assume that $f \neq 0$ (in $M$). Thanks to the Hardy-Littlewood inequality \eqref{THLI} it suffices to prove that there is some finite constant $C_E$ satisfying
	\begin{equation}
	\int_0^{\mu(E)} f^*(t) \: dt \leq  C_E \lVert f \rVert_{(p, q, b)} \label{PP51}
	\end{equation} 
	for all $f \in M$. It holds that
	\begin{equation*}
	\begin{split}
		\lVert f \rVert_{(p, q, b)} &\geq \left \lVert \chi_{(\mu(E), \infty)}(t) t^{\frac{1}{p}-\frac{1}{q} -1}b(t) \int_0^t f^{*}(s) \: ds \right \rVert_q \\
		&\geq C \int_0^{\mu(E)} f^{*}(s) \: ds,
	\end{split}
	\end{equation*}
	where 
	\begin{equation*}
	C = \lVert \chi_{(\mu(E), \infty)}(t) t^{\frac{1}{p}-\frac{1}{q} -1}b(t) \rVert_q > 0,
	\end{equation*}
	since both $t^{\frac{1}{p}-\frac{1}{q} -1}$ and $b$ are non-zero on $(0, \infty)$.
	
	Now, if $C < \infty$ then we have just proved \eqref{PP51} with $C_E = C^{-1}$, while in the other case we get from Proposition~\ref{PP4} that $\lVert f \rVert_{(p, q, b)} = \infty$ and \eqref{PP51} thus holds trivially.
\end{proof}

Now, it follows that if we take $q \in [1, \infty]$ we will have that both $f \mapsto f^{**}$ and $\lVert \cdot \rVert_q$ are subadditive, which gives us the remaining property of r.i.\ Banach function norms. Hence, we have the following proposition.

\begin{proposition} \label{LN}
	Let $p, q, b$ be as in Definition~\ref{DLK} and suppose  $q \in [1, \infty]$. Then $\lVert \cdot \rVert_{(p, q, b)}$ is a rearrangement invariant Banach function norm on $M$, and consequently $L^{(p, q, b)}$ is an r.i.\ Banach function space, if and only if $p$ and $b$ satisfy one of the following conditions:
	\begin{enumerate}
		\item $p \in (1, \infty)$,
		\item $p = 1$ and $\lVert t^{- \frac{1}{q}} b(t) \chi_{(1, \infty)}(t) \rVert_q < \infty$,
		\item $p= \infty$ and $\lVert t^{- \frac{1}{q}} b(t) \chi_{(0, 1)}(t) \rVert_q < \infty$.
	\end{enumerate}
	Furthermore,  if none of those conditions is satisfied, then $L^{(p, q, b)}=\{ 0 \}$.
\end{proposition}

It turns out that the assumption $q \in [1, \infty]$ is also necessary in quite a fundamental way: if $q < 1$ then $\lVert \cdot \rVert_{(p, q, b)}$ is not even equivalent to any Banach function norm. Since the proof requires some tools that will be established below, we prove this later in Theorem~\ref{TLKBFS2}.               

We follow with a useful observation similar to Proposition~\ref{PbND} which tells us that we may, when working with $L^{(1,\infty,b)}$, always assume that $b$ is non-increasing. Note that the assumption that $b$ is continuous is again present only to avoid unnecessary technicalities and comes at no loss of generality due to Theorem~\ref{TSmSV}. Furthermore, the assumption that it is bounded near infinity is natural and, as per Proposition~\ref{PP4}, also causes no loss of generality.

\begin{proposition} \label{PbNI}
	Let $b$ be a continuous s.v. function and assume that it is bounded near infinity. Then $L^{(1,\infty,b)} = L^{(1,\infty, \hat{b}_{\infty})}$, up to equivalence of norms, where $\hat{b}_{\infty}$ is the function defined in \ref{LTb}, that is
	\begin{align*}
		\hat{b}_{\infty}(t) &= \esssup_{s \in (t, \infty)} b(s) &\textup{for } t \in (0, \infty).
	\end{align*}.
\end{proposition}

\begin{proof}
	Since the function $t \mapsto t b(t) f^{**}(t)$ is, by our assumptions, continuous, the result follows from the following chain of relations that hold for every $f \in M$:
	\begin{equation*}
	\begin{split}
		\lVert f \rVert_{L^{(1, \infty, b)}} &\leq \lVert f \rVert_{L^{(1, \infty, \hat{b}_{\infty})}} = \sup_{t \in (0, \infty) } \sup_{s \in [t, \infty)} b(s) \int_0^{t} f^* \: d\lambda \\
		&\leq \sup_{t \in (0, \infty) } \sup_{s \in [t, \infty)} b(s) \int_0^{s} f^* \: d\lambda = \sup_{s \in (0, \infty)} b(s) \int_0^{s} f^* \: d\lambda = \lVert f \rVert_{L^{(1, \infty, b)}}.
	\end{split}
	\end{equation*}
\end{proof}

We conclude with the following important remark concerning the representation spaces of Lorentz--Karamata spaces.
\begin{remark} \label{RemRepSpace}
	The definitions of Lorentz--Karamata functionals and spaces do not in any way depend on the structure of the underlying non-atomic measure space $(R, \mu)$. Specifically, given a Lorentz--Karamata functional $\lVert \cdot \rVert_{p, q, b}$ (or $\lVert \cdot \rVert_{(p, q, b)}$) over $(R, \mu)$, we may use the very same definition to define a Lorentz--Karamata functional over $([0, \infty), \lambda)$. It is evident that all the properties shown in this section also apply to this new functional and that this functional is the representation functional of $\lVert \cdot \rVert_{p, q, b}$ (or $\lVert \cdot \rVert_{(p, q, b)}$, respectively). The same of course holds also for the Lorentz--Karamata spaces $L^{p,q,b}$ (or $L^{(p,q,b)}$).
\end{remark}

\subsection{The fundamental function and endpoint spaces} \label{SecFF}
In this section we compute the fundamental function for all kinds of Lorentz--Karamata spaces and describe the corresponding Lorentz and Marcinkiewicz endpoint spaces. We will always assume that $\mu(R) = \infty$, as the modifications for the remaining case are easy.

\begin{remark}
	Let us note that since Lorentz--Karamata spaces are in general just r.i.~quasi-Banach function spaces, their fundamental functions do not have to be quasi-concave, much less concave, and it might not be possible to renorm them in a way that would correct this. This creates a formal obstacle for the definition of the Lorentz and Marcinkiewicz endpoint spaces, but it can be corrected at least in the case when the fundamental function in question is equivalent to a quasi-concave function, for then there also exists some equivalent concave function for which the Lorentz and Marcinkiewicz endpoint spaces can be correctly defined. We will accept those as the Lorentz and Marcinkiewicz endpoint spaces corresponding to the original fundamental function and keep the notation established in Section~\ref{SSFF}, that is $\Lambda_{\varphi}$ and $M_{\varphi}$, respectively, even though the actual function defining those spaces might be different from $\varphi$. Note that this convention does not cause ambiguity, since the Lorentz and Marcinkiewicz endpoint spaces do not depend on the precise choice of the defining function, in the sense that choosing any equivalent concave function would yield the same spaces (for the Marcinkiewicz endpoint space the same thing holds even for any equivalent quasi-concave function).
\end{remark}

Let us now start with the more pathological cases.

\begin{proposition} \label{PFFT}
	Let $p, q, b$ be as in Definition~\ref{DLK} and suppose that $p<1$ or that both $p = 1$ and $b$ is not equivalent to a non-increasing function. Then the fundamental function $\varphi_X$ of $X = L^{p,q,b}$ is given for $t \in [0, \infty]$ by the formulas
	\begin{align} 
		\varphi_X(t) &\approx t^{\frac{1}{p}}b(t) & \text{for } t < \infty, \label{PFFT1} \\
		\varphi_X(t) &= \infty & \text{for } t = \infty. \label{PFFT2}
	\end{align}
	Thence it is not equivalent to a quasi-concave function and it makes no sense to consider the Lorentz and Marcinkiewicz endpoint spaces.
\end{proposition}

\begin{proof}
	The formulas \eqref{PFFT1} and \eqref{PFFT2} follow from the definition by a simple calculation. \eqref{PFFT1} uses \ref{LEFF}, while \eqref{PFFT2} follows from it using the Fatou property \ref{P3} of $\lVert \cdot \rVert_{p,q,b}$ and \ref{SV3}. Furthermore, since for $t \in (0, \infty)$ we have that
	\begin{align*}
		\frac{t}{\varphi_X(t)} &= t^{1-\frac{1}{p}}b^{-1}(t) &\textup{for } t \in (0, \infty),
	\end{align*}
	which is by our assumptions not equivalent to a non-decreasing function, we obtain that $\varphi_X$ is not equivalent to a quasi-concave function.
\end{proof}

\begin{remark}
	The fundamental function of $L^{(p, q, b)}$ is not interesting for $p<1$ because the space is trivial for this choice of $p$ (see Proposition~\ref{PP4}).
\end{remark}

We now follow with the much more interesting cases where the fundamental function is equivalent to a quasi-concave function.

\begin{theorem} \label{TFFA*}
	Let $p, q, b$ be as in Definition~\ref{DLK} and assume that either $p\in (1,\infty)$ or that both $p=1$ and $b$ is equivalent to a non-increasing function. Then the fundamental function $\varphi_X$ of $X = L^{p,q,b}$ is given for $t \in [0, \infty]$ by the formulas
	\begin{align*} 
		\varphi_X(t) &\approx t^{\frac{1}{p}}b(t) & \text{for } t < \infty, \\
		\varphi_X(t) &= \infty & \text{for } t = \infty.
	\end{align*}
	Thence it is equivalent to a quasi-concave function and the corresponding Lorentz and Marcinkiewicz endpoint spaces are given by
	\begin{align*}
		&\Lambda_{\varphi_X} = L^{p,1,b}, &M_{\varphi_X} = L^{(p, \infty, b)},
	\end{align*}
	up to equivalence of the defining functionals.
	
	Moreover, if $p>1$ then the same holds also for $X = L^{(p,q,b)}$.
\end{theorem}

\begin{proof}
	The formulas for $\varphi_X$ follow by precisely the same computation as in Proposition~\ref{PFFT}, but this time the ratio
	\begin{align*}
		\frac{t}{\varphi_X(t)} &= t^{1-\frac{1}{p}}b^{-1}(t), &t \in (0, \infty),
	\end{align*}
	is equivalent to a non-decreasing function and therefore we obtain that the function $\varphi_X$ is equivalent to a quasi-concave function. That $M_{\varphi_X} = L^{(p, \infty, b)}$ now follows directly from the corresponding definitions. Furthermore, the function $t^{\frac{1}{p}-1}b(t)$ is by our assumptions equivalent to some non-increasing function which we will denote $\psi$. Then it follows from \ref{LEFF} that
	\begin{align*}
		\varphi_X(t) &\approx \Psi(t) = \int_0^t \psi(s) \: ds &\textup{for } t \in (0, \infty), 
	\end{align*}
	where $\Psi$ is a concave function, and we may therefore use it to define our Lorentz endpoint space, that is to say we put $\Lambda_{\varphi_X} = \Lambda_{\Psi}$. We may now conclude that it holds for every $f \in M$ that
	\begin{equation*}
		\lVert f \rVert_{\Lambda_{\Psi}} = \int_0^{\infty} f^* \: d\Psi = \int_0^{\infty} f^* \psi \: d\lambda \approx \int_0^{\infty} f^*(t) t^{\frac{1}{p}-1}b(t) \: dt = \lVert f \rVert_{p,1,b}.
	\end{equation*}
	
	It remains to show that when $p>1$ the formulas for $\varphi_X$ hold also for $X = L^{(p,q,b)}$, but this is again a simple calculation using \ref{LEFF}, the Fatou property \ref{P3} of $\lVert \cdot \rVert_{(p,q,b)}$, and \ref{SV3}.
\end{proof}

\begin{theorem} \label{TFF(1)}
	Let $p, q, b$ be as in Definition~\ref{DLK} and assume that $p=1$ and that $\lVert t^{- \frac{1}{q}} b(t) \chi_{(1, \infty)}(t) \rVert_q < \infty$. Then the fundamental function $\varphi_X$ of $X = L^{(1,q,b)}$ is in the case $q<\infty$ given for $t \in [0, \infty]$ by the formulas
	\begin{align*} 
		\varphi_X(t) &\approx t \left ( \widehat{b^q}(t) \right )^{\frac{1}{q}} & \text{for } t < \infty, \\
		\varphi_X(t) &= \infty & \text{for } t = \infty,
	\end{align*}
	where $ \widehat{b^q}$ is defined as in \ref{LTb}, that is
	\begin{align*}
		\widehat{b^q}(t) &= \int_t^{\infty} s^{-1} b^q(s) \: ds;
	\end{align*}
	while in the case $q = \infty$ it is given for $t \in [0, \infty]$ by the formulas
	\begin{align*} 
		\varphi_X(t) &\approx t \hat{b}_{\infty}(t) & \text{for } t < \infty, \\
		\varphi_X(t) &= \infty & \text{for } t = \infty,
	\end{align*}
	where $\hat{b}_{\infty}(t)$ is defined as in \ref{LTb}, that is
	\begin{align*}
		\hat{b}_{\infty}(t) &= \esssup_{s \in (t, \infty)} b(s) &\textup{for } t \in (0, \infty).
	\end{align*}
	Thence it is in all cases equivalent to a quasi-concave function and the corresponding Lorentz and Marcinkiewicz endpoint spaces are given by
	\begin{align*}
		&\Lambda_{\varphi_X} = L^{1,1,a}, &M_{\varphi_X} = L^{(1, \infty, a)},
	\end{align*}
	up to equivalence of the defining functionals, where $a$ is the function given by
	\begin{align*}
		a &= \widehat{b^q}^{\frac{1}{q}} &\text{in the case } q<\infty, \\
		a &= \hat{b}_{\infty} &\text{in the case } q=\infty.
	\end{align*}
\end{theorem}

\begin{proof}
	We will prove the formula for $\varphi_X(t)$ only for the case when both $q < \infty$ and $t < \infty$, as the approach is similar for $q = \infty$ and $t < \infty$ while the extension to $t = \infty$ is exactly the same as in the proofs above. Under those assumptions a direct computation using \ref{LEFF} shows that
	\begin{equation*}
		\varphi_X(t) \approx t \left( b^q(t) + \widehat{b^q}(t) \right)^{\frac{1}{q}}.
	\end{equation*}
	Since we know from \ref{LTb} that $0 < b^q \lesssim \widehat{b^q}$, we may now simplify this expression to
	\begin{equation*}
		\varphi_X(t) \approx t \left ( \widehat{b^q}(t) \right )^{\frac{1}{q}},
	\end{equation*}
	as desired.
	
	Having obtained the formulas, we may now observe that $\varphi_X$ is in all cases equivalent to a quasi-concave function. This follows because the function $a$, as defined in the formulation of the theorem, is in all cases decreasing. We may thus use the same approach as in the proof of Theorem~\ref{TFFA*} (with $\psi = a$) to show that $\Lambda_{\varphi_X} = L^{1,1,a}$, while it again follows directly from the definitions that $M_{\varphi_X} = L^{(1, \infty, a)}$.
\end{proof}

Note that the assumption $\lVert t^{- \frac{1}{q}} b(t) \chi_{(1, \infty)}(t) \rVert_q < \infty$ in the theorem above is natural, because by Proposition~\ref{PP4} it is equivalent to the non-triviality of the space $L^{(1, q, b)}$.

Moreover, we would like to point out the interesting fact that the fundamental function of $L^{(1,q,b)}$ depends on the value of $q$. To obtain a concrete example, consider the s.v. function $b$ given by
\begin{equation}
	b(t) =
	\begin{cases}
		\log \left(  \frac{e}{t} \right) & \text{for } t \leq 1, \\
		\log^{-1} (et) & \text{for } t \geq 1.
	\end{cases}
\end{equation}
A simple calculation shows that, for this choice of $b$, the fundamental functions of the spaces $L^{(1,q,b)}$ with different values of $q$ are fundamentally different, i.e.~they are not equivalent up to a multiplicative constant. Here we of course consider only the choices of $q$ for which $\lVert t^{- \frac{1}{q}} b(t) \chi_{(1, \infty)}(t) \rVert_q < \infty$, i.e.~$q>1$.

\begin{theorem} \label{TFFE}
	Let $p, q, b$ be as in Definition~\ref{DLK} and assume that $p=\infty$ and that $\lVert t^{- \frac{1}{q}} b(t) \chi_{(0, 1)}(t) \rVert_q < \infty$. In the case when $q=\infty$ we further assume that $b$ is absolutely continuous and non-decreasing. Then the fundamental function $\varphi_X$ of $X = L^{\infty, q, b}$ is in the case $q<\infty$ given for $t \in [0, \infty]$ by the formulas
	\begin{align*} 
		\varphi_X(t) &= \left ( \widetilde{b^q}(t) \right )^{\frac{1}{q}} & \text{for } t < \infty, \\
		\varphi_X(t) &= \int_0^{\infty} s^{-1} b^q(s) \: ds & \text{for } t = \infty,
	\end{align*}
	where $ \widetilde{b^q}$ is defined as in \ref{LTb}, that is
	\begin{align*}
		\widetilde{b^q}(t) &= \int_0^{t} s^{-1} b^q(s) \: ds;
	\end{align*}
	while in the case $q = \infty$ it is given for $t \in [0, \infty]$ by the formulas
	\begin{align}
		\varphi_X(t) &= \tilde{b}_{\infty}(t) & \text{for } t < \infty, \label{TFFE1} \\
		\varphi_X(t) &= \esssup_{s \in (0, \infty)} b(s) & \text{for } t = \infty, \label{TFFE2}
	\end{align}
	where $\tilde{b}_{\infty}(t)$ is defined as in \ref{LTb}, that is
	\begin{align*}
		\tilde{b}_{\infty}(t) &= \esssup_{s \in (0, t)} b(s).
	\end{align*}
	Thence it is in all cases equivalent to a quasi-concave function and the corresponding Marcinkiewicz endpoint spaces is given by
	\begin{align*}
		M_{\varphi_X} = L^{(\infty, \infty, \varphi_X)},
	\end{align*}
	up to equivalence of the defining functionals. As for the Lorentz endpoint space, in the case when $q < \infty$ it is given by
	\begin{align*}
		\Lambda_{\varphi_X} = L^{\infty,1,a}
	\end{align*}
	 where $a$ is the function given by
	\begin{align*}
		a &= \widetilde{b^q}^{\frac{1}{q} - 1}b^q.
	\end{align*}
	In the remaining case $q=\infty$ the Lorentz endpoint space can be characterised as
	\begin{align*}
		\Lambda_{\varphi_X} &= \Lambda^1(b') & \text{in the case when } \lim_{t \to 0^+} b(t) = 0, \\
		\Lambda_{\varphi_X} &= \Lambda^1(b') \cap L^{\infty} & \text{in the case when } \lim_{t \to 0^+} b(t) > 0,
	\end{align*}
	where $b'$ denotes the classical derivative of $b$ and $\Lambda^1(b')$ is a classical Lorentz space.
	
	Moreover, if we denote $\tilde{X} = L^{(p,q,b)}$ then $\varphi_{\tilde{X}} \approx \varphi_X$ and thus the corresponding Lorentz and Marcinkiewicz endpoint spaces remain the same.
\end{theorem}

Note that the assumption that $b$ is absolutely continuous and non-decreasing comes at no loss of generality due to Theorem~\ref{TSmSV} and Proposition~\ref{PbND} (used in this order since the transformation $b \mapsto\tilde{b}_{\infty}$ preserves local Lipschitz continuity). Furthermore, it would be possible under this assumption to replace $\tilde{b}_{\infty}$ on the right-hand side of \eqref{TFFE1} by just $b$ and the essential supremum on the right-hand side of \eqref{TFFE2} by the limit at $\infty$, but we choose to present a formula that is valid regardless of those extra assumptions.
 
\begin{proof}
	The formula for $\varphi_X$ follows immediately from definition, we only need to keep in mind that we are considering a limiting case where we cannot provide a specific value at infinity (see Remark~\ref{RLC}). Furthermore, $\varphi_X$ is clearly equivalent to a quasi-concave function since it is in all cases increasing while the ratio $\frac{t}{\varphi_X(t)}$ is of the form $t \bar{a}$, where $\bar{a}$ is an s.v.~function, which implies that is is equivalent to a non-decreasing function.
	
	It now follows directly from the corresponding definitions that $M_{\varphi_X} = L^{(\infty, \infty, \varphi_X)}$. Note that as per Lemma~\ref{LSV} the function $\varphi_X$ is in all cases an s.v.~function and thus the notation $L^{(\infty, \infty, \varphi_X)}$ is correct. 
	
	On the other hand, in order to treat the Lorentz endpoint $\Lambda_{\varphi_X}$ we need to consider separately the cases $q<\infty$ and $q=\infty$. As for the former, it follows from our assumption $\lVert t^{- \frac{1}{q}} b(t) \chi_{(0, 1)}(t) \rVert_q < \infty$ that $\varphi_X$ is for $q \in (0, \infty)$ an absolutely continuous function. It is therefore differentiable $\lambda$-a.e.~on $(0, \infty)$ and the derivative is given for the appropriate $t \in (0,\infty)$ by
	\begin{align*}
		\varphi'_X(t) &= \left(  \int_0^{t} s^{-1} b^q(s) \: ds \right)^{\frac{1}{q} -1} t^{-1} b^q(t) = t^{-1}a(t).
	\end{align*}
	Since $a$ is an s.v.~function (see Lemma~\ref{LSV}), this is equivalent to some non-increasing function $\psi$. We thus obtain a concave function $\Psi$ such that
	\begin{align*}
		\Psi(t) &= \int_0^{t} \psi \: d\lambda \approx \int_0^{t} s^{-1}a(s) \: ds = \varphi_X(t) &\textup{for } t \in (0, \infty),
	\end{align*}
	which we may use to define the desired space $\Lambda_{\varphi_X}$, that is to say we put $\Lambda_{\varphi_X} = \Lambda_{\Psi}$. It now follows that it holds for every $f \in M$ that
	\begin{equation*}
	\lVert f \rVert_{\Lambda_{\Psi}} = \int_0^{\infty} f^* \: d\Psi = \int_0^{\infty} f^* \psi \: d\lambda \approx \int_0^{\infty} f^*(t) t^{-1}a(t) \: dt = \lVert f \rVert_{\infty,1,a}.
	\end{equation*}
	
	In the remaining case we assume that $\varphi_X = b$ is non-increasing and absolutely continuous, possibly except at zero. Hence we can easily describe the absolutely continuous (with respect to Lebesgue measure $\lambda$) part of the measure induced by it, while the singular part of this measure is just a multiple of the Dirac measure at zero. We can thus rewrite the Lebesgue--Stieltjes integral defining $\lVert \cdot \rVert_{\Lambda_{\varphi_{X}}}$ to obtain
	\begin{equation*}
		\lVert \cdot \rVert_{\Lambda_{\varphi_{X}}} = \left ( \lim_{t \to 0^+} b(t) \right) f^*(0) + \int_0^{\infty} f^* b' \: d\lambda.
	\end{equation*}
	Remembering that $f^*(0) = \lVert f \rVert_{\infty}$, we conclude that the desired characterisation indeed holds.
	
	Finally, to calculate the fundamental function of $L^{(p,q,b)}$ one uses approach almost identical to that employed in the proof of Theorem~\ref{TFF(1)}.
\end{proof}

Let us remark that the characterisation of $\Lambda_{\varphi_{X}}$ in the case when $X = L^{\infty, \infty, b}$ is rather unsatisfying, since unlike in the other cases treated in this section it does not provide any information about the structure of the weight in the space. This could be remedied if the conjecture presented in \cite[Remark~3.7]{PesaSV} were correct. If that were case, then the space $\Lambda_{\varphi_{X}}$ could again be characterised using Lorentz--Karamata spaces.

\subsection{Boyd indices} \label{SBI}
In this section we compute the Boyd indices of the Lorentz--Karamata spaces $L^{p,q,b}$ and $L^{(p,q,b)}$ for all choices of the parameters $p,q,b$ for which the respective spaces are non-trivial. The remaining cases are, of course, of no interest. We do not list the precise choices of parameters that satisfy this assumption because we do not want to clutter the formulation of the theorem. Instead, we refer the reader to Proposition~\ref{LQ} and Proposition~\ref{PP4}.

\begin{theorem} \label{TBI}
	Let $p, q, b$ be as in Definition~\ref{DLK}. Then the Boyd indices of $L^{p,q,b}$ and $L^{(p,q,b)}$ satisfy
	\begin{equation*}
		\underline{\alpha}_{L^{p,q,b}} = \overline{\alpha}_{L^{p,q,b}} = \underline{\alpha}_{L^{(p,q,b)}} = \overline{\alpha}_{L^{(p,q,b)}} = \frac{1}{p},
	\end{equation*}
	provided that the parameters $p,q,b$ are chosen in such a way that the space in question is non-trivial (i.e.~it contains a non-zero function).
\end{theorem}

\begin{proof}	
	In the proof we must work with the representation functional of the Lorentz--Karamata functional in question. As observed in Remark~\ref{RemRepSpace}, the representation functional always exists and it is given by the same formula as the original Lorentz--Karamata functional. We will, therefore, use for it the same notation already established for Lorentz--Karamata functionals. Observe also that it holds for any function $f \in M([0,\infty), \lambda)$ that
	\begin{equation*}
		(D_tf)^* = D_t (f^*).
	\end{equation*}
		
	We employ the characterisation of the indices contained in Proposition~\ref{PBI}. We will only compute $\underline{\alpha}_{L^{p,q,b}}$ as the approach is identical for both indices of $L^{p,q,b}$, while to modify it for the indices of $L^{(p,q,b)}$ one only needs to add one extra change of coordinates. We will also only cover the case $q < \infty$ as the modification for $q= \infty$ is simple. To this end, we assume that $t \in (0,1)$, fix any $f \in L^{p,q,b}$ and compute via a simple change of coordinates:
	\begin{equation*} \label{TBI1}
		\lVert D_{\frac{1}{t}}f \rVert_{p, q, b} = \left( \int_0^{\infty} s^{\frac{q}{p} -1} b^q(s) (f^*(t^{-1}s))^q \: ds \right)^{\frac{1}{q}} = t^{\frac{1}{p}} \left( \int_0^{\infty} r^{\frac{q}{p} -1} b^q(tr) (f^*(r))^q \: dr \right)^{\frac{1}{q}}
	\end{equation*}
	We may now fix any $\varepsilon > 0$ and use \ref{CSV} to see that there is a constant $C_{\varepsilon} \geq 1$, depending only on $\varepsilon$, such that
	\begin{equation*}
		\lVert D_{\frac{1}{t}}f \rVert_{p, q, b} \geq C_{\varepsilon}^{-1} t^{ \frac{1}{p} + \varepsilon} \left( \int_0^{\infty} r^{\frac{q}{p} -1} b^q(r) (f^*(r))^q \: dr \right)^{\frac{1}{q}} = C_{\varepsilon}^{-1} t^{ \frac{1}{p} + \varepsilon} \lVert f \rVert_{p,q,b}.
	\end{equation*}
	Hence, given our assumption that the space in question is non-trivial, the operator norm of $D_{\frac{1}{t}}$ as an operator $L^{p,q,b} \to L^{p,q,b}$ satisfies
	\begin{equation*}
		\left \lVert D_{\frac{1}{t}} \right \rVert \geq C_{\varepsilon}^{-1} t^{ \frac{1}{p} + \varepsilon}.
	\end{equation*}
	This implies the following estimate:
	\begin{equation*}
		\frac{\log\left(\left \lVert D_{\frac{1}{t}} \right \rVert\right)}{\log(t)} \leq \frac{-\log(C_{\varepsilon}) + \left( \frac{1}{p} + \varepsilon \right)\log(t) }{\log(t)}.
	\end{equation*}
    We may now combine this estimate with Proposition~\ref{PBI} to obtain
	\begin{equation*}
		\underline{\alpha}_{L^{p,q,b}} = \lim_{t \to 0^+} \frac{\log\left(\left \lVert D_{\frac{1}{t}} \right \rVert\right)}{\log(t)} \leq \frac{1}{p} + \varepsilon.
	\end{equation*}
	Since $\varepsilon$ was arbitrary, this establishes the estimate
	\begin{equation*}
		\underline{\alpha}_{L^{p,q,b}} \leq \frac{1}{p}.
	\end{equation*}
	The reverse estimate can be established in a similar manner, only using the second inequality in \eqref{CSV1} instead of the first.
\end{proof}

\subsection{Embeddings among the spaces $L^{p,q,b}$} \label{SSE}

Let us now turn our attention to the embeddings among Lorentz--Karamata spaces. In this section we will focus on those cases when both spaces participating in the embedding are of the form $L^{p,q,b}$, the remaining cases will be treated later. The first proposition is a generalisation of a result about embeddings between Lorentz spaces $L^{p,q}$, which can be found for example in \cite[Chapter~4, Proposition~4.2]{BennettSharpley88}.

\begin{proposition} \label{PELK}
	Let $b_1, b_2$ be two s.v.~functions such that $b_2 \lesssim b_1$, let $p \in (0, \infty]$ and suppose  $0 < q < r \leq \infty$. Then 
	\begin{equation}
	\lVert \cdot \rVert_{p, r, b_2} \lesssim \Vert \cdot \rVert_{p, q, b_1} \label{PELK1}
	\end{equation}
	or equivalently,
	\begin{equation*}
	L^{p, q, b_1} \hookrightarrow L^{p, r, b_2},
	\end{equation*}
	and the embedding is proper, in the sense that there is a function $f \in M$ such that $f \in L^{p, r, b_2}$ but $f \notin L^{p, q, b_1}$.
\end{proposition}

\begin{proof}
	Fix arbitrary $f \in M$ and $q \in (0, \infty)$ and assume first that $r = \infty$. Then we may use either \ref{LEFF} or \ref{LTb}, in dependence on whether $p < \infty$ of $p = \infty$, and the fact that $f^*$ is non-increasing to compute
	\begin{equation} \label{PELK4}
	\begin{split}
	\lVert f \rVert_{p, \infty, b_2} &\lesssim \sup_{t \in (0, \infty)} t^{\frac{1}{p}} b_1(t) f^*(t) \\
	&\lesssim \sup_{t \in (0, \infty)} \left ( \int_0^{t} s^{\frac{q}{p} -1} b_1^q(s) \: ds \right )^{\frac{1}{q}} f^*(t)  \\
	&= \lVert f \rVert_{p, q, b_1}
	\end{split}
	\end{equation}
	which proves \eqref{PELK1} for $r = \infty$. Note that the finiteness of the expressions in question is of no concern even for $p = \infty$ since in the case when $b_1$ satisfies $\lVert t^{- \frac{1}{q}} b_1(t) \chi_{(0, 1)}(t) \rVert_q < \infty$ then it is also bounded on $(0, 1)$, which implies that $b_2$ is bounded on $(0, 1)$ too, while in the remaining case there is nothing to prove.
	
	We now use \eqref{PELK4} to compute for $r \in (q, \infty)$
	\begin{equation*}
	\begin{split}
	\lVert f \rVert_{p, r, b_2} &\lesssim \left ( \int_0^{\infty} s^{\frac{q}{p} -1} b_1^q(s) (f^*(s))^q \left [s^{\frac{1}{p}} b_2(s) (f^*(s)) \right ]^{r-q} \: ds \right )^{\frac{1}{r}} \\
	&\leq \lVert f \rVert_{p, \infty, b_2}^{\frac{r-q}{r}} \cdot \lVert f \rVert_{p, q, b_1}^{\frac{q}{r}} \\
	&\lesssim  \lVert f \rVert_{p, q, b_1},
	\end{split}
	\end{equation*}
	which together with \eqref{PELK4} proves \eqref{PELK1} for all the desired parameters $p, q, r, b_1, b_2$. Again, the finiteness of the expression is of no concern thanks to a similar argument as the one presented above.
	
	As for the second part of the statement, is suffices to consider some $f \in M(R, \mu)$ that satisfies $f^*(t) \approx t^{-\frac{1}{p}} \lvert \log(t) \rvert^{-\frac{1}{q}} b_1^{-1}(t) \chi_{(0, \min\{\mu(R), \, 1\})}(t)$.
\end{proof}

In the next theorem we present a full characterisation of embeddings between Lorentz--Karamata spaces. This result first appeared in \cite{Bathory18}, but we chose to present it here because the original formulation does not present the different conditions in their explicit form. We also chose to present its full proof since the original one fails to consider some of the cases and is at times quite vague. In our proof, we use the abstract theory of classical Lorentz spaces and the previous proposition to cover most of the cases, while in some of the remaining ones we employ ideas inspired by the work in \cite{Bathory18} and \cite{Neves02}.
\begin{theorem} \label{TELK}
	Let $b_1, b_2$ be two s.v.\ functions and let $p_1, p_2 \in (0, \infty]$ and $q_1, q_2 \in (0, \infty]$. If $p_i$ is infinite, for $i=1,2$, then we further suppose that the corresponding function $b_i$ satisfies $\lVert t^{- \frac{1}{q}} b_i(t) \chi_{(0, 1)}(t) \rVert_{q_i} < \infty$. We then have the following characterisation:
	\begin{enumerate}
		\item
		If $p_1 > p_2$ then the embedding $L^{p_1, q_1, b_1} \hookrightarrow L^{p_2, q_2, b_2}$ holds if and only if $\mu(R) < \infty$.
		\item 
		If $p_1 = p_2$ and $q_1 \leq q_2$ then the embedding $L^{p_1, q_1, b_1} \hookrightarrow L^{p_2, q_2, b_2}$ holds if and only if one of the following condition holds:
		\begin{enumerate}[label=(\alph*)]
			\item
			$p_1=p_2<\infty$ and $\frac{b_2}{b_1}$ is bounded on $(0, \mu(R))$,			
			\item
			$p_1 = p_2 = \infty$, $q_1, q_2 < \infty$, and the functions $b_1, b_2$ satisfy
			\begin{equation}
			\sup_{r \in (0, \mu(R))} \frac{\left ( \int_0^{r} t^{-1}b_2^{q_2}(t) \: dt \right )^{\frac{1}{q_2}}}{\left ( \int_0^{r} t^{-1}b_1^{q_1}(t) \: dt \right )^{\frac{1}{q_1}}} < \infty, \label{TELK0}
			\end{equation}
			\item
			$p_1 = p_2 = \infty$, $q_1 < \infty$, $q_2 = \infty$, and the functions $b_1, b_2$ satisfy
			\begin{equation}
			\sup_{r \in (0, \mu(R))} \frac{b_2(r)}{\left ( \int_0^{r} t^{-1}b_1^{q_1}(t) \: dt \right )^{\frac{1}{q_1}}} < \infty, \label{TELK0.1}
			\end{equation}
			\item
			$p_1 = p_2 = \infty$, $q_1 = q_2 = \infty$, and $\frac{b_2}{b_1}$ is bounded on $(0, \mu(R))$.
		\end{enumerate}
		\item 
		If $p_1 = p_2$, $q_1 > q_2$, and $r$ is the number satisfying $\frac{1}{r} = \frac{1}{q_2} - \frac{1}{q_1}$, then the embedding $L^{p_1, q_1, b_1} \hookrightarrow L^{p_2, q_2, b_2}$ holds if and only if one of the following condition holds:
		\begin{enumerate}[label=(\alph*)]
			\item
			$p_1=p_2<\infty$ and the functions $b_1, b_2$ satisfy
			\begin{equation}
			\int_0^{\mu(R)} t^{-1} \left ( \frac{b_2(t)}{b_1(t)} \right )^{r} \: dt < \infty, \label{TELK0.2}
			\end{equation}
			\item 
			$p_1 = p_2 = \infty$, $q_1 < \infty$, and the functions $b_1, b_2$ satisfy
			\begin{equation} \label{TELK00}
			\int_0^{\mu(R)} \left ( \frac{\int_0^{t} s^{-1}b_2^{q_2}(s) \: ds}{\int_0^{t} s^{-1}b_1^{q_1}(s) \: ds} \right )^{\frac{r}{q_1}} t^{-1}b_2^{q_2}(t) \: dt < \infty,
			\end{equation}
			\item
			$p_1=p_2 = \infty$, $q_1 = \infty$, and the functions $b_1, b_2$ satisfy
			\begin{equation*}
			\int_0^{\mu(R)} t^{-1} \left ( \frac{b_2(t)}{b_1(t)} \right )^{r} \: dt < \infty. \label{TELK0.3}
			\end{equation*}
		\end{enumerate}
		\item
		If $p_1 < p_2$ then the embedding $L^{p_1, q_1, b_1} \hookrightarrow L^{p_2, q_2, b_2}$ never holds.
	\end{enumerate}
\end{theorem}

\begin{proof}
	We distinguish several cases.
	
	Suppose first that both $q_1, q_2 < \infty$. For this case, we will use \cite[Proposition~1]{Stepanov93} (see also \cite[Theorem~3.1]{CarroPick00} for further context). We further distinguish two subcases.
	
	Let us at first assume that $q_1 \leq q_2$. Then the first part of \cite[Proposition~1]{Stepanov93} (see also \cite{Sawyer90}) tells us that for some two positive weights $v, w$ the inequality
	\begin{equation*}
	\lVert w f^* \rVert_{q_2} \lesssim \lVert v f^* \rVert_{q_1} \label{TELK1}
	\end{equation*}
	holds for all $f \in M$ if and only if
	\begin{equation}
	\sup_{r \in (0, \infty)} \lVert w \cdot \chi_{(0, r)} \rVert_{q_2} \cdot ( \lVert v \cdot \chi_{(0, r)} \rVert_{q_1} )^{-1} < \infty. \label{TELK2} 
	\end{equation}
	Since the non-increasing rearrangement $f^*$ of any $f \in M(R, \mu)$ is supported on the set $[0, \mu(R)]$ and therefore $\lVert f \rVert_{p, q, b} = \lVert t^{\frac{1}{p}-\frac{1}{q}}b(t)\chi_{(0,\mu(R))}(t)f^*(t) \rVert_q$ for all $f \in M(R, \mu)$, we observe that the embedding in question is equivalent to \eqref{TELK2} with
	\begin{align*}
	v(t) &= t^{\frac{1}{p_1}-\frac{1}{q_1}}b_1(t)\chi_{(0,\mu(R))}(t) &\textup{for } t \in (0, \infty),  \\
	w(t) &= t^{\frac{1}{p_2}-\frac{1}{q_2}}b_2(t)\chi_{(0,\mu(R))}(t) &\textup{for } t \in (0, \infty).
	\end{align*}
	We first rewrite it to a more explicit form and get
	\begin{equation}
	\sup_{r \in (0, \mu(R))} \frac{\bigg ( \int_0^{r} t^{\frac{q_2}{p_2}-1}b_2^{q_2}(t) \: dt \bigg )^{\frac{1}{q_2}}}{\bigg ( \int_0^{r} t^{\frac{q_1}{p_1}-1}b_1^{q_1}(t) \: dt \bigg )^{\frac{1}{q_1}}} < \infty. \label{TELK4}
	\end{equation}
	Now, by considering separately the four cases distinguished by finiteness or infiniteness of $p_1$ and $p_2$ and by using \ref{SV1}, \ref{SV2}, \ref{LEFF}, \ref{LTb}, and \ref{SV3}, one obtains that \eqref{TELK4} holds if and only if one of the following conditions holds:
	\begin{enumerate}
		\item $p_1 > p_2$ and $\mu(R) < \infty$,
		\item $p_1 = p_2 < \infty$ and $\frac{b_2(r)}{b_1(r)}$ is bounded on $(0, \mu(R))$,
		\item $p_1 = p_2 = \infty$ and \eqref{TELK0} holds.
	\end{enumerate}
	
	We have thus covered the subcase $q_1 \leq q_2$ and it remains to consider the subcase $q_1 > q_2$. The approach is similar, only we use the second part of \cite[Proposition~1]{Stepanov93} and thus deduce that the estimate in question is equivalent to the convergence of the integral
	\begin{equation} \label{TELK5}
	\int_0^{\mu(R)} \left ( \frac{\int_0^{t} s^{\frac{q_2}{p_2}-1}b_2^{q_2}(s) \: ds}{\int_0^{t} s^{\frac{q_1}{p_1}-1}b_1^{q_1}(s) \: ds} \right )^{\frac{r}{q_1}} t^{\frac{q_2}{p_2}-1}b_2^{q_2}(t) \: dt.
	\end{equation}
	By a similar approach to that above, we deduce that \eqref{TELK5} holds if and only if one of the following conditions holds:
	\begin{enumerate}
		\item $p_1 > p_2$ and $\mu(R) < \infty$,
		\item $p_1 = p_2 < \infty$ and \eqref{TELK0.2} holds,
		\item $p_1 = p_2 = \infty$ and \eqref{TELK00} holds.
	\end{enumerate}
	
	We now turn our attention to the cases when at least one of the parameters $q_1, q_2$ is infinite. Those cases do not follow from the more abstract theory of classical Lorentz spaces as using this theory would force us to add unnecessary assumptions on the s.v.~functions $b_1, b_2$. We first show that the presented conditions are sufficient.
	
	The sufficiency in the case when $p_1 = p_2$, $q_1 \leq q_2 = \infty$ and $\frac{b_2}{b_1}$ is bounded on $(0, \mu(R))$ is covered by Proposition~\ref{PELK} (except the case when also $q_1 = q_2$ which is trivial). That the weaker condition \eqref{TELK0.1} suffices in the case when $p_1 = p_2 = \infty$, $q_1 < q_2 = \infty$ follows by modifying the first step of the proof of Proposition~\ref{PELK}, namely \eqref{PELK4}:	
	\begin{equation*} \label{TELK6}
	\begin{split}
	\lVert f \rVert_{\infty, \infty, b_2} &\lesssim \sup_{t \in (0, \infty)} \left ( \int_0^{t} s^{-1} b_1^{q_1}(s) \: ds \right )^{\frac{1}{q_1}} f^*(t)  \\
	&\leq \lVert f \rVert_{\infty, q_1, b_1}.
	\end{split}
	\end{equation*}

	As for the case when $p_1 = p_2$, $q_2 < q_1 = \infty$, the sufficiency follows from the calculation
	\begin{equation*}
	\begin{split}
	\lVert f \rVert_{p_2, q_2, b_2} &\leq \left ( \int_0^{\mu(R)} s^{-1} b_2^{q_2}(s) b_1^{-q_2}(s) \: ds \right )^{\frac{1}{q_2}} \cdot \sup_{s \in (0, \mu(R))} s^{\frac{1}{p_2}} b_1(s) f^*(s) \\
	&\lesssim \lVert f \rVert_{p_1, q_1, b_1},
	\end{split}
	\end{equation*}
	where the last estimate uses \eqref{TELK0.2} and which holds for any $f \in M$.
	
	It remains to show the sufficiency in the case when $p_1 > p_2$ and $\mu(R) < \infty$. We may assume that $q_1 = \infty$, since we know from Proposition~\ref{PELK} that $L^{p_1, q_1, b_1} \hookrightarrow L^{p_1, \infty, b_1}$. If $q_2 < \infty$, then the desired embedding follows from the calculation
	\begin{equation*}
	\begin{split}
	\lVert f \rVert_{p_2, q_2, b_2} &\leq \left ( \int_0^{\mu(R)} s^{\frac{q_2}{p_2}-\frac{q_2}{p_1}-1} b_2^{q_2}(s) b_1^{-q_2}(s) \: ds \right )^{\frac{1}{q_2}} \cdot \sup_{s \in (0, \mu(R))} s^{\frac{1}{p_1}} b_1(s) (f^*(s)) \\
	&\lesssim \lVert f \rVert_{p_1, \infty, b_1},
	\end{split}
	\end{equation*}
	where the last estimate holds for every $f \in M$ because the function $b_2^{q_2} b_1^{-q_2}$ is s.v. The case when $q_2 = \infty$ is similar.
	
	Finally, we show that the presented conditions for the cases when at least one of the parameters $q_1, q_2$ is infinite are also necessary. 
	
	The case when $p_1 = p_2$ is quite simple. In the subcase when also $q_1 \leq q_2$ it suffices to test the required inequality
	\begin{equation} \label{TELK7}
	\lVert \cdot \rVert_{p_2, q_2, b_2} \lesssim \lVert \cdot \rVert_{p_1, q_1, b_1}
	\end{equation}
	by functions $\chi_{E_t}$, where $t \in (0, \mu(R))$, $E_t \subseteq R$ and $\mu(E_t) = t$, and applying either \ref{LEFF} or monotonicity of $b_i$ as appropriate, depending on the finiteness or infiniteness of the parameters $p_i$ and $q_i$ (the monotonicity in the appropriate cases can be assumed thanks to Proposition~\ref{PbND}). In the remaining subcase when $q_1 > q_2$ we test the inequality \eqref{TELK7} by a function $f \in M(R, \mu)$ satisfying $f^*(t) \approx t^{-\frac{1}{p_1}}b_1(t)^{-1} \chi_{(0, \mu(R))}(t)$.
	
	In the case when $p_1 < p_2$ one can observe that the embedding does not hold by considering some $f \in M(R, \mu)$ such that  $f^*(t) \approx t^{-\frac{1}{p_1}} \lvert \log(t) \rvert^{-\frac{2}{q_1}} b_1^{-1}(t) \chi_{(0, \min\{\mu(R), \, 1 \})}(t)$.

	Similarly, in the case when $p_1 > p_2$ and $\mu(R) = \infty$ it is sufficient to consider some $f \in M(R, \mu)$ such that $f^*(t) \approx \chi_{(0, 1)} + t^{-\frac{1}{p_1}} \lvert \log(t) \rvert^{-\frac{2}{q_1}} b_1^{-1}(t) \chi_{[1, \infty)}(t)$.
\end{proof}
 
We conclude this section by presenting a corollary of Theorem~\ref{TELK} that fully characterises for which choices of the parameters $p,q,b$ does the functional $\lVert \cdot \rVert_{p,q,b}$ satisfy \ref{P5}. The result uses the abstract theory developed in \cite{Pesa22} to make rigorous the rather intuitive idea that the conditions for the validity of \ref{P5} should be the same as the conditions for the embedding $L^{p,q,b} \hookrightarrow L^1$ over the sets of finite measure.

\begin{corollary}\label{CP5b}
	Let $p, q, b$ be as in Definition~\ref{DLK}. Then $\lVert \cdot \rVert_{p, q, b}$ satisfies \ref{P5} if and only if one of the following conditions holds:
	\begin{enumerate}
		\item $p>1$;
		\item $p=1$, $q \leq 1$, and $b^{-1}$ is bounded on $(0, 1)$;
		\item $p=1$, $q>1$, and $\int_0^{1} t^{-1} b^{-q'} \: dt < \infty$.
	\end{enumerate}
\end{corollary}

\begin{proof}
	Assume first that $\mu(R) = \infty$. It then follows from the results contained in \cite{Pesa22} that $\lVert \cdot \rVert_{p, q, b}$ satisfies \ref{P5} if and only if it holds for every $f \in M([0,\infty), \lambda)$ that 
	\begin{equation}\label{CP5b1}
		 \lVert f^* \chi_{(0, 1)} \rVert_{1} \lesssim \lVert f^* \chi_{(0, 1)} \rVert_{p,q,b}.
	\end{equation}
	Here, the functional on the right-hand side is the representation functional of $\lVert \cdot \rVert_{p, q, b}$ which is given by the same formula and for which we therefore use the same notation (see Remark~\ref{RemRepSpace}).
	
	To be more specific, if we follow the notation established in \cite{Pesa22}, that is, we put for any pair of r.i.~quasi-Banach function spaces $A, B$ over $([0,\infty), \lambda)$ and any $f \in M([0,\infty), \lambda)$
	\begin{equation*}
			\lVert f \rVert_{WL(A, B)} = \lVert f^* \chi_{[0,1]} \rVert_A + \lVert f^* \chi_{(1, \infty)} \rVert_B
	\end{equation*}
 	and
 	\begin{equation*}
 		WL(A, B) = \{f \in M; \; \lVert f \rVert_{WL(A, B)} < \infty \},
 	\end{equation*}
 	then we may use \cite[Theorem~5.7]{Pesa22} to obtain that $\lVert \cdot \rVert_{p, q, b}$ as a functional over $([0,\infty), \lambda)$ satisfies \ref{P5} if and only if
 	\begin{equation} \label{CP5b2}
 		L^{p,q, b} \hookrightarrow WL(L^1, L^{p,q,b}).
 	\end{equation}
 	Here, all the spaces are over $([0,\infty), \lambda)$. Furthermore, $\lVert \cdot \rVert_{p, q, b}$ satisfies \ref{P5} as a functional over $([0,\infty), \lambda)$ if and only if it satisfies it as a functional over $(R, \mu)$. Indeed, one implication follows from the estimate
 	\begin{align*}
 		\int_E \lvert f \rvert \: d\mu &\leq \int_0^{\mu(E)} f^* \: d\lambda \lesssim  \lVert f^* \rVert_{p, q, b} = \lVert f \rVert_{p, q, b} &\textup{for } f \in M(R, \mu), 
 	\end{align*}
 	which holds provided that $\lVert \cdot \rVert_{p, q, b}$ satisfies \ref{P5} as a functional over $([0,\infty), \lambda)$. As for the remaining implication, if $\lVert \cdot \rVert_{p, q, b}$ does not satisfy \ref{P5} as a functional over $([0,\infty), \lambda)$ then it follows from \cite[Theorem~3.11]{NekvindaPesa20} that there is a function $f_0 \in M([0,\infty), \lambda)$ and a set $E \subseteq [0, \infty)$ such that $\lVert f_0 \rVert_{p, q, b} < \infty$ and $\lambda(E) < \infty$ while
 	\begin{equation*}
 		\infty = \int_E \lvert f_0 \rvert \: d\lambda \leq \int_0^{\mu(E)} f_0^* \: d\lambda.
 	\end{equation*}
 	By finding some $\tilde{E} \subseteq R$ such that $\mu(\tilde{E}) = \lambda(E)$, we may now use resonance of the space $(R, \mu)$ to obtain
 	\begin{equation*}
 		\sup_{\substack{f \in M(R, \mu) \\ f^* = f_0^*}} \int_{\tilde{E}} \lvert f \rvert \: d\mu = \int_0^{\lambda(E)} f_0^* \: d\lambda = \infty.
 	\end{equation*}
 	As all the functions $f$ considered in the supremum on the left-hand side clearly satisfy
 	\begin{equation*}
 		\lVert f \rVert_{p, q, b} = \lVert f_0 \rVert_{p, q, b} < \infty,
 	\end{equation*}
 	we must conclude that $\lVert \cdot \rVert_{p, q, b}$ does not satisfy \ref{P5} as a functional over $(R, \mu)$.
 	
 	Now, using \cite[Theorem~5.6]{Pesa22} we see that the embedding \eqref{CP5b2} is equivalent to the following one:
 	\begin{equation*}
 		WL(L^{p,q, b}, L^{\infty}) \hookrightarrow WL(L^1, L^{\infty}),
 	\end{equation*}
 	which is the same as the inequality
 	\begin{equation*}
		\lVert f \rVert_{WL(L^1, L^{\infty})} \lesssim \lVert f \rVert_{WL(L^{p,q, b}, L^{\infty})} 
 	\end{equation*}
 	holding for every $f \in M([0,\infty), \lambda)$. Finally, the estimate $(5)$ in \cite[Proposition~3.3]{Pesa22} applied on the both side of this inequality leads to \eqref{CP5b1}. We note that \cite[Proposition~3.3]{Pesa22} assumes that the space in question is a Banach function spaces, and indeed the presented proof of the estimate $(5)$ uses its property \ref{P5}, but this assumption is unnecessary for the validity of the estimate we need, because it can be just as easily proved using the method presented in the last step of the proof of \cite[Theorem~3.4]{Pesa22} and this method relies only on the property \ref{P4} and the second part of the property \ref{Q1}, both of which $\lVert \cdot \rVert_{p, q, b}$ satisfies by Proposition~\ref{LQ}.
	
	By further rewriting \eqref{CP5b1} as
	\begin{equation*}
		\int_0^{1} f^*(t) \: dt \lesssim \left( \int_0^{1} t^{\frac{q}{p}-1} b^q(t) (f^*(t))^q \: dt \right)^{\frac{1}{q}}
	\end{equation*}
	and remembering that for every function $f_0 \in M([0,1], \lambda)$ there is a function $f \in M([0,\infty), \lambda)$ such that $f^* = f_0^*$, we observe that the validity of \eqref{CP5b1} for every $f \in M([0,\infty), \lambda)$ is exactly the same question as the validity of the embedding $L^{p,q,b} \hookrightarrow L^1$ in the case when the underlying measure space is $([0,1], \lambda)$. Since $L^1 = L^{1,1,a}$, where $a$ is identically equal to one on $(0, \infty)$, we may now apply Theorem~\ref{TELK} to obtain the desired characterisation.

	To obtain the same conditions in the case when $\mu(R) < \infty$, one only needs to observe that the validity of \ref{P5} is in this situation trivially equivalent to the embedding $L^{p,q,b} \hookrightarrow L^1$ (both spaces are considered over $(R, \mu)$) and that the conditions characterising this embedding (as obtained in Theorem~\ref{TELK}) are equivalent to the desired ones thanks to the properties of s.v.~functions stated in Lemma~\ref{LSV}.
\end{proof}

\subsection{Relations between \texorpdfstring{$L^{p, q, b}$}{L\^p,q,b}  and \texorpdfstring{$L^{(p, q, b)}$}{L\^(p,q,b)}} \label{SSQN}

In this section we study when $L^{p, q, b}=L^{(p, q, b)}$, which is a property that often plays an important role in applications, see for example \cite{BasteroMilman03}, \cite{MilmanPustylnik04}, and \cite{Turcinova19}.

The first observation follows from Propositions~\ref{LQ} and \ref{LN} since they tell us that, for $p<1$, $L^{(p,q,b)}$ is always trivial while $L^{p,q,b}$ never is. So we obtain $p \geq 1$ as a necessary condition.

Note that because we always have for any $f \in M$ that $f^* \leq f^{**}$, we get trivially from properties of $\lVert \cdot \rVert_q$ that $\lVert \cdot \rVert_{p, q, b} \lesssim \lVert \cdot \rVert_{(p, q, b)}$ for all $p, q, b$ allowed by Definition~\ref{DLK}. Thus, we ``only'' have to examine what are the parameters for which $\lVert \cdot \rVert_{(p, q, b)} \lesssim \lVert \cdot \rVert_{p, q, b}$.

We will first show in the next theorem that $p>1$ is sufficient. To this end, we will employ weighted Hardy inequalities.

\begin{theorem}\label{TEQN}
	Let $p, q, b$ be as in Definition~\ref{DLK} and let $q \in (0, \infty]$ and $p \in (1, \infty]$. Then $\lVert \cdot \rVert_{(p, q, b)} \lesssim \lVert \cdot \rVert_{p, q, b}$.
\end{theorem}

\begin{proof}
	Consider first the case when $q \geq 1$.	We employ a weighted Hardy inequality, which can be found for example in \cite[Section~1.3, Theorem~2]{Maz'ya11} and which yields that the desired conclusion is in fact equivalent to the following condition:
	\begin{equation} \label{TEQN1}
	\sup_{r \in (0, \infty)} \left \lVert t^{\frac{1}{p}-\frac{1}{q} -1}b(t)\chi_{(r, \infty)}(t) \right \rVert_q \cdot \left \lVert \left (t^{\frac{1}{p}-\frac{1}{q}}b(t) \right )^{-1}\chi_{(0, r)}(t) \right \rVert_{q'} < \infty.
	\end{equation}
	
	Firstly, note that both the norms inside the supremum are, thanks to either \ref{SVC} and \ref{SV3} or \ref{SV4}, in dependence on the value of $q$ and $q'$, always finite by our assumption $p \in (1, \infty]$. 
	
	Secondly, since we assume $p>1$ we can use \ref{LEFF} to show that product inside the supremum in \eqref{TEQN1} is on $(0, \infty)$ equivalent to a constant. The condition is thus satisfied.

	Consider now the remaining case $q < 1$. This time we utilise a weighted Hardy-type inequality which has been obtained independently by Carro and Soria in \cite[Proposition~2.6b]{CarroSoria93}, Lai in \cite[Theorem~2.2]{Lai93}, and Stepanov in \cite[Theorem~3b]{Stepanov93} (see also \cite[Theorem~4.1]{CarroPick00} for context as well as a formulation more similar to the problem at hand). From this result, we obtain that the desired inequality holds if and only if the following two conditions hold:
	\begin{align}
		\sup_{r \in (0, \infty)} \bigg ( \int_0^r t^{\frac{q}{p}-1}b^q(t) \: dt \bigg )^{\frac{1}{q}} \cdot \bigg ( \int_0^r t^{\frac{q}{p}-1}b^q(t) \: dt \bigg )^{-\frac{1}{q}} &< \infty, \label{T2EQN1} \\
	\sup_{r \in (0, \infty)} r \cdot \bigg ( \int_r^{\infty} t^{\frac{q}{p}-q-1}b^q(t) \: dt \bigg )^{\frac{1}{q}} \cdot \bigg ( \int_0^r t^{\frac{q}{p}-1}b^q(t) \: dt \bigg )^{-\frac{1}{q}}  &< \infty. \label{T2EQN2}
	\end{align}
	
	As for the condition \eqref{T2EQN1}, we may assume that the integrals inside the supremum are finite. Indeed, if $p \in (1, \infty)$ this follows from \ref{SV4} while in the remaining case we can assume it since otherwise $L^{p,q,b} = L^{(p,q,b)} = \{0\}$ and there is nothing to prove. Therefore, the function we take supremum of in \eqref{T2EQN1} is identically equal to $1$ on $(0, \infty)$ and the condition is satisfied.
	
	It remains to verify \eqref{T2EQN2} which follows from \ref{LEFF} when $p< \infty$ and from \ref{LEFF} and \ref{LTb} when $p=\infty$.
\end{proof}

What now remains is to consider the case $p=1$. We shall show that in this case the equality never holds. The reason behind this behaviour is \eqref{LTb1}, more precisely the relation between $b^q$ and $\widehat{b^q}$ (for $q< \infty$) and between $b^{-1}$ and $\widetilde{b^{-1}}$ (for $q = \infty$), where $\widehat{b^q}$ and $\widetilde{b^{-1}}$ are the functions defined in \ref{LHb} (for $b^q$ and $b^{-1}$, respectively), that is
\begin{align*}
\widehat{b^q}(t) &= \int_t^{\infty} s^{-1} b^q(s) \: ds, & \widetilde{b^{-1}}(t) &= \int_0^{t} s^{-1} b^{-1}(s) \: ds &\textup{for } t \in (0, \infty).
\end{align*}

We begin with the case $q=1$ which is worth singling out.

\begin{proposition} \label{PEbbH}
	Let $p, q, b$ be as in Definition~\ref{DLK} and suppose $q = p = 1$ and that $b$ satisfies $\lVert t^{- 1} b(t) \chi_{(1, \infty)}(t) \rVert_1 < \infty$. Then 
	\begin{equation}
	L^{(1,1,b)} = L^{1, 1, \hat{b}}, \label{PEbbH1}
	\end{equation}
	where $\hat{b}$ is the function defined in \ref{LHb}. Consequently, $L^{(1,1,b)} \neq L^{1,1,b}$ for any s.v.~function $b$.
\end{proposition}

\begin{proof}
	The conclusion \eqref{PEbbH1} follows from Definition~\ref{DLK} by simple calculation, using only the classical Fubini's theorem. 
	
	For the remaining part, it follows from Theorem~\ref{TELK} that $L^{1, 1, \hat{b}} = L^{1,1,b}$ if and only if $b \approx \hat{b}$ on $(0, \infty)$, which can never hold as follows from \eqref{LTb1}.
\end{proof}

For the remaining cases we will have to use some abstract theory.

\begin{theorem}\label{T3EQN}
	Let $p, q, b$ be as in Definition~\ref{DLK} and suppose $p = 1$. Then $L^{(1,1,b)} \neq L^{1,1,b}$.
\end{theorem}

\begin{proof}
	The case $q=1$ has already been covered in the previous proposition.
	
	If $q \in (1, \infty)$, then it follows from \cite[Theorem~4]{Sawyer90} that $L^{(1,q,b)} = L^{1,q,b}$ if and only if
	\begin{equation} \label{T3EQN1}
		\int_t^{\infty} s^{-1} b^q(s) \: ds \lesssim \frac{1}{t^q} \int_0^t s^{q-1} b^q(s) \: ds
	\end{equation}
	on $(0,\infty)$. However, by \ref{LEFF}, the right-hand side of \eqref{T3EQN1} is equivalent to $b^q(t)$ while the left-hand side is just $\widehat{b^q}$. Hence, \eqref{T3EQN1} holds if and only if $b^q \approx \widehat{b^q}$ on $(0, \infty)$, which can never hold as follows from \eqref{LTb1}.

	For the case $q \in (0,1)$ we use the same result that we used when proving the second part of Theorem~\ref{TEQN} and that has been obtained independently in \cite[Proposition~2.6b]{CarroSoria93}, \cite[Theorem~2.2]{Lai93}, and \cite[Theorem~3b]{Stepanov93} and which can be also found in \cite[Theorem~4.1]{CarroPick00}. It follows from this result that in order for $L^{(1,q,b)} = L^{1,q,b}$ to hold it is necessary and sufficient that the conditions \eqref{T2EQN1} and \eqref{T3EQN1} hold. Now, while \eqref{T2EQN1} always holds by the same argument as in the proof of Theorem~\ref{TEQN}, the second condition \eqref{T3EQN1} is never satisfied as follows from \eqref{LTb1}.

	Finally if $q=\infty$, it follows from \cite[Theorem~5.1]{GogatishviliSoudsky14} and \ref{LEFF} that $L^{(1,\infty,b)} = L^{1,\infty,b}$ if and only if the function $b \widetilde{b^{-1}}$ is bounded on $(0, \infty)$, which is in turn equivalent to $b^{-1} \approx \widetilde{b^{-1}}$ on $(0, \infty)$ and thus not true for any s.v.~function $b$ (again, by \ref{LTb1}).
\end{proof}

The following corollary collects these results into a single statement.

\begin{corollary} \label{CEQNa}
	Let $p, q, b$ be as in Definition~\ref{DLK}. Then $L^{(p,q,b)} = L^{p,q,b}$ if and only if $p>1$.
\end{corollary}

In conclusion of this section, we present two consequences of Corollary~\ref{CEQNa} which extend some properties of $\lVert \cdot \rVert_{(p, q, b)}$ to $\lVert \cdot \rVert_{p, q, b}$. The second corollary is much less interesting than the first, as its content is already contained in Corollary~\ref{CP5b}, but we believe it worth stating nonetheless because we have obtained the result by much different means.

\begin{corollary} \label{CEQN}
	Let $p, q, b$ be as in Definition~\ref{DLK}, suppose that $p>1$ and that $q \in [1, \infty]$. Then $\lVert \cdot \rVert_{p, q, b}$ is equivalent to an r.i.\ Banach function norm $\lVert \cdot \rVert_{(p, q, b)}$ and the space $L^{p, q, b} = L^{(p, q, b)}$ equipped with this norm is an r.i.\ Banach function space.
\end{corollary}

\begin{corollary}
	Let $p, q, b$ be as in Definition~\ref{DLK} and suppose that $p>1$. Then $\lVert \cdot \rVert_{p, q, b}$ satisfies \ref{P5}.
\end{corollary}

\subsection{Absolute continuity of the quasinorm} \label{SecACN}

In this section we characterise which functions in Lorentz--Karamata spaces have absolutely continuous quasi-norm. We begin with the spaces $L^{p,q,b}$.

\begin{theorem} \label{TACN}
	Let $p, q, b$ be as in Definition~\ref{DLK}. In the case when $p=q=\infty$ we further assume that $b$ is non-decreasing. If $q < \infty$, then $L^{p,q,b}_a = L^{p,q,b}$ (i.e.~the space $L^{p,q,b}$ has absolutely continuous quasinorm), while in the case when $q = \infty$ it holds that
	\begin{equation*}
	L^{p, \infty, b}_a = \left\{ f \in L^{p, \infty, b}; \; \lim_{t \to 0_+} f^*(t) t^{\frac{1}{p}} b(t) = \lim_{t \to \infty} f^*(t) t^{\frac{1}{p}} b(t) = 0 \right\}.
	\end{equation*}
	
	Specifically, if the parameters $p$ and $b$ are chosen in such a way that the space $L^{p, \infty, b}$ is non-trivial, then it does not have absolutely continuous quasinorm.
\end{theorem}

Note that the monotonicity assumption on $b$ in the case when $p=q=\infty$ comes at no loss of generality as follows from Proposition~\ref{PbND}.

\begin{proof}
	Let us first consider, for any $f \in M_0$ and any $t_0 \in (0, \infty)$, the functions
	\begin{align*}
	f_0 &= \min \left\{\lvert f \rvert, \, f^*(t_0)\right\} \sgn(f), \\
	f_1 &= \max \left\{\lvert f \rvert - f^*(t_0), \, 0 \right\} \sgn(f),
	\end{align*}
	and the set
	\begin{equation*}
	R_1 = \left\{ x \in R; \; \lvert f_1 \rvert > 0 \right\} = \left\{ x \in R; \; \lvert f \rvert > f^*(t_0) \right\}.
	\end{equation*}
	We immediately see that $f = f_0 + f_1$ and that $\mu(R_1) \leq t_0 < \infty$. Moreover, it is an exercise to verify that
	\begin{align*}
	f_0^* &= f^*(t_0) \chi_{[0, t_0)} + f^* \chi_{[t_0, \infty)}, \\
	f_1^* &= (f^* - f^*(t_0)) \chi_{[0, t_0)}.
	\end{align*}
	
	For the remainder of the proof, $E_k$ will denote some sequence of arbitrary subsets of $R$ such that $\chi_{E_k} \to 0$ $\mu$-a.e.
	
	If we now assume that $q < \infty$ and $f \in L^{p,q,b}$, then for any fixed $\varepsilon > 0$ we may choose some sufficiently large $t_0 \in (0, \infty)$ such that with the notation established above we have
	\begin{equation*}
	\lVert f_0 \rVert_{p, q, b} = \left( \int_{0}^{\infty} (f_0^*(t))^q t^{\frac{q}{p} -1} b^q(t) \: dt  \right)^{\frac{1}{q}} < \varepsilon.
	\end{equation*}	
	This follows from the Lebesgue dominated convergence theorem, because any $f \in L^{p,q,b}$ satisfies
	\begin{equation*}
	\lim_{t \to \infty} f^*(t) = 0
	\end{equation*}
	and thus for any sequence $t_n$ of positive numbers, such that $t_n \to \infty$ as $n \to \infty$, the corresponding functions $f_0^*$ converge to zero $\lambda$-a.e., while we also have that
	\begin{equation*}
	\int_0^{\infty} (f_0^*(t))^q t^{\frac{q}{p} -1} b^q(t) \: dt \leq \lVert f \rVert_{p, q, b}^q < \infty
	\end{equation*}
	for any such $f_0$.
	
	Now, we get that 
	\begin{equation*}
	\lVert f\chi_{E_k} \rVert_{p, q, b} \lesssim \lVert f_1 \chi_{E_k} \rVert_{p, q, b} + \lVert f_0 \chi_{E_k} \rVert_{p, q, b},
	\end{equation*}
	since $\lVert \cdot \rVert_{p, q, b}$ is a quasinorm (by Proposition~\ref{LQ}). As for the first term, we obtain from the continuity of measure that
	\begin{align*}
	\mu \left( E_k \cap R_1 \right) &\leq \mu \left( \left( \cup_{l \geq k}  E_l \right) \cap R_1 \right) \to 0 &\textup{as } k \to \infty.
	\end{align*}
	It then follows from the Lebesgue dominated convergence theorem that
	\begin{align*}
	\lVert f_1\chi_{E_k} \rVert_{p, q, b} &\leq \left( \int_0^{\mu \left( E_k \cap R_1 \right)} (f^*(t))^q t^{\frac{q}{p} -1} b^q(t) \: dt  \right)^{\frac{1}{q}} \to 0 &\textup{as } k \to \infty,
	\end{align*}
	since the integrals in question are again estimated by $\lVert f \rVert_{p, q, b}^q$. The estimate for the second term now follows directly from our choice of $t_0$, as
	\begin{equation*}
	\lVert f_0 \chi_{E_k} \rVert_{p, q, b} \leq \lVert f_0 \rVert_{p, q, b} <\varepsilon.
	\end{equation*}
	
	We have thus shown that
	\begin{equation*}
	\limsup_{k \to \infty} \lVert f\chi_{E_k} \rVert_{p, q, b} \lesssim \varepsilon.
	\end{equation*}
	Since $\varepsilon$ was arbitrary, we have thus proved the statement for the case $q < \infty$.
	
	Let us now turn to the case when $q = \infty$. We will assume that $L^{p,\infty,b}$ is non-trivial (i.e.~that either $p < \infty$ or $p = \infty$ and $b$ is bounded on some neighbourhood of zero), because the remaining case is trivial.
	
	To establish the sufficiency, we proceed by fixing some $\varepsilon < \infty$ and considering some function $f \in L^{p, \infty, b}$ such that
	\begin{equation} \label{TACN0}
	\lim_{t \to 0_+} f^*(t) t^{\frac{1}{p}} b(t) = \lim_{t \to \infty} f^*(t) t^{\frac{1}{p}} b(t) = 0
	\end{equation}
	and some sufficiently large $t_0 \in (0, \infty)$ such that
	\begin{align*}
	f^*(t) t^{\frac{1}{p}} b(t) &< \varepsilon &\textup{for } t \in [t_0, \infty).
	\end{align*}
	It follows that the corresponding function $f_0$ satisfies
	\begin{align} \label{TACN1}
	f_0^*(t) t^{\frac{1}{p}} b(t) &\lesssim \varepsilon &\textup{for } t \in (0, \infty),
	\end{align}
	where the extension follows directly from the definition of s.v.~functions in the case $p < \infty$ and from our assumption that $b$ is non-decreasing in the case $p = \infty$. We stress that the constant in the estimate depends only on the function $b$, not on our choice of $t_0$.
	
	As before, we may estimate
	\begin{equation*}
	\lVert f\chi_{E_k} \rVert_{p, \infty, b} \lesssim \lVert f_1 \chi_{E_k} \rVert_{p, \infty, b} + \lVert f_0 \chi_{E_k} \rVert_{p, \infty, b},
	\end{equation*}
	where \eqref{TACN1} immediately yields that $\lVert f_0 \chi_{E_k} \rVert_{p, \infty, b} \lesssim \varepsilon$. As for the remaining term, the same argument as above can be used to obtain $\mu \left( E_k \cap R_1 \right) \to \infty$ as $k \to \infty$ and thus, by our assumption \eqref{TACN0},
	\begin{align*}
	\lVert f_1 \chi_{E_k} \rVert_{p, \infty, b} &\leq \esssup_{t \in (0, \mu(E_k \cap R_1))} f^*(t) t^{\frac{1}{p}} b(t) \to 0 &\textup{as } k \to \infty.
	\end{align*}	
	
	We have thus shown that
	\begin{equation*}
	\limsup_{k \to \infty} \lVert f\chi_{E_k} \rVert_{p, \infty, b} \lesssim \varepsilon.
	\end{equation*}
	Since $\varepsilon$ was arbitrary, we have thus proved the sufficiency of the condition \eqref{TACN0} in the case $q = \infty$.
	
	Let us now concentrate on the necessity of \eqref{TACN0}, so let us fix some $f \in L^{p, \infty, b}_a$. That
	\begin{equation*}
	\lim_{t \to 0_+} f^*(t) t^{\frac{1}{p}} b(t) = 0
	\end{equation*}
	follows by taking as $E_k$ some sequence of sets such that
	\begin{equation} \label{TACN1.5}
	\left\{ x \in R; \; \lvert f(x) \rvert > f^*(k^{-1}) \right\} \subseteq E_k \subseteq \left\{ x \in R; \; \lvert f(x) \rvert \geq f^*(k^{-1}) \right\}
	\end{equation}
	and $\mu(E_k) = k^{-1}$. Indeed, for this choice of $E_k$ we have that
	\begin{equation*}
	(f \chi_{E_k})^* = f^*\chi_{[0, k^{-1})}
	\end{equation*}
	and thus
	\begin{align*}
	\sup_{t \in (0, k^{-1})} f^*(t) t^{\frac{1}{p}} b(t) &= \lVert f \chi_{E_k} \rVert_{p, \infty, b} \to 0 &\textup{as } k \to \infty.
	\end{align*}
	
	As for the necessity of
	\begin{equation} \label{TACN2}
	\lim_{t \to \infty} f^*(t) t^{\frac{1}{p}} b(t) = 0,
	\end{equation}
	we will assume that $\mu(R) = \infty$ because otherwise there is nothing to prove. We consider first the cases when $p < \infty$ or both $p = \infty$ and $b$ is not bounded on any neighbourhood of infinity. In those cases we know that any $f \in L^{p, \infty, b}$ satisfies
	\begin{equation*}
	\lim_{t \to \infty} f^*(t) = 0.
	\end{equation*}
	Therefore we may choose the sequence of functions
	\begin{equation*}
	f_n = \min \left\{\lvert f \rvert, \, f^*(n)\right\}
	\end{equation*}
	to obtain
	\begin{align*}
	\sup_{t \in (n, \infty)} f^*(t) t^{\frac{1}{p}} b(t) &\leq \lVert f_n \rVert_{p,\infty,b} \to 0 &\textup{as } n \to \infty,
	\end{align*}
	where the first estimate is due to
	\begin{equation*}
	f_n^* = f^*(n) \chi_{[0, n)} + f^* \chi_{[n,\infty)}
	\end{equation*}
	and the convergence is due to Proposition~\ref{PropACN}, since clearly $\lvert f \rvert \geq \lvert f_n \rvert \downarrow 0$ $\mu$-a.e. (as $n \to \infty$).
	
	Finally, let us consider the necessity of \eqref{TACN2} in the case then $p= \infty$ and $b$ is bounded on some neighbourhood of infinity (i.e.~$b$ is bounded on $(0, \infty)$ as we assume that it is non-decreasing). In this case, \eqref{TACN2} is equivalent to simply
	\begin{equation*}
	\lim_{t \to \infty} f^*(t) = 0.
	\end{equation*}
	We will assume that a given function $f \in L^{\infty, \infty, b}$ satisfies
	\begin{equation*}
	\lim_{t \to \infty} f^*(t) > 0
	\end{equation*}
	and show that this implies that $f \notin L^{\infty, \infty, b}_a$. With this assumption, we see that there is some $\alpha > 0$ such that
	\begin{equation*}
	\mu \left( \left\{x \in R; \; \lvert f \rvert > \alpha \right\} \right) = \infty.
	\end{equation*}
	By fixing some increasing sequence $R_k$ of sets of finite measure such that
	\begin{equation*}
	R = \bigcup_{k \in \mathbb{N}} R_k
	\end{equation*}
	and putting
	\begin{equation*}
	E_k = \left\{x \in R; \; \lvert f \rvert > \alpha \right\} \setminus R_k,
	\end{equation*}
	we see $\mu(E_k) = \infty$ for every $k \in \mathbb{N}$ and that $\chi_{E_k} \to 0$ as $k \to \infty$. We now compute
	\begin{equation*}
	\lVert f \chi_{E_k} \rVert_{\infty, \infty, b} \geq \lVert \alpha \chi_{E_k} \rVert_{\infty, \infty, b} = \alpha \sup_{t \in (0, \infty) } b(t) > 0.
	\end{equation*} 
	Hence, $f$ does not have absolutely continuous quasinorm.
	
	It remains to show that if $L^{p, \infty, b}$ is non-trivial, then it contains a function violating \eqref{TACN0}, which is an easy exercise and therefore left to the reader.
\end{proof}

We now turn our attention to the spaces $L^{(p,q,b)}$. We only need to consider the case $p=1$, since the cases when $p>1$ are, thanks to Corollary~\ref{CEQNa}, already covered by Theorem~\ref{TACN}, while for $p<1$ we know thanks to Proposition~\ref{PP4} that $L^{(p,q,b)} = \{ 0\}$ and thus the question is both trivial and uninteresting.

\begin{theorem}
	Let $p, q, b$ be as in Definition~\ref{DLK} and assume that $p=1$. In the case when $q=\infty$ we further assume that $b$ is non-increasing. If $q < \infty$, then $L^{(1,q,b)}_a = L^{(1,q,b)}$ (i.e.~the space $L^{(1,q,b)}$ has absolutely continuous quasinorm), while in the case when $q = \infty$ it holds that
	\begin{align*}
	L^{(1, \infty, b)}_a &=	\left\{ f \in L^{(1, \infty, b)}; \; \lim_{t \to 0_+} tb(t) f^{**}(t) = \lim_{t \to \infty} tb(t) f^{**}(t) = 0 \right\} &\text{if } \lim_{t \to \infty} b(t) = 0, \\
	L^{(1, \infty, b)}_a &= \left\{ f \in L^{(1, \infty, b)}; \; \lim_{t \to 0_+} tb(t) f^{**}(t) = 0 \right\} &\text{if } \lim_{t \to \infty} b(t) > 0.
	\end{align*}
	
	Specifically, if the parameter $b$ is chosen in such a way that the space $L^{(1, \infty, b)}$ is non-trivial, then it does not have absolutely continuous quasinorm.
\end{theorem}

Note that the assumption that $b$ is non-increasing if $q = \infty$ comes at no loss of generality due to Proposition~\ref{PbNI}.

\begin{proof}
	We will use the notion already introduced in the proof of Theorem~\ref{TACN}, so we refer the reader there for the meanings of the symbols $f, f_0, f_1, R_1, t_0$.
	
	The case when $q<\infty$ is analogous to that in Theorem~\ref{TACN} so we will be brief and mainly outline the differences. We have for any $f \in L^{(1,q,b)}$ that
	\begin{equation*}
	\lim_{t \to \infty} f^*(t) \leq \lim_{t \to \infty} f^{**}(t) = 0
	\end{equation*}
	and
	\begin{equation*}
	f_0^{**} \leq f^*(t_0),
	\end{equation*}
	so we may again choose for any $\varepsilon > 0$ some $t_0 \in (0, \infty)$ such that
	\begin{equation*}
	\lVert f_0 \rVert_{(1, q, b)} < \varepsilon
	\end{equation*}
	and thus
	\begin{equation*}
	\lVert f \chi_{E_k} \rVert_{(1, q, b)} \lesssim \lVert f_1 \chi_{E_k} \rVert_{(1, q, b)} + \lVert f_0 \chi_{E_k} \rVert_{(1, q, b)} < \lVert f_1 \chi_{E_k} \rVert_{(1, q, b)} + \varepsilon.
	\end{equation*}
	Furthermore, it holds for every $t \in (0, \infty)$ that
	\begin{align*}
	(f_1 \chi_{E_k})^{**}(t) &\leq \frac{1}{t} \int_0^{t} f^* \chi_{[0, \mu(E_k \cap R_1))} \: d\lambda \to 0 & \text{as } k \to \infty,
	\end{align*}
	since the integrand converges to zero $\lambda$-a.e.~and we have the majorant $f^* \chi_{[0, t_0)}$. Hence we obtain as before that
	\begin{align*}
	\lVert f_1 \chi_{E_k} \rVert_{(1, q, b)} &\to 0 & \text{as } k \to \infty.
	\end{align*}
	Since $\varepsilon$ was arbitrary, we have thus covered the case when $q < \infty$.
	
	It remains to consider the case $q= \infty$. We begin with sufficiency by estimating the norms of $f_0$ and $f_1 \chi_{E_k}$. We have for $f_0$ that
	\begin{equation} \label{T2ACN2}
	\begin{split}
	\lVert f_0 \rVert_{(1, \infty, b)} &= \sup_{t \in (0, \infty)} b(t) \int_0^{t} f^*(t_0) \chi_{[0, t_0)} + f^* \chi_{[t_0, \infty)} \: d\lambda \\
	&= \max \left \{ \sup_{t \in (0, t_0]} tb(t)f^*(t_0), \, \sup_{t \in [t_0, \infty)} b(t)t_0f^*(t_0) + b(t)\int_{t_0}^{t} f^* \: d\lambda \right \}.
	\end{split}
	\end{equation}
	Now, $t \mapsto tb(t)$ is equivalent to a non-decreasing function, while we assume that $b$ is non-increasing. It follows that
	\begin{align*}
	\sup_{t \in (0, t_0]} tb(t)f^*(t_0) &\approx t_0 b(t_0) f^*(t_0), \\
	\sup_{t \in [t_0, \infty)} b(t)t_0f^*(t_0) &= t_0 b(t_0) f^*(t_0).
	\end{align*}
	Thence we conclude that
	\begin{equation*}
	\lVert f_0 \rVert_{(1, \infty, b)} \lesssim t_0 b(t_0) f^*(t_0) + \sup_{t \in [t_0, \infty)} b(t)\int_{t_0}^{t} f^* \: d\lambda.
	\end{equation*}
	As for $f_1 \chi_{E_k}$, we put $t_k = \mu(E_5 \cap R_1)$ and compute
	\begin{align*}
	\lVert f_1 \chi_{E_k} \rVert_{(1, \infty, b)} &\leq \sup_{t \in (0, \infty)} b(t) \int_0^{t} f^* \chi_{[0,t_k)} \: d\lambda \\
	&= \max \left \{ \sup_{t \in (0, t_k]} b(t) \int_0^{t} f^*  \: d\lambda , \, \sup_{t \in [t_k, \infty)} b(t) \int_0^{t_k} f^* \: d\lambda \right \}.
	\end{align*}
	Since we assume that $b$ is non-increasing, we conclude that
	\begin{equation*}
	\lVert f_1 \chi_{E_k} \rVert_{(1, \infty, b)} \leq  \sup_{t \in (0, t_k]} b(t) \int_0^{t} f^*  \: d\lambda.
	\end{equation*}
	
	Consider now the case when $\lim_{t \to \infty} b(t) = 0$ and assume we have some $f \in L^{(1,\infty, b)}$ satisfying
	\begin{equation} \label{T2ACN1}
	\lim_{t \to 0_+} tb(t) f^{**}(t) = \lim_{t \to \infty} tb(t) f^{**}(t) = 0.
	\end{equation}
	Since
	\begin{equation*}
	t_0 b(t_0) f^*(t_0) + \sup_{t \in [t_0, \infty)} b(t)\int_{t_0}^{t} f^* \: d\lambda \leq t_0 b(t_0) f^{**}(t_0) + \sup_{t \in [t_0, \infty)} tb(t) f^{**}(t)
	\end{equation*}
	the condition on limit at infinity in \eqref{T2ACN1} guarantees that we may choose for any $\varepsilon > 0$ some $t_0 \in (0, \infty)$ such that
	\begin{equation*}
	\lVert f_0 \rVert_{(1, \infty, b)} < \varepsilon.
	\end{equation*}
	On the other hand, the condition on limit at zero in \eqref{T2ACN1} guarantees that
	\begin{align*}
	\lVert f_1 \chi_{E_k} \rVert_{(1, \infty, b)} &\to 0 &\textup{as } k \to \infty.
	\end{align*}
	As before, this shows that the condition \eqref{T2ACN1} is sufficient.
	
	In the case when $\lim_{t \to \infty} b(t) = \beta > 0$, we obtain that any $f \in L^{(1, \infty, b)}$ satisfies
	\begin{equation*}
	\int_0^{\infty} f^* \: d\lambda = \sup_{t \in [0, \infty)} t f^{**}(t) \leq \beta^{-1} \sup_{t \in [0, \infty)} tb(t) f^{**}(t) = \lVert f \rVert_{(1, \infty, b)} < \infty.
	\end{equation*}
	It thus follows that
	\begin{align*}
	\sup_{t \in [t_0, \infty)} b(t) \int_{t_0}^{t} f^* \: d\lambda &\leq b(t_0) \int_{t_0}^{\infty} f^* \: d\lambda \to 0 & \text{as } t_0 \to \infty.
	\end{align*}
	Furthermore, since we have for $t \geq 2t_0$ that
	\begin{equation*}
	b(t) \int_{t_0}^{t} f^* \: d\lambda \geq b(t) (t-t_0) f^*(t) \geq \frac{1}{2} tb(t)f^*(t),
	\end{equation*}
	we obtain that
	\begin{equation*}
	\lim_{t \to \infty} tb(t) f^*(t) = 0.
	\end{equation*}
	We may thus again choose, for any $\varepsilon > 0$, some $t_0 \in (0, \infty)$ such that
	\begin{equation*}
	\lVert f_0 \rVert_{(1, \infty, b)} < \varepsilon.
	\end{equation*}
	The rest of the argument is now exactly the same as in the previous case.
	
	The sufficiency being established, we turn to the necessity and fix some $f \in L^{(1,\infty,b)_a}$. As for the condition
	\begin{equation*}
	\lim_{t \to 0_+} tb(t) f^{**}(t) = 0
	\end{equation*}
	which is common for both of the cases, we observe that if we choose, as in the proof of Theorem~\ref{TACN}, a sequence $E_k$ of sets such that \eqref{TACN1.5} holds and $\mu(E_k) = k^{-1}$, then we have (by using the arguments the from proof of Theorem~\ref{TACN}) that
	\begin{align*}
	\sup_{t \in (0, k^{-1})} tb(t) f^{**}(t) &\leq \sup_{t \in (0, \infty)} b(t) \int_0^{t} f^*\chi_{[0, k^{-1})} \: d\lambda = \lVert f \chi_{E_k} \rVert_{(1, \infty, b)} \to 0 &\textup{as } k \to \infty.
	\end{align*}
	
	It remains to show the necessity of
	\begin{equation} \label{T2ACN3}
	\lim_{t \to \infty} tb(t) f^{**}(t) = 0
	\end{equation}
	under the assumption $\lim_{t \to \infty} b(t) = 0$. To this end, consider the functions $f_n$ given by
	\begin{equation*}
	f_n = \min \left \{ \lvert f \rvert, \, f^*(n) \right \}.
	\end{equation*}
	Since we have by \ref{SV3} that
	\begin{equation*}
	\lim_{t \to \infty} tb(t) = \infty,
	\end{equation*}
	it follows that it holds for every $f \in L^{(1,\infty,b)}$ that
	\begin{equation*}
	\lim_{t \to \infty} f^*(t) \leq \lim_{t \to \infty} f^{**}(t) = 0.
	\end{equation*}
	The functions $f_n$ thus converge to zero $\mu$-a.e.~and we may use Proposition~\ref{PropACN} together with \eqref{T2ACN2} (with $t_0=n$) to get
	\begin{align*}
	\sup_{t \in [n, \infty)} b(t)\int_{n}^{t} f^* \: d\lambda &\leq  \lVert f_n \rVert_{(1, \infty, b)} \to 0 &\textup{as } n \to \infty.
	\end{align*}
	It now trivially follows that
	\begin{align*}
	\limsup_{t \to \infty} b(t)\int_{n}^{t} f^* \: d\lambda & \to 0 &\textup{as } n \to \infty.
	\end{align*}
	Finally, we may use our assumption $\lim_{t \to \infty} b(t) = 0$ to compute
	\begin{equation*}
	\limsup_{t \to \infty} tb(t)f^{**}(t) \leq \limsup_{t \to \infty} b(t)\int_{n}^{t} f^* \: d\lambda + \limsup_{t \to \infty} b(t)\int_{0}^{n} f^* \: d\lambda = \limsup_{t \to \infty} b(t)\int_{n}^{t} f^* \: d\lambda.
	\end{equation*}
	Since the left-hand side does not depend on $n$ (and since the remaining estimate is trivial), we conclude that $f$ indeed satisfies \eqref{T2ACN3}.
	
	It remains to show that if the space $L^{(1,\infty,b)}$ is non-trivial then it contains a function $f$ for which 
	\begin{equation*}
	\lim_{t \to 0^+} tb(t) f^{**}(t) > 0,
	\end{equation*}
	which is an easy exercise that we leave to the reader.
\end{proof}

We would like to point out that our proof of the necessity of \eqref{T2ACN3} uses some ideas inspired by \cite{SlavikovaUnp}.

\subsection{Associate spaces of Lorentz--Karamata spaces} \label{SSLKAS}
In this section we provide a full description of the associate spaces of Lorentz--Karamata spaces. We first single out those $p$ for which $(L^{p,q,b})' = \{ 0 \}$ independently of $b$.

\begin{proposition} \label{PAS}
	Let $p, q, b$ be as in Definition~\ref{DLK}. If $p \in (0, 1)$ then $(L^{p,q,b})' = \{ 0 \}$.
\end{proposition}

\begin{proof}
	Denote $r= \min \{ 1, \, \mu(R) \}$ and let $g_0$ be such a function from $M(R, \mu)$ that satisfies $g_0^*(t) = t^{-1} \chi_{(0, r)}(t)$. That such a function exists follows for example from \cite[Chapter~2, Corollary~7.8]{BennettSharpley88}. Then $\lVert g_0 \rVert_{p, q, b} < \infty$, as can be easily verified using \ref{LEFF}. We may assume that $\lVert g_0 \rVert_{p, q, b} \leq 1$, otherwise we would multiply $g_0$ by an appropriate number. On the other hand, $g_0^*$ is certainly not integrable on any deleted neighbourhood of zero and thus for any function $f \in M(R, \mu)$ other than the zero function we can use the resonance of $(R, \mu)$ to obtain
	\begin{equation*}
		\lVert f \rVert_{(L^{p,q,b})'} = \sup_{\substack{g \in M \\ \lVert g \rVert \leq 1}} \int_R \lvert fg \rvert \: d\mu \geq \sup_{\substack{g \in M \\ g^* = g_0^*}} \int_R \lvert fg \rvert \: d\mu = \int_0^{\infty} f^*g_0^*  = \infty.
	\end{equation*}
\end{proof}

This also shows that if $p<1$, then $\lVert \cdot \rVert_{p, q, b}$ does not satisfy \ref{P5}, because its associate functional does not satisfy the condition \ref{P4}. This result is already contained in Corollary~\ref{CP5b}, but the construction above proves it in a much more direct way.

For the proofs of the theorems bellow, we use a combination of approaches. In most cases we opted for comparatively direct and elementary methods, often inspired by the approach used by Opic and Pick in \cite{OpicPick99} for describing associate spaces of generalised Lorentz-Zygmund spaces. On the other hand, there are few limiting cases where those methods fail to provide a full solution, in the sense that they work only with some additional assumptions. In those cases we apply the abstract theory of classical Lorentz spaces to obtain the results in full generality. We recognise that this theory could also be applied to the remaining cases, but we believe that it is meaningful to show that this theory is in many cases not necessary.

\begin{theorem}  \label{TAS}
	Let $p, q, b$ be as in Definition~\ref{DLK} and suppose $q \in (1, \infty)$ and $p \in [1, \infty)$. Then
	\begin{align}
	(L^{p,q,b})' = L^{(p',q',b^{-1})} \label{TAS1}
	\end{align}
	up to equivalence of the defining functionals.
\end{theorem}

\begin{proof}
	We begin by proving that 
	\begin{equation*}
	(L^{p, q, b})'~\hookrightarrow~L^{(p', q', b^{-1})}
	\end{equation*}
	or, equivalently, that it holds for all $f \in M$ that
	\begin{equation}
	\lVert f \rVert_{(p', q', b^{-1})}~\lesssim~\lVert f \rVert_{(L_{p, q, b})'}. \label{TAS2}
	\end{equation}
	To this end, let us fix some arbitrary $f \in M$ and denote, for $t \in (0, \infty)$,
	\begin{align*}
	\varrho(t) &= (f^{**}(t))^{q'-1} t^{\frac{q'}{p'} -1}b^{-q'}(t) \\
	g(t) &= \int_t^{\infty} \frac{\varrho(s)}{s} \: ds. 
	\end{align*}
	Since $g$ is non-increasing we can find a function $ \hat{g} \in M(R, \mu)$ such that $\hat{g}^* = g$. We may now employ classical Fubini's theorem and the H{\"o}lder inequality \eqref{THAS} to compute
	\begin{equation*}
	\begin{split}
		\lVert f \rVert_{(p', q', b^{-1})}^{q'} &= \int_0^{\infty} \frac{\varrho(t)}{t} \int_0^{t}   f^*(s) \: ds \: dt \\
		&= \int_0^{\infty} f^*(s) g(s) \: ds \\
		&\leq \lVert \hat{g} \rVert_{p, q, b} \cdot \lVert f \rVert_{(L_{p, q, b})'}.
	\end{split}
	\end{equation*}
	Thus, it suffices to show that
	\begin{equation} 
		\lVert \hat{g} \rVert_{p, q, b} \lesssim \lVert f \rVert_{(p', q', b^{-1})}^{q'-1}.  \label{TAS3}
	\end{equation}
	 Rewriting this expression in a more explicit form and using the substitution
	\begin{equation*}
	h(t) = (f^{**}(t))^{q'-1} t^{\frac{q'}{p'} -2} b^{-q'}(t) 
	\end{equation*}
	we obtain that \eqref{TAS3} is equivalent to a weighted Hardy inequality
	\begin{equation*}
		\left \lVert t^{\frac{1}{p} - \frac{1}{q}} b(t) \int_t^{\infty} h(s) \: ds \right \lVert_q \lesssim \left \lVert t^{\frac{1}{p} + \frac{1}{q'}} b(t) h(t) \right \rVert_q, 
	\end{equation*}
	which, as it is characterised for example in \cite[Section~1.3, Theorem~3]{Maz'ya11}, holds if and only if the following condition holds:
	\begin{equation*}
	\sup_{r \in (0, \infty)} \bigg (\int_0^{r}t^{\frac{q}{p}-1} b^q(t) \:dt \bigg )^{\frac{1}{q}} \cdot \bigg (\int_r^{\infty} t^{-\frac{q'}{p}-1} b^{-q'}(t) \:dt \bigg )^{\frac{1}{q'}} < \infty.
	\end{equation*}
	Remembering \ref{LEFF} and that we assume $p < \infty$, we can easily check that this condition is satisfied. 
	
	Since $f \in M$ was arbitrary, we have thus proven \eqref{TAS2}. We now turn our attention to the embedding
	\begin{equation*}
	L^{(p', q', b^{-1})}~\hookrightarrow~(L^{p, q, b})',
	\end{equation*}
	which holds if and only if
	\begin{equation}
	\lVert f \rVert_{(L_{p, q, b})'}~\lesssim~\lVert f \rVert_{(p', q', b^{-1})} \label{TAS8}
	\end{equation}
	for all $f \in M$.
	It follows from the Hardy-Littlewood inequality \eqref{THLI} that in order to get \eqref{TAS8} it is sufficient to show that 
	\begin{equation*}
	\int_0^{\infty} f^*(t)g^*(t) dt \lesssim \lVert g \rVert_{p, q, b} \cdot \lVert f \rVert_{(p', q', b^{-1})}
	\end{equation*}
	holds for all $f, g \in M$. Now, to get this estimate we use the weighted inequality of Sawyer, proved in \cite[Theorem~1]{Sawyer90}, which for a weight $v$ gives us the following:
	\begin{equation}
	\begin{split}
	\int_0^{\infty} f^*(t)g^*(t) dt \leq& \bigg ( \int_0^{\infty} (g^*(t))^q v(t) \: dt \bigg )^{\frac{1}{q}} \\
	&\cdot \Bigg [ \bigg ( \int_0^{\infty} (f^{**}(t))^{q'} \tilde{v}(t) \: dt \bigg )^{\frac{1}{q'}} + \frac{\int_0^{\infty} f^*(t) \: dt}{\big ( \int_0^{\infty} v(t) \: dt \big )^{\frac{1}{q}}} \Bigg ], \label{TAS10}
	\end{split}
	\end{equation}
	where $\tilde{v}(s)$ is given for $t \in (0, \infty)$ by
	\begin{equation}
	\tilde{v}(t) = \frac{t^{q'}v(t)}{\bigg ( \int_0^{t} v(s) \: ds \bigg )^{q'}}. \label{TAS11}
	\end{equation}
	For our purpose, we put, for $t \in (0, \infty)$,
	\begin{equation}
	v(t) = t^{\frac{q}{p} - 1} b^q(t) \label{TAS12}
	\end{equation}
	and see immediately that $v$ is not integrable on $(0, \infty)$, as follows from \ref{SV4}, and thus by Sawyer's convention $``\frac{\infty}{\infty} = 0 " $ we get that the second summand at the second line of \eqref{TAS10} is zero. Therefore, to obtain \eqref{TAS8} it suffices to show that
	\begin{equation*}
	\tilde{v}(t) \lesssim t^{\frac{q'}{p'}-1} b^{-q'}(t)
	\end{equation*}
	holds for all $t \in (0, \infty)$, which follows from \ref{LEFF}.
\end{proof}

The method employed in the proof of Theorem~\ref{TAS} could be modified to solve even the subcase $p= \infty$, but it would require some additional technical assumptions. In order to solve this subcase in full generality, we turn to the abstract theory of classical Lorentz spaces contained in the paper \cite{GogatishviliPick03}.

\begin{theorem} \label{T1AS}
	Let $p, q, b$ be as in Definition~\ref{DLK} and suppose $q \in (1, \infty)$, $p = \infty$, and $\lVert t^{- \frac{1}{q}} b(t) \chi_{(0, 1)}(t) \rVert_q < \infty$. Then
	\begin{align}
		(L^{\infty,q,b})' = L^{(1,q',a)} \label{T1AS1}
	\end{align}
	up to equivalence of the defining functionals, where $a$ is given by
	\begin{align}  \label{T1AS1a}
		a(t) &= \left ( \int_0^{t} s^{-1} b^q(s) \: ds \right )^{-1} b^{q-1}(t)  &\textup{for } t \in (0, \infty).
	\end{align}
\end{theorem}

Note that the condition $\lVert t^{- \frac{1}{q}} b(t) \chi_{(0, 1)}(t) \rVert_q < \infty$ is natural, since by Proposition~\ref{LQ} it is equivalent to the non-triviality of $L^{\infty,q,b}$.

\begin{proof}
	It follows directly from the corresponding definitions that a Lorentz--Karamata space $L^{\infty,q,b}$ is, for our choices of $q$, also the classical Lorentz space $\Gamma^q(v)$ where $v$ is given by
	\begin{align} \label{T1AS2}
		v(t) &= t^{-1} b^q(t) &\textup{for } t \in (0, \infty).
	\end{align}
	It thus follows from \cite[Theorem~6.2]{GogatishviliPick03} that the associate space is the classical Lorentz space $\Gamma^{q'}(w)$ where $w$ is given by
	\begin{align*}
		w(t) &= \frac{t^{q'+q-1} \left( \int_0^{t} v(s) \: ds \right) \left( \int_t^{\infty} s^{-q} v(s) \: ds \right)  }{\left(  \left( \int_0^{t} v(s) \: ds \right) + t^q \left( \int_t^{\infty} s^{-q} v(s) \: ds  \right)  \right)^{q'+1}} &\textup{for } t \in (0, \infty).
	\end{align*}
	By plugging in \eqref{T1AS2} and by using \ref{LEFF} we may now simplify this expression to
	\begin{align} \label{T1AS3}
		w(t) &\approx \frac{t^{q'-1} \left( \int_0^{t} s^{-1} b^q(s) \: ds \right)b^q(t)  }{\left(  \left( \int_0^{t} s^{-1} b^q(s) \: ds \right) + b^q(t)  \right)^{q'+1}} &\textup{for } t \in (0, \infty).
	\end{align}
	Since we know from \ref{LTb} that
	\begin{align*}
		\int_0^{t} s^{-1} b^q(s) \: ds &\leq \left( \int_0^{t} s^{-1} b^q(s) \: ds \right) + b^q(t) \lesssim \int_0^{t} s^{-1} b^q(s) \: ds &\textup{for } t \in (0, \infty)
	\end{align*}
	we may further simplify \eqref{T1AS3} to get
	\begin{align*}
		w(t) &\approx \frac{t^{q'-1} \left( \int_0^{t} s^{-1} b^q(s) \: ds \right)b^q(t)  }{\left( \int_0^{t} s^{-1} b^q(s) \: ds\right)^{q'+1}} = t^{q'-1} \left( \int_0^{t} s^{-1} b^q(s) \: ds \right)^{-q'}b^q(t) &\textup{for } t \in (0, \infty).
	\end{align*}
	Finally, we obtain from the corresponding definitions that the classical Lorentz space $\Gamma^{q'}(w)$ is, for this $w$, exactly the Lorentz--Karamata space $L^{(1, q', a)}$, where $a$ is given by \eqref{T1AS1a}.
\end{proof}

\begin{theorem} \label{T2AS}
	Let $p, q, b$ be as in Definition~\ref{DLK} and suppose $q \in (0, 1]$. Then the associate spaces of $L^{p,q,b}$, up to equivalence of the defining functionals, can be described as follows:
	\begin{enumerate}
		\item
		If $p \in [1, \infty)$ then  
		\begin{equation*} \label{T2AS1}
		(L^{p,q,b})' = L^{(p',\infty,b^{-1})}.
		\end{equation*}
		\item
		If $p=\infty$ and $b$ satisfies $\lVert t^{-\frac{1}{q}} b(t) \chi_{(0, 1)}(t) \rVert_q < \infty$ then
		\begin{equation*} \label{T2AS2}
		(L^{\infty,q,b})' = L^{(1,\infty,a)}
		\end{equation*}
		where $a$ is given by
		\begin{align} \label{T2AS2.1}
		a(t) &= \left( \int_0^{t} s^{-1} b^q(s) \: ds \right) ^{-\frac{1}{q}} &\textup{for } t \in (0, \infty).
		\end{align}
	\end{enumerate}
\end{theorem}

Note that the condition $\lVert t^{-\frac{1}{q}} b(t) \chi_{(0, 1)}(t) \rVert_q < \infty$ is again natural, with the same reasoning as in the previous theorem. Furthermore, we observe that $a$ is non-increasing and therefore $\lVert a(t) \chi_{(1, \infty)}(t) \rVert_{\infty} < \infty$.

\begin{proof}
	Suppose first $q = 1$ and $p \in (1, \infty)$. Then we have by Theorem~\ref{TFFA*} that $L^{p,1,b}=\Lambda_{\varphi}$, where
	\begin{align*}
		\varphi(t) &= t^{\frac{1}{p}}b(t) &\textup{for } t \in (0, \infty)
	\end{align*}
	is the fundamental function of $L^{p,1,b}$. From this it follows (see Section~\ref{SSFF}) that
	\begin{equation}
	(L^{p,1,b})' = M_{\bar{\varphi}(t)}, \label{T2AS2.5}
	\end{equation}
	where $\bar{\varphi}$ is defined on $(0, \infty)$ by
	\begin{equation} \label{T2AS2.6}
	\bar{\varphi}(t) = \frac{t}{\varphi(t)}
	\end{equation}
	that is, it holds for all $f \in M$ that 
	\begin{equation*}
	\lVert f \rVert_{(L^{p,1,b})'} = \sup_{t \in (0, \infty)} \frac{t}{\int_0^{t} s^{\frac{1}{p}-1} b(s) \: ds} f^{**}(t).
	\end{equation*}
	As follows from \ref{LEFF}, the expression on the right-hand side is for $p \in [1, \infty)$ equivalent to $\lVert f \rVert_{(p',\infty,b^{-1})}$. 
	
	We now turn our attention to the case $q \in (0, 1)$ and $p \in (1, \infty)$. It follows from Proposition \ref{PELK} that $L^{p,q,b} \hookrightarrow L^{p,1,b}$ which implies $(L^{p,1,b})' \hookrightarrow (L^{p,q,b})'$. This, when put together with \eqref{T2AS2.5}, yields
	\begin{equation}
	M_{\bar{\varphi}(t)} \hookrightarrow (L^{p,q,b})'. \label{T2AS3}
	\end{equation}	
	Now, if follows from Proposition \ref{LQ} that $\lVert \cdot \rVert_{p, q, b}$ satisfies \ref{P4}, since we have $p < \infty$. By Corollary~\ref{CP5b} it also satisfies \ref{P5}. Hence, $\lVert \cdot \rVert_{p, q, b}$ satisfies the conditions from Theorem~\ref{TFA} and we thus have that $(L^{p,q,b})'$ is a Banach function space. The remark after Definition~\ref{DMES} therefore implies that 
	\begin{equation}
	(L^{p,q,b})' \hookrightarrow M_{\bar{\psi}}, \label{T2AS4}
	\end{equation}
	where $\psi$ denotes the fundamental function of the r.i.\ quasi-Banach function space $L^{p,q,b}$ and $\bar{\psi}$ denotes the fundamental function of $(L^{p,1,b})'$, which implies $\bar{\psi}(t) \geq \frac{t}{\psi(t)}$ on $(0, \infty)$. Thus, if we had 
	\begin{equation}	
	\varphi \approx \psi, \label{T2AS5}
	\end{equation}
	then it would follow that $\bar{\varphi}(t) \approx\frac{t}{\psi(t)} \leq \bar{\psi}(t)$ on $(0, \infty)$ and consequently 
	\begin{equation}	
	M_{\bar{\psi}} \hookrightarrow M_{\bar{\varphi}}. \label{T2AS6}
	\end{equation}
	We could then combine \eqref{T2AS6} with \eqref{T2AS3} and \eqref{T2AS4} to get $(L^{p,q,b})' = M_{\bar{\varphi}(t)}$ and the conclusion of the theorem would follow. So it remains only to verify \eqref{T2AS5}, which follows from Theorem~\ref{TFFA*}.
	
	Let us now turn to the case when $p = \infty$. In this case, the fundamental function $\varphi$ of $L^{\infty,q,b}$ is given by
	\begin{align*}
		\varphi(t) &= \left(  \int_0^{t} s^{-1}b^q(s) \: ds \right)^{\frac{1}{q}} &\textup{for } t \in (0, \infty), 
	\end{align*}
	as follows from Theorem~\ref{TFFE}, and the corresponding Lorentz endpoint space is characterised in the same theorem as the Lorentz--Karamata space $L^{\infty,1,c}$, where $c$ is the s.v.~function given by the formula
	\begin{align*}
		c(t) &= \left(  \int_0^{t} s^{-1}b^q(s) \: ds \right)^{\frac{1}{q} - 1}b^q(t) &\textup{for } t \in (0, \infty).
	\end{align*}
	
	Now, it follows from Theorem~\ref{TELK} that $L^{\infty, q, b} \hookrightarrow L^{\infty, 1, c}$. Indeed, the sufficient condition holds as $q<1$ and
	\begin{equation*}
		\sup_{r \in (0, \infty)} \frac{\int_0^{r} t^{-1} c(t) \: dt }{\left( \int_0^{r} t^{-1} b^q(t) \: dt \right) ^{\frac{1}{q}}} = \sup_{r \in (0, \infty)} \frac{\left( \int_0^{r} t^{-1} b^q(t) \: dt \right) ^{\frac{1}{q}}}{\left( \int_0^{r} t^{-1} b^q(t) \: dt \right) ^{\frac{1}{q}}} = 1,
	\end{equation*}
	since $r\mapsto r^{-1}c(r)$ is (up to a set of measure zero) the derivative of the locally absolutely continuous function
	\begin{align*}
		r &\mapsto \left( \int_0^{r} t^{-1} b^q(t) \: dt \right) ^{\frac{1}{q}} &\textup{for } r \in (0, \infty).
	\end{align*}
	We are thus in a similar situation as in the previous case, i.e.~we know that the space in question is embedded into its Lorentz endpoint space (with the difference being that now the formula for $\varphi$ depends on $q$ and therefore we get different Lorentz endpoint spaces for different values of $q$). We can thus use the same argument as before to obtain that
	\begin{equation*}
		(  L^{\infty, q, b} )' = ( \Lambda_{\varphi} )' = M_{\bar{\varphi}}
	\end{equation*}
	where $\bar{\varphi}$ is as in \eqref{T2AS2.6}. Put explicitly,
	\begin{align*}
		\bar{\varphi}(t) &= t \left( \int_0^{t} s^{-1} b^q(s) \: ds \right) ^{-\frac{1}{q}} = t a(t) &\textup{for } t \in (0, \infty),
	\end{align*}
	where $a$ is the function given by \eqref{T2AS2.1}, and thus, by comparing the definitions,
	\begin{equation*}
		M_{\bar{\varphi}} = L^{(1,\infty,a)},
	\end{equation*}
	as desired.
	
	Finally, we consider the case $p=1$. Here we turn to the abstract theory of the classical Lorentz spaces contained in \cite[Theorem~9.1]{CarroPick00}. For our choice of parameters $p,q,b$ we have that
	$L^{1,q,b} = \Lambda^q(v)$ with
	\begin{align*}
		v(t) &= t^{q-1}b^q(t) &\textup{for } t \in (0, \infty).
	\end{align*}
	The case (i) of the cited theorem thus yields that it holds for all $f \in M$ that
	\begin{equation*}
		\lVert f \rVert_{(L^{1,q,b})'} = \sup_{t \in (0, \infty)} f^{**}(t) \frac{t}{\left(  \int_0^{t} t^{q-1}b^q(t) \: dt \right)^{\frac{1}{q}}} = \sup_{t \in (0, \infty)} f^{**}(t) b^{-1}(t) = \lVert f \rVert_{L^{(\infty, \infty, b^{-1})}},
	\end{equation*}
	where the second equality is due to \ref{LEFF}. This concludes the proof.
\end{proof}

The subcase $p=1$ in the previous theorem could be solved by the same method as the rest of the theorem if we assumed that $b$ is equivalent to a non-increasing function. The abstract theory of classical Lorentz spaces is used because it allows us to remove this additional assumption.

\begin{theorem} \label{T3AS}
	Let $p, q, b$ be as in Definition~\ref{DLK} and suppose $q = \infty$ and $p \in [1, \infty]$. If $p=\infty$ then we assume that $b$ is non-decreasing. Then
	\begin{equation*}
	(L^{p,\infty,b})' = L^{p',1,b^{-1}}
	\end{equation*}
	up to equivalence of the defining functionals.
\end{theorem}

We would like to remind the reader that the assumption that $b$ is non-decreasing in the case when $p=q=\infty$ causes no loss of generality thanks to Proposition~\ref{PbND}.

\begin{proof}
	Suppose first that $p \in (1, \infty)$. Then by our conditions on $p, q, b$ we have, by the virtue of Corollary~\ref{CEQNa} and \ref{LEFF} with ($\alpha = \frac{1}{p'}$) that it holds for every $f \in M$ that
	\begin{equation*}
	\lVert f \rVert_{p, \infty, b} \approx \sup_{t \in (0, \infty)} \frac{t}{\int_0^{t} s^{\frac{1}{p'}-1} b^{-1}(s) \: ds} f^{**}(t). \label{T3AS2}
	\end{equation*}
	Therefore, we have from definition that $L^{p,q,b}=M_{\varphi}$, where 
	\begin{align*}
	\varphi(t) &= \frac{t}{\int_0^{t} s^{\frac{1}{p'}-1} b^{-1}(s) \: ds} &\textup{for } t \in (0, \infty),
	\end{align*}
	from which it follows (see Section~\ref{SSFF}) that 
	\begin{equation*}
	(L^{p,q,b})' = \Lambda_{\bar{\varphi}(t)},
	\end{equation*}
	where
	\begin{align*}
	\bar{\varphi}(t) &= \frac{t}{\varphi(t)} = \int_0^{t} s^{\frac{1}{p'}-1} b^{-1}(s) \: ds &\textup{for } t \in (0, \infty).
	\end{align*} 
	But since $p' > 1$, as follows from our assumptions, we already know from the proof of Theorem \ref{T2AS} that 
	\begin{equation*}
	\Lambda_{\bar{\varphi}(t)} = L^{p',1,b^{-1}}.
	\end{equation*}
	
	Now let us turn our attention to the remaining cases. We suppose that either $p=1$ or $p=\infty$, in which case we assume that $b$ is non-decreasing. The embedding 
	\begin{equation}
	L^{p',1,b^{-1}} \hookrightarrow (L^{p,\infty,b})' \label{T3AS4}
	\end{equation}
	follows from the classical H{\"o}lder inequality for the Lebesgue spaces over $(0, \infty)$, since it yields  
	\begin{equation*}
	\begin{split}
	\int_0^{\infty} f^*(t) g^*(t) \: dt &\leq \left \lVert t^{\frac{1}{p} }b(t)g^*(t) \right \rVert_{\infty} \cdot \left \lVert t^{\frac{1}{p'} -1}b^{-1}(t)f^*(t) \right \rVert_1 \\
	&= \lVert g \rVert_{p, \infty, b} \cdot \lVert f \rVert_{p', 1, b^{-1}}.
	\end{split}
	\end{equation*}
	The embedding \eqref{T3AS4} then follows by combining this estimate with the Hardy--Littlewood inequality~\eqref{THLI} and taking the supremum over all $g \in M$ such that $ \lVert g \rVert_{p, \infty, b} \leq 1$.
	
	For the converse embedding
	\begin{equation}
	(L^{p,\infty,b})' \hookrightarrow L^{p',1,b^{-1}}, \label{T3AS5}
	\end{equation}
	define
	\begin{align*}
	g(t) &= t^{\frac{1}{p'} -1} b^{-1}(t) &\textup{for } t \in (0, \infty).
	\end{align*}
	In both of the cases, our assumptions guarantee that $g \approx g^*$, from which it follows that any function $\tilde{g} \in M(R, \mu)$ such that $\tilde{g}^* = g^*$ satisfies $\lVert \tilde{g} \rVert_{p, \infty, b} < \infty$. Furthermore, we may use the fact that $(R, \mu)$ is resonant and the H{\"o}lder inequality \eqref{THAS} to get
	\begin{equation*}
	\begin{split}
	\lVert f \rVert_{p', 1, b^{-1}} &\approx \sup_{\substack{\tilde{g} \in M \\ \tilde{g}^*=g^*}} \int_R \lvert \tilde{g} f \rvert \\
	&\leq \lVert \tilde{g_0} \rVert_{p, \infty, b} \cdot \lVert f \rVert_{(L^{p',\infty,b})'}
	\end{split}
	\end{equation*}
	where $\tilde{g_0}$ is any fixed function from $M(R, \mu)$ such that $\tilde{g_0}^* = g^*$.  The embedding \eqref{T3AS5} follows.
\end{proof}

We conclude this section by presenting the associate spaces of the spaces $L^{(1,q,b)}$. We note that the choice $p=1$ is the only meaningful one; indeed for $p<1$ it holds that $L^{(p,q,b)} = \{0\}$ (see Proposition~\ref{PP4}) while for $p>1$ we have that $L^{(p,q,b)} = L^{p,q,b}$ (see Corollary~\ref{CEQNa}) and the associate spaces are thus described in the theorems above. Furthermore, it also follows from Proposition~\ref{PP4} that we only need to consider the case when $\lVert t^{- \frac{1}{q}} b(t) \chi_{(1, \infty)}(t) \rVert_q < \infty$.

\begin{theorem} \label{T4AS}
	Let $p, q, b$ be as in Definition~\ref{DLK} and consider the case when $p = 1$ and $\lVert t^{- \frac{1}{q}} b(t) \chi_{(1, \infty)}(t) \rVert_q < \infty$. Then the associate space of $L^{(1,q,b)}$, up to equivalence of the defining functionals, can be described as follows:
	\begin{enumerate}
		\item \label{T4ASi} If $q \in (0,1]$ then $(L^{(1,q,b)})' = L^{(\infty, \infty, a)}$ where a is given by
		\begin{align} \label{T4AS01}
			a(t) &= \left( \int_t^{\infty} s^{-1} b^q(s) \: ds \right)^{-\frac{1}{q}} &\textup{for } t \in (0, \infty).
		\end{align}
		\item \label{T4ASii} If $q \in (1, \infty)$ then $(L^{(1,q,b)})' = L^{(\infty, q', a)}$ where a is given by
		\begin{align} \label{T4AS02}
			a(t) &= \left( \int_t^{\infty} s^{-1} b^q(s) \: ds \right)^{-1}b^{q-1}(t) &\textup{for } t \in (0, \infty).
		\end{align}
		\item \label{T4ASiii} If $q=\infty$ and $b$ is assumed to be absolutely continuous and non-increasing then 
		\begin{align}
			(L^{(1,\infty,b)})' &= \label{T4AS04}
			\begin{cases}
				\Lambda^1((b^{-1})') & \text{in the case when } \lim_{t \to 0^+} b^{-1}(t) = 0, \\
				\Lambda^1((b^{-1})') \cap L^{\infty} & \text{in the case when } \lim_{t \to 0^+} b^{-1}(t) > 0,
			\end{cases}
		\end{align}
		where $\Lambda^1((b^{-1})')$ is a classical Lorentz space.
	\end{enumerate}
\end{theorem}

We note that the extra assumptions in the case \ref{T4ASiii} come at no loss of generality (due to Theorem~\ref{TSmSV} and Proposition~\ref{PbNI}, used in this order since the transformation $b \mapsto \hat{b}_{\infty}$ preserves local Lipschitz continuity) and ensure that $(b^{-1})'$ exists and is non-negative $\lambda$-a.e.

\begin{proof}
	For the cases \ref{T4ASi} and \ref{T4ASii} we employ \cite[Theorem~6.2]{GogatishviliPick03}. It follows directly from the definition that in those cases the Lorentz--Karamata space $L^{(1,q,b)}$ is also the classical Lorentz space $\Gamma^q(v)$ where $v$ is given by
	\begin{align} \label{T4AS03}
		v(t) &= t^{q-1} b^q(t) &\textup{for } t \in (0, \infty).
	\end{align}
	
	In the case \ref{T4ASi} it then follows from the said theorem by a simple calculation using \ref{LEFF} that the associate norm is given by
	\begin{equation} \label{T4AS1}
		\lVert f \rVert_{(L^{(1,q,b)})'} \approx \sup_{t \in (0, \infty) } f^{**}(t) \left( b^q(t) + \int_t^{\infty} s^{-1} b^q(s) \: ds \right)^{-\frac{1}{q}}.
	\end{equation}
	Now, we know from \ref{LHb} that
	\begin{align*}
		\int_t^{\infty} s^{-1} b^q(s) \: ds &\leq \int_t^{\infty} s^{-1} b^q(s) \: ds + b^q(t) \lesssim \int_t^{\infty} s^{-1} b^q(s) \: ds &\textup{for } t \in (0, \infty)
	\end{align*}
	and thus \eqref{T4AS1} simplifies to
	\begin{equation*}
		\lVert f \rVert_{(L^{(1,q,b)})'} \approx \sup_{t \in (0, \infty) } f^{**}(t) \left(\int_t^{\infty} s^{-1} b^q(s) \: ds \right)^{-\frac{1}{q}} = \lVert f \rVert_{L^{(\infty,\infty,a)}},
	\end{equation*}
	where $a$ is given by \eqref{T4AS01}.
	
	In the case \ref{T4ASii} it follows from the same theorem that the associate space is the classical Lorentz space $\Gamma^{q'}(w)$ where $w$ is given by
	\begin{align*}
		w(t) &= \frac{t^{q'+q-1} \left( \int_0^{t} v(s) \: ds \right) \left( \int_t^{\infty} s^{-q} v(s) \: ds \right)  }{\left(  \left( \int_0^{t} v(s) \: ds \right) + t^q \left( \int_t^{\infty} s^{-q} v(s) \: ds  \right)  \right)^{q'+1}} &\textup{for } t \in (0, \infty).
	\end{align*}
	We may now plug in \eqref{T4AS03} and use \ref{LEFF}, the same argument as in the previous step, and some elementary calculations to simplify this expression to
	\begin{align*}
		w(t) &\approx t^{-1} \left( \left( \int_t^{\infty} s^{-1} b^q(s) \: ds \right)^{-1}b^{q-1}(t)  \right)^{q'} &\textup{for } t \in (0, \infty).
	\end{align*}
	Hence, it follows from the corresponding definitions that $(L^{(1,q,b)})' = L^{(\infty, q', a)}$ where $a$ is given by \eqref{T4AS02}.

	As for the remaining case \ref{T4ASiii}, it follows from Theorem~\ref{TFF(1)} that $L^{(1, \infty, b)} = M_{\varphi}$ where $\varphi(t) = t\hat{b}_{\infty}(t) = tb(t)$ for $t \in (0, \infty)$ (the last equality is due to our assumption on monotonicity of $b$). It thus follows (see Section~\ref{SecFF}) that $(L^{(1,\infty,b)})' = \Lambda_{\bar{\varphi}(t)}$ where $\bar{\varphi}(t) = b^{-1}(t)$ for $t \in (0,\infty)$. Since $b^{-1}$ is non-decreasing and absolutely continuous, we know from Theorem~\ref{TFFE} that the corresponding Lorentz endpoint space can be described precisely as asserted in \eqref{T4AS04}.
\end{proof}

We would like to note that in the case \ref{T4ASiii} of the previous theorem it is also possible to deduce the characterisation using \cite[Corollary~1.9]{GogatishviliPick06}. However, as the above presented elementary proof shows, such an advanced machinery is not necessary for our purposes.

\subsection{Lorentz--Karamata spaces as Banach function spaces} \label{SSBFS}
In this section we study when a given Lorentz--Karamata space is equivalent to a Banach function space. We start with the spaces $L^{p,q,b}$, for which we provide a full characterisation in the following theorem. Most of the cases follow from the results already proved in the paper; for the remaining cases we turn to the more abstract theory of classical Lorentz spaces, relevant aspects of which have been covered in \cite{CarroSoria96}, \cite{CarroRaposo07}, \cite{Sawyer90} and even quite recently in \cite{GogatishviliSoudsky14}.

\begin{theorem} \label{TLKBFS}
	Let $p,q,b$ be as in Definition~\ref{DLK}. The space $L^{p,q,b}$ can be equipped with a Banach function norm equivalent to $\lVert \cdot \rVert_{p, q, b}$ if and only if $q \in [1, \infty]$ and one of the following conditions holds:
	\begin{enumerate}
		\item \label{TLKBFS_1i}
		$p\in (1, \infty)$,
		\item \label{TLKBFS_1ii}
		$p=\infty$ and $\lVert t^{- \frac{1}{q}} b(t) \chi_{(0, 1)}(t) \rVert_q < \infty$,
		\item \label{TLKBFS_1iii}
		$p=1$, $q = 1$, and $b$ is equivalent to a non-increasing function.
	\end{enumerate}
\end{theorem}

\begin{proof}
	As for the sufficiency, parts~\ref{TLKBFS_1i} and \ref{TLKBFS_1ii} constitute the content of Corollary~\ref{CEQN} while the part~\ref{TLKBFS_1iii} follows the characterisation of normability of the classical Lorentz space $\Lambda^1$ obtained in \cite[Theorem~2.3]{CarroSoria96} and \ref{LEFF}. Alternatively, one can prove that \ref{TLKBFS_1iii} is sufficient by using Theorem~\ref{T3AS}, Corollary~\ref{CP5b}, Theorem~\ref{TFA}, and Theorem~\ref{TELK}.
	
	Let us now turn our attention to the necessity of those conditions. The case when $q \in (0,1)$ follows from \cite[Theorem~2.5.8]{CarroRaposo07}, the case when $p \in (0,1)$ from Proposition~\ref{PAS} (or Corollary~\ref{CP5b}), and the case when $p=\infty$ and $\lVert t^{- \frac{1}{q}} b(t) \chi_{(0, 1)}(t) \rVert_q = \infty$ from Proposition~\ref{LN}. Moreover, to cover all the cases when $p=1$ and $q \in (1, \infty]$ one only has to combine Corollary~\ref{CEQNa} with \cite[Theorem~4]{Sawyer90} (for $q \in (1, \infty)$) and \cite[Theorem~5.1]{GogatishviliSoudsky14} (for $q=\infty$).
	
	Finally, to obtain the necessity for the case when $p=q=1$ we turn again to \cite[Theorem~2.3]{CarroSoria96} which, when combined with \ref{LEFF}, provides as a necessary condition for normability that there must exist a constant $C > 0$ such that it holds for all $0 < s \leq r < \infty$ that
	\begin{equation} \label{TLKBFS1}
		b(r) \leq C b(s).
	\end{equation}
	We now assert that $b$ is not equivalent to any non-increasing function, as a special case of which we get that $b$ is not equivalent to $B$ defined for $s \in (0, \infty)$ by
	\begin{equation*}
		B(s) = \sup_{t \in [s, \infty)} b(t).
	\end{equation*}
	This property together with the fact that $B \geq b$ on $(0,\infty)$ allows us to find, for any given $C>0$, the required pair of numbers $r, s$ that violates \eqref{TLKBFS1}.
\end{proof}

In the second theorem of this section we provide a full characterisation for the spaces $L^{(p,q,b)}$. We treat only the case when $\mu(R) = \infty$, because in the case $\mu(R) < \infty$ the situation gets rather messy, as there are several competing definitions of those spaces and the properties depend on the version chosen.

\begin{theorem} \label{TLKBFS2}
	Let $p,q,b$ be as in Definition~\ref{DLK} and assume that $\mu(R) = \infty$. The space $L^{(p,q,b)}$ can be equipped with a Banach function norm equivalent to $\lVert \cdot \rVert_{(p, q, b)}$ if and only if $q \in [1, \infty]$ and one of the following conditions holds:
	\begin{enumerate}
		\item \label{TLKBFS2_1i}
		$p\in (1, \infty)$,
		\item \label{TLKBFS2_1ii}
		$p= 1$ and $\lVert t^{- \frac{1}{q}} b(t) \chi_{(1, \infty)}(t) \rVert_q < \infty$,
		\item \label{TLKBFS2_1iii}
		$p= \infty$ and $\lVert t^{- \frac{1}{q}} b(t) \chi_{(0, 1)}(t) \rVert_q < \infty$.
	\end{enumerate}

	Moreover, if those conditions are satisfied then $\lVert \cdot \rVert_{(p, q, b)}$ is itself a Banach function norm.
\end{theorem}

\begin{proof}
	The sufficiency of the presented condition for $\lVert \cdot \rVert_{(p, q, b)}$ being a Banach function norm has already been established in Proposition~\ref{LN}. Furthermore, it follows from the same Proposition that if none of the conditions \ref{TLKBFS2_1i}, \ref{TLKBFS2_1ii}, and \ref{TLKBFS2_1iii} is satisfied, then the space is trivial and thus its defining functional cannot be equivalent to any Banach function norm. Thus it remains only to prove the necessity of $q \in [1, \infty]$.
	
	The case when $p>1$ follows from Theorem~\ref{TLKBFS}, since we know from Corollary~\ref{CEQNa} that under this assumption $L^{(p,q,b)} = L^{p,q,b}$ with equivalent quasinorms. Thence it remains to consider only the case $p=1$. In this case, we know that $L^{(1,q,b)} = \Gamma^q(v)$, where $\Gamma^q(v)$ is the classical Lorentz space given defined by the weight $v$ which is given by
	\begin{align*}
		v(t) &= t^{q-1} b^{q}(t) &\textup{for } t \in (0, \infty).
	\end{align*}
	It now follows from \cite[Theorem~1]{Soudsky16} that a necessary condition for $\Gamma^q(v)$ with $q<1$ to be normable is that $v$ is integrable on $(0, \infty)$. Since this is not true, as follows from \ref{SV4}, we conclude that $\lVert \cdot \rVert_{(p, q, b)}$ cannot be equivalent to a norm, much less to a Banach function norm.
\end{proof}

\bibliographystyle{dabbrv}
\bibliography{bibliography}

\begin{thebibliography}{10}
\expandafter\ifx\csname url\endcsname\relax
  \def\url#1{\texttt{#1}}\fi
\expandafter\ifx\csname doi\endcsname\relax
  \def\doi#1{\burlalt{doi:#1}{http://dx.doi.org/#1}}\fi
\expandafter\ifx\csname urlprefix\endcsname\relax\def\urlprefix{URL }\fi
\expandafter\ifx\csname href\endcsname\relax
  \def\href#1#2{#2}\fi
\expandafter\ifx\csname burlalt\endcsname\relax
  \def\burlalt#1#2{\href{#2}{#1}}\fi

\bibitem{ArinoMuckenhoupt90}
M.~A. Ari\~{n}o and B.~Muckenhoupt.
\newblock Maximal functions on classical {L}orentz spaces and {H}ardy's
  inequality with weights for nonincreasing functions.
\newblock {\em Trans. Amer. Math. Soc.}, 320(2):727--735, 1990.

\bibitem{Baena-MiretGogatishvili22}
S.~Baena-Miret, A.~Gogatishvili, Z.~Mihula, and L.~Pick.
\newblock Reduction principle for {G}aussian {$K$}-inequality.
\newblock {\em J. Math. Anal. Appl.}, 516(2):Paper No. 126522, 2022.

\bibitem{BasteroMilman03}
J.~Bastero, M.~Milman, and F.~J. Ruiz~Blasco.
\newblock A note on {$L(\infty,q)$} spaces and {S}obolev embeddings.
\newblock {\em Indiana Univ. Math. J.}, 52(5):1215--1230, 2003.

\bibitem{Bathory18}
M.~Bathory.
\newblock Joint weak type interpolation on {L}orentz-{K}aramata spaces.
\newblock {\em Math. Inequal. Appl.}, 21(2):385--419, 2018.

\bibitem{BennettRudnick80}
C.~Bennett and K.~Rudnick.
\newblock On {L}orentz-{Z}ygmund spaces.
\newblock {\em Dissertationes Math. (Rozprawy Mat.)}, 175:67, 1980.

\bibitem{BennettSharpley88}
C.~Bennett and R.~Sharpley.
\newblock {\em Interpolation of operators}, volume 129 of {\em Pure and Applied
  Mathematics}.
\newblock Academic Press, Inc., Boston, MA, 1988.

\bibitem{BinghamGoldie87}
N.~H. Bingham, C.~M. Goldie, and J.~L. Teugels.
\newblock {\em Regular variation}, volume~27 of {\em Encyclopedia of
  Mathematics and its Applications}.
\newblock Cambridge University Press, Cambridge, 1987.

\bibitem{Boyd69}
D.~W. Boyd.
\newblock Indices of function spaces and their relationship to interpolation.
\newblock {\em Canadian J. Math.}, 21:1245--1254, 1969.

\bibitem{CaetanoGogatishvili16}
A.~Caetano, A.~Gogatishvili, and B.~Opic.
\newblock Compactness in quasi-{B}anach function spaces and applications to
  compact embeddings of {B}esov-type spaces.
\newblock {\em Proc. Roy. Soc. Edinburgh Sect. A}, 146(5):905--927, 2016.

\bibitem{CaetanoGogatishvili11}
A.~M. Caetano, A.~Gogatishvili, and B.~Opic.
\newblock Embeddings and the growth envelope of {B}esov spaces involving only
  slowly varying smoothness.
\newblock {\em J. Approx. Theory}, 163(10):1373--1399, 2011.

\bibitem{CarroGogatishvili08}
M.~Carro, A.~Gogatishvili, J.~Mart\'{\i}n, and L.~Pick.
\newblock Weighted inequalities involving two {H}ardy operators with
  applications to embeddings of function spaces.
\newblock {\em J. Operator Theory}, 59(2):309--332, 2008.

\bibitem{CarroPick00}
M.~Carro, L.~Pick, J.~Soria, and V.~D. Stepanov.
\newblock On embeddings between classical {L}orentz spaces.
\newblock {\em Math. Inequal. Appl.}, 4(3):397--428, 2001.

\bibitem{CarroSoria96}
M.~J. Carro, A.~Garc\'{\i}a~del Amo, and J.~Soria.
\newblock Weak-type weights and normable {L}orentz spaces.
\newblock {\em Proc. Amer. Math. Soc.}, 124(3):849--857, 1996.

\bibitem{CarroRaposo07}
M.~J. Carro, J.~A. Raposo, and J.~Soria.
\newblock Recent developments in the theory of {L}orentz spaces and weighted
  inequalities.
\newblock {\em Mem. Amer. Math. Soc.}, 187(877):xii+128, 2007.

\bibitem{CarroSoria93}
M.~J. Carro and J.~Soria.
\newblock Boundedness of some integral operators.
\newblock {\em Canad. J. Math.}, 45(6):1155--1166, 1993.

\bibitem{CarroSoria97}
M.~J. Carro and J.~Soria.
\newblock The {H}ardy-{L}ittlewood maximal function and weighted {L}orentz
  spaces.
\newblock {\em J. London Math. Soc. (2)}, 55(1):146--158, 1997.

\bibitem{CianchiPick09}
A.~Cianchi and L.~Pick.
\newblock Optimal {G}aussian {S}obolev embeddings.
\newblock {\em J. Funct. Anal.}, 256(11):3588--3642, 2009.

\bibitem{CianchiPick16}
A.~Cianchi and L.~Pick.
\newblock Optimal {S}obolev trace embeddings.
\newblock {\em Trans. Amer. Math. Soc.}, 368(12):8349--8382, 2016.

\bibitem{CianchiPick15}
A.~Cianchi, L.~Pick, and L.~Slav\'{\i}kov\'{a}.
\newblock Higher-order {S}obolev embeddings and isoperimetric inequalities.
\newblock {\em Adv. Math.}, 273:568--650, 2015.

\bibitem{EdmundsEvans04}
D.~E. Edmunds and W.~D. Evans.
\newblock {\em Hardy operators, function spaces and embeddings}.
\newblock Springer Monographs in Mathematics. Springer-Verlag, Berlin, 2004.

\bibitem{EdmundsGurka95}
D.~E. Edmunds, P.~Gurka, and B.~Opic.
\newblock Double exponential integrability of convolution operators in
  generalized {L}orentz-{Z}ygmund spaces.
\newblock {\em Indiana Univ. Math. J.}, 44(1):19--43, 1995.

\bibitem{EdmundsKerman00}
D.~E. Edmunds, R.~Kerman, and L.~Pick.
\newblock Optimal {S}obolev imbeddings involving rearrangement-invariant
  quasinorms.
\newblock {\em J. Funct. Anal.}, 170(2):307--355, 2000.

\bibitem{EdmundsOpic08}
D.~E. Edmunds and B.~Opic.
\newblock Alternative characterisations of {L}orentz-{K}aramata spaces.
\newblock {\em Czechoslovak Math. J.}, 58(133)(2):517--540, 2008.

\bibitem{FernandezSignes14}
P.~Fern\'{a}ndez-Mart\'{\i}nez and T.~M. Signes.
\newblock An application of interpolation theory to renorming of
  {L}orentz-{K}aramata type spaces.
\newblock {\em Ann. Acad. Sci. Fenn. Math.}, 39(1):97--107, 2014.

\bibitem{GogatishviliKrepela17}
A.~Gogatishvili, M.~K\v{r}epela, L.~Pick, and F.~Soudsk\'{y}.
\newblock Embeddings of {L}orentz-type spaces involving weighted integral
  means.
\newblock {\em J. Funct. Anal.}, 273(9):2939--2980, 2017.

\bibitem{GogatishviliNeves10}
A.~Gogatishvili, J.~S. Neves, and B.~Opic.
\newblock Optimal embeddings of {B}essel-potential-type spaces into generalized
  {H}\"{o}lder spaces involving {$k$}-modulus of smoothness.
\newblock {\em Potential Anal.}, 32(3):201--228, 2010.

\bibitem{GogatishviliOpic04}
A.~Gogatishvili, B.~Opic, and J.~S. Neves.
\newblock Optimality of embeddings of {B}essel-potential-type spaces into
  {L}orentz-{K}aramata spaces.
\newblock {\em Proc. Roy. Soc. Edinburgh Sect. A}, 134(6):1127--1147, 2004.

\bibitem{GogatishviliOpic05}
A.~Gogatishvili, B.~Opic, and W.~Trebels.
\newblock Limiting reiteration for real interpolation with slowly varying
  functions.
\newblock {\em Math. Nachr.}, 278(1-2):86--107, 2005.

\bibitem{GogatishviliPick03}
A.~Gogatishvili and L.~Pick.
\newblock Discretization and anti-discretization of rearrangement-invariant
  norms.
\newblock {\em Publ. Mat.}, 47(2):311--358, 2003.

\bibitem{GogatishviliPick06}
A.~Gogatishvili and L.~Pick.
\newblock Embeddings and duality theorems for weak classical {L}orentz spaces.
\newblock {\em Canad. Math. Bull.}, 49(1):82--95, 2006.

\bibitem{GogatishviliSoudsky14}
A.~Gogatishvili and F.~Soudsk\'{y}.
\newblock Normability of {L}orentz spaces---an alternative approach.
\newblock {\em Czechoslovak Math. J.}, 64(139)(3):581--597, 2014.

\bibitem{GurkaOpic07}
P.~Gurka and B.~Opic.
\newblock Sharp embeddings of {B}esov-type spaces.
\newblock {\em J. Comput. Appl. Math.}, 208(1):235--269, 2007.

\bibitem{HillePhillips57}
E.~Hille and R.~S. Phillips.
\newblock {\em Functional analysis and semi-groups}.
\newblock American Mathematical Society Colloquium Publications, Vol. 31.
  American Mathematical Society, Providence, R.I., 1957.
\newblock rev. ed.

\bibitem{Ho20}
K.-P. Ho.
\newblock Sublinear operators on radial rearrangement-invariant quasi-{B}anach
  function spaces.
\newblock {\em Acta Math. Hungar.}, 160(1):88--100, 2020.

\bibitem{JohnsonMaurey79}
W.~B. Johnson, B.~Maurey, G.~Schechtman, and L.~Tzafriri.
\newblock Symmetric structures in {B}anach spaces.
\newblock {\em Mem. Amer. Math. Soc.}, 19(217):v+298, 1979.

\bibitem{Karamata30}
J.~Karamata.
\newblock Sur un mode de croissance re{\^a}guilie{\'a}re des fonctions.
\newblock {\em Mathematica (Cluj)}, 4:38--53, 1930.

\bibitem{Karamata33}
J.~Karamata.
\newblock Sur un mode de croissance r\'{e}guli\`ere. {T}h\'{e}or\`emes
  fondamentaux.
\newblock {\em Bull. Soc. Math. France}, 61:55--62, 1933.

\bibitem{KreinPetunin82}
S.~G. Kre\u{\i}n, Y.~I. Petun\={\i}n, and E.~M. Sem\"{e}nov.
\newblock {\em Interpolation of linear operators}, volume~54 of {\em
  Translations of Mathematical Monographs}.
\newblock American Mathematical Society, Providence, R.I., 1982.
\newblock Translated from the Russian by J. Sz\H{u}cs.

\bibitem{Lai93}
S.~Lai.
\newblock Weighted norm inequalities for general operators on monotone
  functions.
\newblock {\em Trans. Amer. Math. Soc.}, 340(2):811--836, 1993.

\bibitem{Lorentz51}
G.~G. Lorentz.
\newblock On the theory of spaces {$\Lambda$}.
\newblock {\em Pacific J. Math.}, 1:411--429, 1951.

\bibitem{Lorentz61}
G.~G. Lorentz.
\newblock Relations between function spaces.
\newblock {\em Proc. Amer. Math. Soc.}, 12:127--132, 1961.

\bibitem{Maligranda85}
L.~Maligranda.
\newblock Indices and interpolation.
\newblock {\em Dissertationes Math. (Rozprawy Mat.)}, 234:49, 1985.

\bibitem{Maz'ya11}
V.~Maz'ya.
\newblock {\em {S}obolev Spaces}.
\newblock Number 242 in {G}rundlehren der mathematischen {W}issenschaften.
  {S}pringer, 2 edition, 2011.

\bibitem{MilmanPustylnik04}
M.~Milman and E.~Pustylnik.
\newblock On sharp higher order {S}obolev embeddings.
\newblock {\em Commun. Contemp. Math.}, 6(3):495--511, 2004.

\bibitem{NekvindaPesa20}
A.~Nekvinda and D.~Peša.
\newblock On the properties of quasi-{B}anach function spaces.
\newblock 2020, \burlalt{arXiv:2004.09435 }{http://arxiv.org/abs/2004.09435}.

\bibitem{Neves02}
J.~S. Neves.
\newblock Lorentz-{K}aramata spaces, {B}essel and {R}iesz potentials and
  embeddings.
\newblock {\em Dissertationes Math. (Rozprawy Mat.)}, 405:46, 2002.

\bibitem{NevesOpic20}
J.~S. Neves and B.~Opic.
\newblock Optimal local embeddings of {B}esov spaces involving only slowly
  varying smoothness.
\newblock {\em J. Approx. Theory}, 254:105393, 25, 2020.

\bibitem{OpicGroverTBD}
B.~Opic and M.~Grover.
\newblock Description of {$K$}-spaces by means of {$J$}-spaces and the reverse
  problem in the limiting real interpolation.
\newblock {\em Math. Nachr.}, To appear.

\bibitem{OpicPick99}
B.~Opic and L.~Pick.
\newblock On generalized {L}orentz-{Z}ygmund spaces.
\newblock {\em Math. Inequal. Appl.}, 2(3):391--467, 1999.

\bibitem{Pesa22}
D.~Pe\v{s}a.
\newblock Wiener-{L}uxemburg amalgam spaces.
\newblock {\em J. Funct. Anal.}, 282(1):Paper No. 109270, 47, 2022.

\bibitem{PesaSV}
D.~Peša.
\newblock On the smoothness of slowly varying functions.
\newblock 2023, \burlalt{arXiv:2304.14148 }{http://arxiv.org/abs/2304.14148}.

\bibitem{FucikKufner13}
L.~Pick, A.~Kufner, O.~John, and S.~Fu{\v c}{\'i}k.
\newblock {\em Function Spaces, Vol. 1}.
\newblock Number~14 in {D}e {G}ruyter Series in Nonlinear Analysis and
  Applications. {W}alter de {G}ruyter, 2 edition, 2013.

\bibitem{Lopez08}
S.~Rodríguez~L{\' o}pez.
\newblock {\em Transference theory between quasi-{B}anach function spaces with
  applications to the restriction of {F}ourier multipliers}.
\newblock Thesis, Universitat de Barcelona, 2008.

\bibitem{Sawyer90}
E.~Sawyer.
\newblock Boundedness of classical operators on classical {L}orentz spaces.
\newblock {\em Studia Math.}, 96(2):145--158, 1990.

\bibitem{Sinnamon02}
G.~Sinnamon.
\newblock Embeddings of concave functions and duals of {L}orentz spaces.
\newblock {\em Publ. Mat.}, 46(2):489--515, 2002.

\bibitem{SlavikovaUnp}
L.~Slav\'{\i}kov\'{a}.
\newblock Unpublished manuscript.

\bibitem{Soudsky16}
F.~Soudsk\'{y}.
\newblock Note on linearity of rearrangement-invariant spaces.
\newblock {\em Ann. Funct. Anal.}, 7(2):232--239, 2016.

\bibitem{Stepanov93}
V.~D. Stepanov.
\newblock The weighted {H}ardy's inequality for nonincreasing functions.
\newblock {\em Trans. Amer. Math. Soc.}, 338(1):173--186, 1993.

\bibitem{Turcinova19}
H.~Tur{\v c}inov{\' a}.
\newblock Basic functional properties of certain scale of
  rearrangement-invariant spaces.
\newblock 2019, \burlalt{arXiv:2009.05351 }{http://arxiv.org/abs/2009.05351}.

\bibitem{Zygmund35_3}
A.~Zygmund.
\newblock {\em Trigonometric series. {V}ol. {I}, {II}}.
\newblock Cambridge Mathematical Library. Cambridge University Press,
  Cambridge, third edition, 2002.
\newblock With a foreword by Robert A. Fefferman.

\end{thebibliography}

\end{document}